\documentclass[a4paper, oneside, 10pt]{amsart}
\usepackage[utf8]{inputenc}

\title{$\td$-tilting theory for linear Nakayama algebras}

\usepackage[foot]{amsaddr}

\author[E. S. Rundsveen]{Endre S. Rundsveen\, \orcidlink{0009-0000-0589-8399}{$^\ast$}}
\address{Department of mathematical sciences, NTNU, Trondheim, Norway \\ \orcidlink{0009-0000-0589-8399}\, 0009-0000-0589-8399}
\email{endre.s.rundsveen@ntnu.no}
\author[L. Vaso]{Laertis Vaso\, \orcidlink{0000-0003-0313-7777}}
\address{Department of mathematical sciences, NTNU, Trondheim, Norway \\  \orcidlink{0000-0003-0313-7777}\, 0000-0003-0313-7777}
\email{laertis.vaso@ntnu.no}
\address{$^{\ast}$Corresponding author}
\subjclass{16G20,16S90,18E10,18G99}

\DeclareRobustCommand{\SkipTocEntry}[5]{} 

\usepackage[table]{xcolor}
\usepackage{diagbox}
\usepackage{multirow}
\usepackage{array}
\usepackage{amsmath,amssymb,amsthm,mathtools,stmaryrd} 
\usepackage{tikz}
\usetikzlibrary{calc,positioning,arrows}
\usepackage{todonotes}

\usepackage{imakeidx}
\usepackage{hyperref}
\usepackage{orcidlink}

\usepackage[shortlabels]{enumitem}

\usepackage{tikz-cd}
\usetikzlibrary{decorations.pathreplacing}
\usetikzlibrary{decorations.pathmorphing} 
\usepackage{quiver}
\usepackage{cleveref}
\usepackage{comment}
\usepackage{adjustbox}

\usepackage{drawTaud}

\newsavebox{\wideeqbox}
\newenvironment{wideeq}
  {\begin{displaymath}\begin{lrbox}{\wideeqbox}$\displaystyle}
  {$\end{lrbox}\makebox[0pt]{\usebox{\wideeqbox}}\end{displaymath}}

\DeclareMathOperator{\m}{mod}

\newcommand{\syzygy}{\Omega}
\newcommand{\cosyzygy}{\Omega^{-}}
\newcommand{\rad}[1]{\mathrm{Rad}(#1)}
\newcommand{\ALinOr}[1]{\overset{\rightarrow}{{\mathbb{A}}_{#1}}}
\newcommand{\K}{\mathbf{k}}
\newcommand{\C}{\mathcal{C}}
\newcommand{\add}[1]{\mathrm{add}\left( #1\right)}
\newcommand{\modfin}[1]{\m#1}
\newcommand{\ind}[2]{M(#1,#2)}
\newcommand{\Hom}[3]{\mathrm{Hom}_{#1}\left(#2,#3\right)}
\newcommand{\diag}[1]{\mathcal{D}_{#1}}
\newcommand{\diagm}[1]{D_{#1}}
\newcommand{\updiag}[2]{\diag{#1}^\uparrow(#2)}
\newcommand{\updiagm}[2]{\diagm{#1}^\uparrow(#2)}
\newcommand{\downdiag}[2]{\diag{#1}^\downarrow(#2)}
\newcommand{\downdiagm}[2]{\diagm{#1}^\downarrow(#2)}
\newcommand{\length}[1]{\ell(#1)}

\newcommand{\Socle}[1]{\mathrm{Soc}(#1)}
\newcommand{\coker}{coker}
\newcommand{\im}{\mathrm{im}}
\newcommand{\td}{\tau_d^{\phantom{-}}}
\newcommand{\tdo}{\tau_d^{-}}
\newcommand{\gldim}{\mathrm{gl.dim.}}

\newcommand{\projinj}{\mathcal{PI}}

\newcommand{\fac}[1]{\mathrm{Fac}(#1)}
\newcommand{\proj}[1]{\mathrm{proj}\,#1}

\newcommand{\HomotopyC}[2]{\mathrm{K}^{#2}(#1)}
\newcommand{\thick}[1]{\mathrm{thick}\left( #1\right)}
\newcommand{\silt}[1]{\mathrm{silt}\,#1}
\newcommand{\trunc}[2]{\sigma_{#2}#1}
\newcommand{\shift}[2]{#1[#2]}
\newcommand{\stautilt}[1]{\mathrm{s}\tau\operatorname{-tilt}(#1)}
\newcommand{\ftors}[1]{\operatorname{f-tors}(#1)}
\newcommand{\dtors}[1]{\operatorname{d-tors}(#1)}
\newcommand{\dsilt}[1]{(d+1)\operatorname{-silt}(#1)}
\newcommand{\fdtors}[1]{\operatorname{f-d-tors}(#1)}

\newcommand{\twosilt}[1]{2\operatorname{-silt}(#1)}

\newcommand{\cC}{\mathcal{C}}
\newcommand{\cP}{\mathcal{P}}
\newcommand{\cM}{\mathcal{M}}

\newcommand{\cU}{\mathcal{U}}
\newcommand{\cA}{\mathcal{A}}
\newcommand{\isom}{\cong}
\DeclarePairedDelimiter\abs{\lvert}{\rvert}%

\newcommand{\Ext}{\mathrm{Ext}}
\newcommand{\blue}{B}
\newcommand{\red}{R}
\newcommand{\notblue}[1]{B_{#1}^{\times}}
\newcommand{\notred}[1]{R_{#1}^{\times}}

\newtheorem{theorem}{Theorem}[section] 
\newtheorem{theoremIntro}{Theorem}

\newtheorem{corollary}[theorem]{Corollary}
\newtheorem{lemma}[theorem]{Lemma}
\newtheorem{proposition}[theorem]{Proposition}
\theoremstyle{definition}
\newtheorem{definition}[theorem]{Definition}
\newtheorem{remark}[theorem]{Remark}
\newtheorem{example}[theorem]{Example}

\newtheorem{conditions}[theorem]{Conditions}

\numberwithin{equation}{section}

\makeindex[name=definitions, columns=2, title=List of definitions]
\makeindex[name=symbols, columns=2, title=List of symbols, options={-s index_style.ist}]
\indexsetup{level=\section*} 

\begin{document}

\begin{abstract}
    Support $\tau$-tilting pairs, functorially finite torsion classes and $2$-term silting complexes are three much studied concepts in the representation theory of finite-dimensional algebras, which moreover turn out to be connected via work of Adachi, Iyama and Reiten. We investigate their higher-dimensional analogues via $\td$-rigid pairs, $d$-torsion classes and $(d+1)$-term silting complexes as well as the connections between these three concepts. Our work is done in the setting of truncated linear Nakayama algebras $\Lambda(n,l)=\K \mathbb{A}_{n}/\rad{\K \mathbb{A}_{n}}^l$ admitting a $d$-cluster tilting module. More specifically, we classify $\td$-rigid pairs $(M,P)$ of $\Lambda(n,l)$ with $\abs{M}+\abs{P}=n$ via an explicit combinatorial description and show that they can be characterized by a certain maximality condition as well as by giving rise to a $(d+1)$-term silting complex in $\HomotopyC{\proj{\Lambda(n,l)}}{b}$. We also describe all $d$-torsion classes of $\Lambda(n,l)$. Finally, we compare our results to the classical case $d=1$ and investigate mutation with a special emphasis on the case where $d=\gldim(\Lambda)$. 
\end{abstract}
\maketitle 
\tableofcontents

\setcounter{section}{0}

\section{Introduction}
\subsection{Historical background and motivation} Tilting theory came into being through a series of papers in the mid to late 1900's \cite{gelfand1972problems,bernstein1973coxeter,auslander1979coxeter,brenner1980generalizations,happel1982tilted}; for details see the survey \cite{hugel2007handbook}. It soon found itself safely in the toolboxes of both representation theorists in particular and algebraists in general. Concurrently and intertwined with the birth of tilting theory, Auslander--Reiten theory was also developed \cite{auslander1974representation,auslander1975representation,auslander1977representation}, providing yet another tool in the study of finite-dimensional algebras. This article lies in the intersection of certain generalizations of these two theories that we now describe.

\subsubsection*{Tilting theory.} Recall that a \emph{partial tilting}\index[definitions]{tilting module} module $T$ is a rigid module with projective dimension 1 or less, and if in addition $\abs{T}=\abs{\Lambda}$, then $T$ is called a \emph{tilting module}\index[definitions]{tilting module!partial}. Bongartz proved that any partial tilting module can be completed to a tilting module \cite{BongCompl}. Later it was shown that for an \emph{almost complete}\index[definitions]{tilting module!almost complete} tilting module $T$ (i.e. a partial tilting module with $\abs{T}=\abs{\Lambda}-1$), there are at most two ways of completing it \cite{riedtmann1991simplicial,UNGER1990205}. The interest in the number of completions arose from the fact that two completions of the same almost complete tilting module relate to each other through a natural short exact sequence.  

If $\Lambda$ is hereditary, then there are exactly two ways to complete an almost complete tilting module $T$ if and only if $T$ is sincere \cite{happel1989almost}. To remove the assumption on $T$ being sincere, the notion of \emph{support tilting} modules was introduced \cite{Ringel_2007,ingalls2009noncrossing}. A support tilting module\index[definitions]{tilting module!support} $T$ is a partial tilting module which is tilting over $\modfin{\Lambda/\langle e\rangle}$ for some idempotent $e\in\Lambda$. Over hereditary algebras this was a sufficient generalization, i.e. almost complete support tilting modules of hereditary algebras can be completed in two different ways. 

\subsubsection*{\texorpdfstring{$\tau$}{t}-tilting theory.} In 2014, Adachi, Iyama and Reiten \cite{adachi_-tilting_2014} proposed a broader generalization of tilting modules to remove the assumption on the algebra being hereditary. They introduced the notion of \emph{$\tau$-rigid}\index[definitions]{$\tau$-rigid module} pairs defined as a module $M\in\modfin\Lambda$ and a projective module $P\in\proj\Lambda$ such that $\Hom{\Lambda}{M}{\tau M}=0$ and $\Hom{\Lambda}{P}{M}=0$ where $\tau$ is the Auslander--Reiten translation. They next defined a \emph{$\tau$-tilting}\index[definitions]{$\tau$-tilting module} module $M$ to be a $\tau$-rigid module such that $\abs{M}=\abs{\Lambda}$, and a \emph{support $\tau$-tilting pair}\index[definitions]{support $\tau$-tilting pair} $(M,P)$ to be a $\tau$-rigid pair such that $M$ is $\tau$-tilting in $\modfin{\Lambda/\langle e\rangle}$ where $P=e\Lambda$. This generalization was shown to be broad enough for the initial motivation, as well as preserving a close relationship with functorially finite torsion classes which tilting modules had already enjoyed.

\subsubsection*{Auslander--Reiten theory.} Auslander--Reiten theory, developed by Maurice Auslander and Idun Reiten in the 1970s, is central in the representation theory of finite-dimensional algebras. The theory focuses on the study of almost split sequences and Auslander--Reiten quivers, which encode indecomposable modules and irreducible morphisms between them. An important role is played by the \emph{Auslander--Reiten} translations $\tau$ and $\tau^{-}$ which act as bijections between the sets of indecomposable non-injective and indecomposable non-projective modules over an algebra. 

\subsubsection*{Higher-dimensional homological algebra.} The generalization of Auslander--Reiten theory takes place in the world of higher-dimensional homological algebra introduced by Iyama \cite{Iyama200722,iyama2008auslander,iyama2011cluster}. A central role in this theory is played by certain subcategories $\cC$ of $\modfin{\Lambda}$ called \emph{$d$-cluster tilting}. In such a subcategory one finds a $d$-dimensional version of Auslander--Reiten theory. This leads to a natural generalization of $\tau$-rigidity in $\cC$ given by replacing $\tau$ with the higher Auslander--Reiten translation $\td$. 

\subsubsection*{Motivation.} Many notions from classical representation theory and homological algebra have been lifted and studied in the setting of higher-dimensional homological algebra; a relevant example for us is the notion of $d$-torsion classes introduced by Jørgensen \cite{Jorgensen2014,asadollahi2022higher,august2023characterisation}. In such generalizations usually the role of the module category $\modfin{\Lambda}$ is played by a $d$-cluster tilting subcategory $\cC\subseteq\modfin{\Lambda}$. However, this approach is not fruitful if one tries to generalize support $\tau$-tilting modules in this setting. Indeed, if $\Lambda$ admits a $d$-cluster tilting subcategory $\cC$ and $e\in\Lambda$ is an idempotent, then, in general, $\Lambda/\langle e\rangle$ does not admit a $d$-cluster tilting subcategory.

Nevertheless, inspired by the maximality properties of $\tau$-tilting modules, several authors have suggested possible generalizations of support $\tau$-tilting modules in higher-dimensional homological algebra \cite{JACOBSEN2020119,mcmahon2021support,ZHOU2023193,MARTINEZ202398}. In a similar spirit, there is work by several authors studying how the relationship between support $\tau$-tilting pairs, functorially finite torsion classes and $2$-silting complexes generalizes in the higher setting \cite{august2023characterisation,MARTINEZ202398,august2024taudtilting}.

The above facts can be seen as a motivation for this paper. How do the different types of higher $\tau$-tilting theories found in the literature look like? Is it possible to infer a type of maximality for pairs which have the same amount of indecomposable summands as $\Lambda$, as in the classical case? Inspired by Adachi's classification of $\tau$-tilting modules of Nakayama algebras \cite{ADACHI2016227}, we chose to investigate these questions over Nakayama algebras.

\subsection{This paper} As mentioned we chose the realm of Nakayama algebras as base camp for our reconnaissance. However, when exploring new land one tries to tackle bite sized hurdles. We focus on acyclic Nakayama algebras bounded by homogeneous relations for two reasons. Firstly, there exists an explicit classification of Nakayama algebras bounded by homogeneous relations which admit $d$-cluster tilting subcategories \cite{VASO20192101, darpo2020representationfinite} while such a classification is lacking for general Nakayama algebras. Second, acyclic ones were chosen since in this case the $d$-cluster tilting subcategory is unique, in contrast to the cyclic case. The algebras we are working with are therefore bounded quiver algebras of the form 
$\K \ALinOr{n}/R^l$ where $\ALinOr{n}$ is the quiver given by
\[
\begin{tikzcd}
	\ALinOr{n}\colon\qquad n\rar{}&n-1\rar{}&\cdots\rar{}&3\rar{}&2\rar{}&1.
\end{tikzcd}
\]
and $R$ is the arrow ideal of $\K\ALinOr{n}$. We denote $\Lambda(n,l)\coloneqq \K \ALinOr{n}/R^l$. The triples $(n,l,d)$ for which $\Lambda(n,l)$ admits a $d$-cluster tilting subcategory $\cC$ can be found in Theorem \ref{thm:VasoClassifyAcyclicCluster}, and the description of $\cC$ is given in Section \ref{subsec:d-cluster tilting subcategories for Lambda(n,l)}. 

\subsubsection*{Summand-maximal $\td$-rigid pairs} We say that a $\td$-rigid pair $(M,P)$ is \emph{summand-maximal} if for any other $\td$-rigid pair $(N,Q)$ we have $\abs{N}+\abs{Q}\leq \abs{M}+\abs{P}$. This is the main maximality property of $\td$-rigid pairs which we study. Over an algebra $\Lambda(n,l)$ which admits a $d$-cluster tilting subcategory, we can describe all summand-maximal $\td$-rigid pairs in an explicit combinatorial way encoded in the notion of \emph{well-configured pairs}, see Definition \ref{def:well-configured}. This allows us to prove our first main result:

\begin{theoremIntro}[Theorem \ref{thrm:taud tilting is well-configured}]\label{Theorem A}
Assume $\Lambda=\Lambda(n,l)$ admits a $d$-cluster tilting subcategory. If $M\in\cC$ and $P\in\proj\Lambda$, then the following are equivalent
\begin{enumerate}
    \item[(a)] $(M,P)$ is a summand-maximal $\td$-rigid pair,
    \item[(b)] $(M,P)$ is a $\td$-rigid pair and $\abs{M}+\abs{P}=\abs{\Lambda}$, and
    \item[(c)] $(M,P)$ is well-configured (Definition \ref{def:well-configured}).
\end{enumerate}
\end{theoremIntro}

Notice that part (b) of Theorem \ref{Theorem A} provides an easy way of checking if a $\td$-rigid pair is summand-maximal, while part (c) provides an easy way of constructing summand-maximal $\td$-rigid pairs.

\subsubsection*{$d$-torsion classes} Next we focus on the study of $d$-torsion classes of $\Lambda(n,l)$ when there exists a $d$-cluster tilting subcategory $\cC\subseteq\modfin{\Lambda}$. Once again we provide an explicit combinatorial method of constructing $d$-torsion classes in this case through paths in the graphs $G_1$ (eq. \ref{eq:multigraph1 for d-torsion}) when $l>2$ and $G_2$ (eq. \ref{eq:multigraph2 for d-torsion}) when $l=2$. It turns out that our method constructs all $d$-torsion classes which leads to the following result.

\begin{theoremIntro}[Theorem \ref{thrm:classification of d-torsion classes} and Theorem \ref{thrm:classification of d-torsion classes in l=2}]\label{Theorem B}
\begin{enumerate}
    \item[(a)] If $l>2$, then there exists a bijection between the set of paths $\chi$ in $G_1$ of length $p-1$ starting at an odd vertex and the set of $d$-torsion classes $\cU$ in $\cC$.
    \item[(b)] If $l=2$, then there exists a bijection between the set of paths $\chi$ in $G_2$ of length $p-1$ and the set of $d$-torsion classes $\cU$ in $\cC$.
\end{enumerate}
\end{theoremIntro}

The case $l=2$ is in fact a very degenerative version of the case $l>2$, but we distinguish between the two cases for technical reasons. The above bijection, the details of which are given in Sections \ref{subsubsec:the case l>2} and \ref{subsubsec:the case l=2}, gives an easily implemented algorithm to construct $d$-torsion classes, and a relatively more cumbersome but still manageable way to check that a given subcategory of $\cC$ is a $d$-torsion class.

\subsubsection*{$(d+1)$-term silting complexes} Our last main result gives a characterization of summand-maximal $\td$-rigid pairs in terms of $(d+1)$-term silting complexes. 

\begin{theoremIntro}[Theorem \ref{Theorem:strongly maximal and silting}] \label{Theorem C}
    Assume $\Lambda=\Lambda(n,l)$ admits a $d$-cluster tilting subcategory. Let $M\in\cC$ and $P\in\proj\Lambda$.
    Then $(M,P)$ is a summand-maximal $\td$-rigid pair if and only if $\mathbf{P}^\bullet_{(M,P)}\coloneqq \shift{P}{d}\oplus\trunc{\mathbf{P}^\bullet(M)}{\geq -d}$ is a silting complex in $\HomotopyC{\proj{\Lambda}}{b}$.
\end{theoremIntro}

\subsubsection*{Comparing the classical and the higher case} At the very end of the paper we take a step back and look at how properties of support $\tau$-tilting modules can be traced to different generalizations. In addition we compare the suggested frameworks in \cite{JACOBSEN2020119}, \cite{ZHOU2023193}, \cite{MARTINEZ202398} and our paper. These observations and comparisons can be roughly summarized in the following diagram

\begin{wideeq}
\begin{tikzpicture}
    \draw[rounded corners=2pt,thick,gray!50] (-3.5,.5) rectangle (9.5,-.45); 
    \node at (-3.5,0) [anchor=west] {Classical};
    \draw[rounded corners=2pt,thick,gray!50] (-3.5,-.55) rectangle (9.5,-5);
    \node at (-3.5,-2.75) [anchor=west] {Generalizations};
    
    \node (ftors) at (0,0) [] {$\ftors{\Lambda}$};
    \node (sttilt) at (4,0) [] {$\stautilt{\Lambda}$};
    \node (2silt) at (8,0) [] {$\twosilt{\Lambda}$};
    \node (dsilt) at (8,-4.5) [] {$\dsilt{\Lambda}$};
    \node (fdtors) at (0,-4.5) [] {$\fdtors{\Lambda}$};
    \node (mtdrigid) at (2,-1.5) [scale=.8] {$\left\{\begin{array}{c}
         \mathrm{Maximal}  \\
          \td\text{-\,rigid\ pairs}
    \end{array}\right\}$};
    \node (stdtilt) at (6,-1.5) [scale=.8] {$\left\{\begin{array}{c}
         \mathrm{Support}  \\
          \td\text{-\,tilting\ pairs}
    \end{array}\right\}$};
    \node (mtdrigidWithSum) at (2,-3) [scale=.8] {$\left\{\begin{array}{c}
         \mathrm{Maximal}  \\
          \td\text{-\,rigid\ pairs}\\
          \mathrm{s.t.}\ \abs{M}+\abs{P}=\abs{\Lambda}
    \end{array}\right\}$};
    \node (stdrigid) at (6,-3) [scale=.8] {$\left\{\begin{array}{c}
         \mathrm{Summand\ maximal}  \\
          \td\text{-\,rigid\ pairs}
    \end{array}\right\}$};

    \draw[decorate,decoration={zigzag,segment length=1mm,amplitude=1pt,pre=lineto,pre length=4pt,post=lineto,post length=4pt},->] (ftors)--(fdtors)node[midway,left,scale=.8]{\cite{Jorgensen2014}};
    \draw[decorate,decoration={zigzag,segment length=1mm,amplitude=1pt,pre=lineto,pre length=4pt,post=lineto,post length=4pt},->] (2silt)--(dsilt)node[midway,right,scale=.8]{\cite{MARTINEZ202398}};
    \draw[decorate,decoration={zigzag,segment length=1mm,amplitude=1pt,pre=lineto,pre length=4pt,post=lineto,post length=4pt},->] (sttilt)--(mtdrigid)node[midway,anchor=west,yshift=-3pt,scale=.8]{\cite{JACOBSEN2020119}};
    \draw[decorate,decoration={zigzag,segment length=1mm,amplitude=1pt,pre=lineto,pre length=4pt,post=lineto,post length=4pt},->] (sttilt)--(stdtilt)node[midway,anchor=east,yshift=-3pt,scale=.8]{\cite{ZHOU2023193}};
    
    \draw[<->] (sttilt)--(2silt)node[midway,above,scale=.8]{$1-1$}node[midway,below,scale=.8]{\cite{adachi_-tilting_2014}};
    \draw[<->] (sttilt)--(ftors)node[midway,above,scale=.8]{$1-1$}node[midway,below,scale=.8]{\cite{adachi_-tilting_2014}};
    \draw[->] (fdtors)--(dsilt)node[midway,above,scale=.8]{\cite{august2024taudtilting}};
    \draw[right hook->] ([xshift=10pt]fdtors.north)--([yshift=-4pt]mtdrigidWithSum.south)node[midway,below,scale=.8]{\cite{august2024taudtilting}};
    \draw[right hook->] (mtdrigidWithSum)--(mtdrigid);
    \draw[right hook->] (stdrigid)--(stdtilt);
    \draw[double equal sign distance] (stdrigid)--(mtdrigidWithSum)node[midway,below,scale=.8]{$\ast$}node[midway,yshift=-10pt,scale=.8]{Thm. \ref{Theorem A}};
    \draw[right hook->] ([yshift=-.2cm]stdrigid.south)--(dsilt)node[midway,anchor=east,xshift=-5pt]{$\ast$}node[midway,xshift=-33pt,scale=.8]{Thm. \ref{Theorem C}};
    \draw[left hook->] ([yshift=-4pt]stdtilt.west)--([yshift=-4pt]mtdrigid.east)node[midway, scale=.8,below]{$\ast$}node[midway,yshift=-12pt,scale=.8]{Cor. \ref{lem:support tau_d tilting and strongly maximal is the same if no homs}};
    \draw[right hook->] ([yshift=4pt]mtdrigid.east)--([yshift=4pt]stdtilt.west);
\end{tikzpicture}
\end{wideeq}

\noindent where the zigzag arrows trace different generalizations and the arrows marked with an asterix hold for $\Lambda=\Lambda(n,l)$ admitting a $d$-cluster tilting subcategory. We also present some observations regarding mutation of summand-maximal $\td$-rigid pairs and the lattice of $d$-torsion classes and investigate in more details the special case where $d=\gldim(\Lambda(n,l))$.

Another direction of generalization that we haven't considered in this paper is given through \cite{gupta2024dtermsiltingobjectstorsion} and \cite{zhou2025tiltingextendedmodule}. Rather than working in a $d$-cluster tilting subcategory of $\modfin{\Lambda}$, this generalization takes place in the \emph{\(m\)-extended module category} \(m~-~\modfin{\Lambda}\), which is a subcategory of  \(\mathrm{D}^b(\modfin{\Lambda})\) given by complexes with non-zero homology only in degrees \(i\) with \(i\in [-(m-1),0]\).

\subsection{Structure} The paper consists of six sections. Section 2 sets the stage for the rest of the text by introducing some fundamental notation, recalling a few results and proving some preliminary observations. In Section 3 we further build our toolbox and prove Theorem \ref{Theorem A}. In Section 4 we recall some results and notions about $d$-torsion classes before proving Theorem \ref{Theorem B}. In Section 5 we recall results about $(d+1)$-silting complexes before proving Theorem \ref{Theorem C}. Sections 4 and 5 are independent of each other. Section 6 depends on all preceding sections and acts both as a summary and as a collection of observations. In there we present an overview of our results and explore how they do or do not generalize relevant results from the case $d=1$. Furthermore we present examples and results regarding mutation of summand-maximal $\td$-rigid pairs with special emphasis put on the case $d=\gldim(\Lambda(n,l))$. In Section \ref{sec:appendix} we include an illustration that showcases the most widely used notation in this article. We also include an index of definitions and symbols used in this paper. 

We have made a website that we hope can serve as a supplemental tool for the reader during their exploration of the paper. The site can be found at \url{https://endresr.github.io/Higher_Tau_Nakayama/}.

\addtocontents{toc}{\SkipTocEntry} 
\section*{Acknowledgements}
The authors gratefully acknowledge the support and resources provided by the Centre for Advanced Studies at the Norwegian Academy of Science, and thank them for their kind hospitality during the preparation of part of this paper. 

The authors thank Steffen Oppermann for his support and feedback throughout the project. They would also like to thank Jacob Fjeld Grevstad and Erlend Due B\o rve for useful comments and discussions. Finally, the authors express their gratitude to the authors of \cite{august2024taudtilting} for sharing a preliminary version of their manuscript with us.

\section{Preliminaries}
In this section we introduce some notation and recall all the background facts that we need.

\subsection{Notation and conventions} We start with some basic conventions that hold throughout the whole article.

All subcategories in this article will be assumed to be full, additive and closed under isomorphisms, unless otherwise specified. 

For a set $A$ we denote by $\abs{A}$ the cardinality of $A$. If $B$ is another set, then we denote by $A\sqcup B$\index[symbols]{((a@$\sqcup$} the \emph{disjoint union} of $A$ and $B$, that is when we write $X=A\sqcup B$, we mean that $X=A \cup B$ and $A\cap B=\varnothing$. For integers $a$ and $b$ we denote by $[a,b]$ the \emph{interval} of all integers between $a$ and $b$, that is $[a,b] = \{z\in \mathbb{Z} \mid a\leq z\leq b\}$. The \emph{length} of $[a,b]$ is defined to be its cardinality.

Throughout this article we fix a field $\K$\index[symbols]{Ka@$\K$}. We denote by $n$, $l$ and $d$ three positive integers with $d\geq 2$, unless stated otherwise. By an algebra we mean a finite-dimensional associative unital algebra over $\K$. If $\Lambda$ is an algebra, we denote by $\modfin{\Lambda}$ the category of finitely-generated right $\Lambda$-modules. We denote by $\underline{\mathrm{mod}}({\Lambda})$\index[symbols]{Mb@$\underline{\mathrm{mod}}({\Lambda})$} respectively $\overline{\mathrm{mod}}(\Lambda)$\index[symbols]{Ma@ $\overline{\mathrm{mod}}(\Lambda)$}, the \emph{projectively stable category} respectively \emph{injectively stable category}. We denote by $\tau$ and $\tau^{-}$\index[symbols]{T@$\tau(M)$}\index[symbols]{T@$\tau^{-}(M)$} the \emph{Auslander--Reiten translations} of $\Lambda$ and we denote by $\gldim(\Lambda)$ the \emph{global dimension}\index[symbols]{G@$\gldim(\Lambda)$} of $\Lambda$. 

For a module $M\in\modfin{\Lambda}$ we denote by $\syzygy(M)$ the \emph{syzygy} of $M$,
that is the kernel of a projective cover of $M$, and by $\cosyzygy(M)$\index[symbols]{O@$\Omega(M)$}\index[symbols]{O@$\Omega^{-}(M)$} the \emph{cosyzygy} of $M$, that is the cokernel of an injective envelope of $M$. We denote by $\abs{M}$\index[symbols]{((b@$\abs{M}$} the number of nonisomorphic indecomposable direct summands of $M$; we usually use the same letter as we used for the module in small font for that, so for example $m=\abs{M}$. We denote by $\length{M}$\index[symbols]{L@$\length{M}$} the \emph{Loewy length} of $M$. We write $\add{M}$\index[symbols]{A@$\add{M}$} for the \emph{additive closure} of $M$ in $\modfin{\Lambda}$, that is the smallest subcategory of $\modfin{\Lambda}$ closed under taking direct sums of direct summands of $M$; we usually use the same letter as we used for the module in script font for that, so for example $\cM=\add{M}$. 

A \emph{quiver} $Q=(Q_0,Q_1,s,t)$ consists of a set $Q_0$ of \emph{vertices}, a set $Q_1$ of \emph{arrows} and \emph{source} and \emph{target} maps $s,t:Q_1\to Q_0$. All quivers in this article are \emph{finite}, that is the sets $Q_0$ and $Q_1$ are finite sets. 

If $\Lambda=\K Q/I$ is a bound-quiver algebra and $i\in Q_0$ is a vertex, then we denote by $S(i)$\index[symbols]{Sa@$S(i)$}\index[symbols]{Pa@$P(i)$}\index[symbols]{Ia@$I(i)$} the simple module corresponding to the vertex $i$. We denote by $P(i)$ the indecomposable projective $\Lambda$-module with top $S(i)$ and by $I(i)$ the indecomposable injective $\Lambda$-module with socle $S(i)$. For a subset $J\subseteq Q_0$ we denote the direct sums
$$
S(J)\coloneqq \bigoplus_{i\in J}S(i), \;\; P(J)\coloneqq \bigoplus_{i\in J}P(i), \;\;
I(J)\coloneqq \bigoplus_{i\in J}I(i).
$$\index[symbols]{Sb@$S(J)$}\index[symbols]{Pb@$P(J)$}\index[symbols]{Ib@$I(J)$}

For more details on the representation theory of finite-dimensional algebras, on Auslander--Reiten theory and on representations of quivers with relations we refer to the textbooks \cite{ARS} and \cite{ASS}.

\subsection{\texorpdfstring{$d$}{d}-cluster tilting subcategories and \texorpdfstring{$\td$}{td}-rigid pairs} We continue by recalling some central notions of higher-dimensional homological algebra.

For a subcategory $\cC$ of an abelian category $\cA$ and an object $A\in \cA$, we call a morphism $f\colon A\to C$ a \emph{left $\cC$-approximation of $A$}\index[definitions]{left approximation} if $C\in \cC$ and $-\circ f\colon \Hom{\cA}{C}{-}|_\cC\to \Hom{\cA}{A}{-}|_\cC$ is an epimorphism. Further, if every object $A\in\cA$ admit a left $\cC$-approximation, we call $\cC$ \emph{covariantly finite in $\cA$}\index[definitions]{covariantly finite subcategory}. The notions of \emph{right $\cC$-approximations} \index[definitions]{right approximation} and \emph{contravariantly finite}\index[definitions]{contravariantly finite subcategory} subcategories are defined dually. A \emph{functorially finite} \index[definitions]{functorially finite subcategory} subcategory is a subcategory which is both covariantly and contravariantly finite.

Let $\Lambda$ be a finite-dimensional algebra. Following \cite{iyama2008auslander} we recall the concept of a $d$-cluster tilting subcategory.

\begin{definition}\label{def:n-ct}
A functorially finite subcategory $\C\subseteq\modfin{\Lambda}$ is called  a \emph{$d$-cluster tilting subcategory} if \index[definitions]{$d$-cluster tilting subcategory}
\begin{align*}
\C &= \{ M \in\modfin{\Lambda} \mid \Ext_{\Lambda}^{i}(M,\C) = 0 \text{ for $1\leq i\leq d-1$}\}\\
&= \{ M \in\modfin{\Lambda} \mid \Ext_{\Lambda}^{i}(\C,M) = 0 \text{ for $1\leq i\leq d-1$}\}.
\end{align*}
If $\C=\add{C}$ for some module $C\in\modfin{\Lambda}$, then $C$ is called a \emph{$d$-cluster tilting module}\index[definitions]{$d$-cluster tilting module}.
\end{definition}

For more details on higher-dimensional Auslander--Reiten theory we refer to the survey \cite{iyama2008auslander}.

Notice that any subcategory with an additive generator is functorially finite. Notice also that a $d$-cluster tilting subcategory always contains the projective and the injective $\Lambda$-modules. We define the \emph{$d$-Auslander--Reiten translations} as in \cite{iyama2008auslander} by $\td=\tau\syzygy^{d-1}$ and $\tdo=\tau^{-}\syzygy^{-(d-1)}$\index[symbols]{T@$\td(M)$}\index[symbols]{T@$\tdo(M)$}. They are crucial in the study of $d$-cluster tilting subcategories.
\begin{theorem}[{\cite[Theorem 2.8 and 2.9]{iyama2008auslander}}]
    Let $\cC$ be a $d$-cluster tilting subcategory of $\modfin{\Lambda}$ and $\underline{\cC}$ and $\overline{\cC}$ the corresponding subcategories of $\underline{\mathrm{mod}}({\Lambda})$ and $\overline{\mathrm{mod}}(\Lambda)$. Then $\td$ and $\tdo$ induce mutually quasi-inverse equivalences
    \[
    \td\colon \underline{\cC}\to\overline{\cC}\quad \text{ and }\quad \tdo\colon \overline{\cC}\to\underline{\cC}.
    \]    
\end{theorem}

The following definition introduces the main objects of interest in this article. Parts (a) and (b) in it were first introduced in \cite{JACOBSEN2020119}. 

\begin{definition}\label{def:taud-tilting}
Let $\C$ be a $d$-cluster tilting subcategory of $\modfin{\Lambda}$. Let $M\in\C$ and $P\in\modfin{\Lambda}$ be a projective module. Then
\begin{enumerate}[label=(\alph*)]
    \item $M$ is called \emph{$\td$-rigid}\index[definitions]{$\td$-rigid module} if $\Hom{\Lambda}{M}{\td(M)}=0$.
    \item The pair $(M,P)$ is called a \emph{$\td$-rigid pair}\index[definitions]{$\td$-rigid pair} if $M$ is $\td$-rigid and $\Hom{\Lambda}{P}{M}=0$.
    \item The pair $(M,P)$ is called a \emph{summand-maximal $\td$-rigid pair}\index[definitions]{$\td$-rigid pair!summand-maximal}, if it is a $\td$-rigid pair and for any other $\td$-rigid pair $(N,Q)$ we have $\abs{N}+\abs{Q}\leq \abs{M}+\abs{P}$.
\end{enumerate}
\end{definition}

Note that in particular a summand-maximal $\td$-rigid pair $(M,P)$ has at least $\abs{\Lambda}$ summands, since $(\Lambda,0)$ is $\td$-rigid.

\subsection{Nakayama algebras}
We now introduce the main class of algebras which we study in this article.

If $Q$ is a quiver, we denote by $R=R_Q$\index[symbols]{Rc@$R_Q$} the ideal of the path algebra $\K Q$ generated by all arrows of $Q$. We denote by $\ALinOr{n}$\index[symbols]{A@$\ALinOr{n}$} the quiver
\[
\begin{tikzcd}
	\ALinOr{n}\colon\qquad n\rar{}&n-1\rar{}&\cdots\rar{}&3\rar{}&2\rar{}&1.
\end{tikzcd}
\]
A bound quiver algebra $\Lambda$ of the form $\Lambda=\K \ALinOr{n}/I$ is called a \emph{linear Nakayama algebra} and if $I=R^l$ for some $l\geq 2$, then $\Lambda$ is called \emph{homogeneous}. 

We denote the bound quiver algebra $\K\ALinOr{n}/ R^l$ by $\Lambda(n,l)$\index[symbols]{L@$\Lambda(n,l)$}. The category $\modfin{\Lambda(n,l)}$ is well-understood. We proceed by recalling a description of the category $\modfin{\Lambda(n,l)}$ by describing the indecomposable modules and the morphisms between them. For the rest of this section we set $\Lambda=\Lambda(n,l)$.

Indecomposable $\Lambda$-modules are defined uniquely by their support which in turn is given by an interval $[a,b]\subseteq[1,n]$ with $b-a\leq l-1$. Thus, we denote by $\ind{a}{b}$ the indecomposable $\Lambda$-module with support $[a,b]$\index[symbols]{M@$\ind{a}{b}$}. As a representation of $\ALinOr{n}$ bound by $R^l$, this is given by setting $K$ in the vertices $i\in Q_0$ with $a\leq i \leq b$, and by setting all the morphisms corresponding to arrows in $\ALinOr{n}$ to be the identity if possible and zero otherwise. In particular, we can compute the length of $\ind{a}{b}$ as
\[
\length{\ind{a}{b}} = b - a + 1.
\]
If $b<a$, or $b<1$ or $n<a$ then we set $\ind{a}{b}=0$. Morphisms in $\modfin{\Lambda}$ are described by the following lemma.

\begin{lemma}[{\cite[Lemma 2.4]{ADACHI2016227}}]\label{Lemma:HomSpaceAdachi}
    Let $\ind{a}{b}$ and $\ind{c}{d}$ be indecomposable $\Lambda$-modules. The following conditions are equivalent:
    \begin{enumerate}
        \item[(a)] $\Hom{\Lambda}{\ind{a}{b}}{\ind{c}{d}}\neq 0$.
        \item[(b)] $\Hom{\Lambda}{\ind{a}{b}}{\ind{c}{d}}=\K$.
        \item[(c)] $b\in [c,d] $ and $c\in [a,b]$.
    \end{enumerate}
    Moreover, if $d-c\leq b-a $, then the following is also equivalent:
    \begin{enumerate}
        \item[(d)] $\Hom{\Lambda}{P(b)}{\ind{c}{d}}\neq 0$.
    \end{enumerate}
\end{lemma}

In particular, we have that if $M,N\in\modfin{\Lambda}$ are indecomposable, then there exists a non-zero morphism from $M$ to $N$ if and only if the top of $M$ is a composition factor of $N$ and the socle of $N$ is a composition factor of $M$. The following lemma classifies indecomposable projective, injective and projective-injective $\Lambda$-modules. 

\begin{lemma}[{\cite[Lemma 4.3]{VASO20192101}}]\label{lem:proj inj}
    Let $k\in [1,n]$ be a vertex of $\ALinOr{n}$.
    \begin{enumerate}[label=(\alph*)]
        \item For the indecomposable projective $\Lambda$-module $P(k)$ we have that 
        \[
        P(k)=\begin{cases}
        \ind{1}{k},&\text{if }1\leq k\leq l,\\
        \ind{k-l+1}{k},&\text{if }l+1\leq k\leq n.
        \end{cases}
        \]
        \item For the indecomposable injective $\Lambda$-module $I(k)$ we have that 
        \[
        I(k)=\begin{cases}
            \ind{k}{k+l-1},&\text{if }1\leq k\leq n-l+1,\\
            \ind{k}{n},&\text{if }n-l+2\leq k\leq n.
        \end{cases}
        \]
        \item An indecomposable $\Lambda$-module $\ind{a}{b}$ is both projective and injective if and only if $b-a=l-1$.
    \end{enumerate}
\end{lemma}

\subsection{\texorpdfstring{$d$}{d}-cluster tilting subcategories for \texorpdfstring{$\Lambda(n,l)$}{L(n,l)}}\label{subsec:d-cluster tilting subcategories for Lambda(n,l)}
Our motivation for studying the homogeneous Nakayama algebras $\Lambda(n,l)$ is that not only their representation theory is well-understood, but also that they provide a non-trivial class of examples where one can apply higher-dimensional Auslander--Reiten theory. In this paragraph we recall these examples through the classification of $d$-cluster tilting subcategories of $\modfin{\Lambda(n,l)}$. The following theorem is the main result.

\begin{theorem}[{\cite[Theorem 2]{VASO20192101}}]\label{thm:VasoClassifyAcyclicCluster}
  Let $\Lambda=\Lambda(n,l)$. There exists a $d$-cluster tilting subcategory $\cC\subseteq\modfin{\Lambda}$ if and only if there exists $p\geq 1$ such that
  \begin{equation}\label{eq:connection between n,p,d,l}
      n = (p-1)\left(\frac{d-1}{2}l+1\right)+\frac{l}{2}
  \end{equation}
  and either
  \begin{enumerate}
      \item[(i)] $l=2$, or
      \item[(ii)] $l>2$ and $d$ and $p$ are even
  \end{enumerate}
  holds. The $d$-cluster tilting subcategory $\C$ is then given by 
    \begin{equation}\label{eq:TheNCT} \C=\add{\bigoplus_{r=0}^{p-1} \tau_d^{-r} (\Lambda)}.
    \end{equation}
\end{theorem}

\begin{remark}\label{rem:d=gldim(Lambda) for Nakayama}
    If an algebra $A$ admits a $d$-cluster tilting module for $d=\gldim(A)$, then $A$ is called \emph{$d$-representation-finite}.\index[definitions]{$d$-representation-finite algebra} This case is of special interest in higher-dimensional homological algebra. From \cite[Theorem 3]{VASO20192101} we have that $\Lambda=\Lambda(n,l)$ admits a $d$-cluster tilting subcategory for $d=\gldim(\Lambda)$ if and only if $d=2\tfrac{n-1}{l}$. Using Theorem \ref{thm:VasoClassifyAcyclicCluster} it follows that this is the case if and only if $p=2$.
\end{remark}

For the rest of the section we fix integers $n$, $l$ and $d$ such that there exists a $d$-cluster tilting subcategory $\cC\subseteq\modfin{\Lambda}$, where $\Lambda=\Lambda(n,l)$. The indecomposable modules lying in the subcategory $\cC$ are crucial in our investigation. As it can be inferred from Theorem \ref{thm:VasoClassifyAcyclicCluster}, they are of two kinds: they are either projective-injective, or they are of the form $\tau_d^{-r}(P)$ for some indecomposable projective non-injective module $P\in\modfin{\Lambda}$. In fact, the modules of the second kind can be grouped with respect to the power $r$. To describe how this works, we need to recall how to apply the functors $\td$ and $\tdo$ to indecomposable modules in $\cC$.

\begin{lemma}[{\cite[Lemma 4.8]{VASO20192101}}]\label{lem:computation of taud}
Let $\ind{a}{b}\in\C$ be an indecomposable module.
\begin{enumerate}[label=(\alph*)]
    \item If $\ind{a}{b}$ is nonprojective, then
    \[
    \td(\ind{a}{b})=\ind{b-\tfrac{d}{2}l}{a-\tfrac{d-2}{2}l-2}
    \]
    and the length of $\td(\ind{a}{b})$ satisfies
    \[
    \length{\td (\ind{a}{b})}=l-\length{\ind{a}{b}}.
    \]
    \item If $\ind{a}{b}$ is noninjective, then
    \[
    \tdo(\ind{a}{b})=\ind{b+\tfrac{d-2}{2}l+2}{a+\tfrac{d}{2}l}
    \]
    and the length of $\tdo(\ind{a}{b})$ satisfies
    \[
    \length{\tdo(\ind{a}{b})}=l-\length{\ind{a}{b}}.
    \]
\end{enumerate}
\end{lemma}
	
Using (\ref{eq:TheNCT}) and Lemma \ref{lem:computation of taud} it is easy to give a description of all modules in the $d$-cluster tilting subcategory $\cC$. First, by Lemma \ref{lem:proj inj}, we have that 
\[
P([l,n]) = \bigoplus_{i\in [l,n]} P(i)
\]
is the direct sum of all indecomposable projective-injective $\Lambda$-modules. We denote by $\projinj$\index[symbols]{Pe@$\projinj$} the subcategory $\add{P([l,n])}$ of $\modfin{\Lambda}$. Notice that by (\ref{eq:TheNCT}) we have that $\projinj\subseteq \cC$. Next, for $1\leq i\leq p$ we define 
\begin{equation}\label{eq:the definition of s_i}
s_i = \begin{cases}
    (i-1)\frac{d-1}{2}l+i, &\mbox{if $i$ is odd,} \\
    (i-1)(\frac{d-1}{2}l+1)+\frac{l}{2}, &\mbox{if $i$ is even.}
\end{cases}
\end{equation}
\index[symbols]{Sc@$s_i$}In particular, we immediately have that $s_1=1$ and, using (\ref{eq:connection between n,p,d,l}), we also obtain that $s_p=n$. A direct computation using Lemma \ref{lem:computation of taud} gives that a simple module $S$ is in $\cC$ if and only if $S\isom S(s_i)$ for some $1\leq i\leq p$. Thus the set $\{S(s_1),S(s_2),\ldots,S(s_p)\}$ is a complete and irredundant collection of representatives of isomorphism classes of simple $\Lambda$-modules in $\cC$. Moreover, immediately from the definition of $s_i$, we can show the following equalities which we use throughout 
\begin{equation}\label{eq:difference of consecutive simples}
    s_{i+1}-s_{i}=\begin{cases}
        \frac{d}{2}l, &\mbox{if $i$ is odd,} \\
        \frac{d-2}{2}l+2, &\mbox{if $i$ is even,}
    \end{cases} \text{ and } s_{i+2}-s_{i}=(d-1)l+2.
\end{equation}
We have seen that the sequence of integers $s_1,\ldots,s_p$ can be used to describe the simple $\Lambda$-modules in $\cC$, but it can also be used to describe the rest of the indecomposable modules in $\cC$, excluding the projective-injective modules which as we have seen are easy to describe. Indeed, we define 
\begin{equation}\label{eq:definition of diagonals}
    \diagm{i} = \begin{cases}
    \bigoplus\limits_{x=1}^{l-1} \ind{s_i}{s_i+x-1}, &\mbox{if $i$ is odd,}\\
    \bigoplus\limits_{x=1}^{l-1} \ind{s_i-x+1}{s_i}, &\mbox{if $i$ is even.}
    \end{cases}
\end{equation}
\index[symbols]{Da@$\diagm{i}$}\index[symbols]{Dd@$\diag{i}$}We also set $\diag{i}=\add{\diagm{i}}$ and we call it \emph{the $i$-th diagonal};\index[definitions]{$i$-th diagonal} we will soon motivate the name by an example. Notice that by Lemma \ref{lem:proj inj} it follows that $\diag{1}$ is the category of projective non-injective $\Lambda$-modules and that $\diag{p}$ is the category of injective non-projective $\Lambda$-modules. Another direct computation using Lemma \ref{lem:computation of taud} yields that $\tdo(\diag{i})=\diag{i+1}$ for $1\leq i\leq p-1$. Using (\ref{eq:TheNCT}), we can conclude that
\begin{equation}\label{eq:description of C}
\C=\add{ \left(\bigoplus_{i=2}^{p} D_i \right)
\oplus P([1,n])}.
\end{equation}
The following example is indicative of the general picture.

\begin{example}\label{Example:NamingDiagonals} 
The algebra $\Lambda(9,3)=\K\ALinOr{9}/R^3$ admits a $2$-cluster subcategory $\cC$. Below we have drawn the Auslander--Reiten quiver of this algebra, where the indecomposable modules are given by their composition series. The indecomposable modules in $\cC$ are the boxed vertices in the Auslander--Reiten quiver.

    $$
    \includegraphics{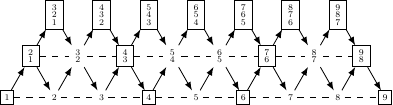}
    $$
The relative position of the vertices carries all the essential information of the associated indecomposables. Hence in the following we omit the arrows and use circles for indecomposable modules and filled circles for modules in $d$-cluster tilting subcategories. With this convention, the above picture simplifies to
    $$
    \includegraphics{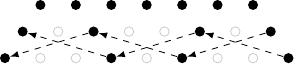}
    $$
    where we have also drawn the action of $\tau_2^{\phantom{-}}$ using dashed arrows. The projective-injective indecomposable modules in $\cC$ correspond to the top row of filled circles. The other eight indecomposable modules in $\cC$ form four diagonals $\diag{1}$, $\diag{2}$, $\diag{3}$, $\diag{4}$, read from left to right. For example, we have that $\diag{1}=\add{S(1)\oplus P(2)}$ and $\diag{4}=\add{S(9)\oplus I(8)}$.
\end{example}

Let $1\leq x\leq l-1$. For later convenience, we denote by $\updiagm{i}{x}$ the direct sum of all indecomposable direct summands of $\diagm{i}$ which have length larger than or equal to $x$. Similarly, we denote by $\downdiagm{i}{x}$ the direct sum of all indecomposable summands of $\diagm{i}$ which have length smaller than or equal to $x$, and we also denote $\downdiagm{i}{0}=0$. We also set $\updiag{i}{x}=\add{\updiagm{i}{x}}$ and $\downdiag{i}{x}=\add{\downdiagm{i}{x}}$.\index[symbols]{Db@$\downdiagm{i}{x}$}\index[symbols]{De@$\downdiag{i}{x}$}\index[symbols]{Dc@$\updiagm{i}{x}$}\index[symbols]{Df@$\updiag{i}{x}$} 

\begin{remark}\label{rem:first observations}
\begin{enumerate}
    \item[(a)] Let $1\leq i\leq p$. If $i$ is odd, then all indecomposable modules in $\diag{i}$ have the same socle $S(s_i)$. Similarly, if $i$ is even, then all indecomposable modules in $\diag{i}$ have the same top $S(s_i)$. 
    \item[(b)] For $2\leq i\leq p$ we have $\td(\updiag{i}{t})=\downdiag{i-1}{l-t}$ and $\td(\downdiag{i}{t})=\updiag{i-1}{l-t}$.  
\end{enumerate}
\end{remark}

If $M$ is a module in $\cC$, then computing $\Hom{\Lambda}{M}{\td(M)}$ is all that is needed for checking if $M$ is $\td$-rigid. On the other hand, if $M$ is in $\cC$, then we have seen that each of its indecomposable summands either is projective-injective or lies in some diagonal $\diag{i}$. Hence it becomes important to compute homomorphisms between indecomposable modules lying in the diagonals $\diag{i}$. First we notice that when indecomposable modules lie in the same diagonal, then the only thing that determines the existence of a morphism between them is the length of the modules.

\begin{lemma}
    \label{Lemma:MorhphismInDiagonal}
    Let $1\leq i\leq p$ and $M,N\in \diag{i}$ be indecomposable. 
    \begin{enumerate}
        \item[(a)] If $i$ is even, then $\Hom{\Lambda}{M}{N}\neq 0$ if and only if $\length{N}\leq \length{M}$,
        \item[(b)] If $i$ is odd, then $\Hom{\Lambda}{M}{N}\neq 0$ if and only if $\length{M}\leq \length{N}$.
    \end{enumerate}
\end{lemma}

\begin{proof}
Follows immediately by Lemma \ref{Lemma:HomSpaceAdachi}.
\end{proof}

Things are generally simpler when indecomposable modules lie in different diagonals. In this case, there are no nonzero morphisms between them, with one notable exception which is recorded in the following lemma.

\begin{lemma}\label{Lemma:DiagZeroHom}
Let $1\leq i\leq p$ and $1\leq j\leq p$ with $i\neq j$. Then the following are equivalent.
\begin{enumerate}
    \item[(a)] $\Hom{\Lambda}{\diag{i}}{\diag{j}} \neq 0$.
    \item[(b)] $d=2$ and $l>3$ and $i$ is odd and $j=i+1$. 
\end{enumerate}
\end{lemma}

\begin{proof}
Assume first that (b) holds and we show (a). Since $l>3$ we have
\[
s_i < s_i+2 \leq s_i+l-2 < s_i+l, 
\]
where $s_i$ is defined in (\ref{eq:the definition of s_i}). By Lemma \ref{Lemma:HomSpaceAdachi} we obtain that \[
\Hom{\Lambda}{\ind{s_i}{s_i+l-2}}{\ind{s_i+2}{s_i+l}}\neq 0.\]
Clearly the indecomposable module $\ind{s_i}{s_i+l-2}$ is in $\diag{i}$ and it is enough to show that $\ind{s_i+2}{s_i+l}\in\diag{j}=\diag{i+1}$. Since $i$ is odd and $d=2$ and using (\ref{eq:difference of consecutive simples}) we have $s_i=s_{i+1}-l$. Hence $\ind{s_i+2}{s_i+l}=\ind{s_{i+1}-l+2}{s_{i+1}}\in\diag{i+1}$, as required.

Assume now that (a) holds and we show (b). By (a) we have that there exist indecomposable modules $\ind{a}{b}\in\diag{i}$ and $\ind{c}{d}\in\diag{j}$ such that 
\[
\Hom{\Lambda}{\ind{a}{b}}{\ind{c}{d}}\neq 0.
\]
By (\ref{eq:definition of diagonals}) we have that $b=s_i+x-1$ and $c=s_j-y+1$ for some $x,y\in [1,l-1]$. By Lemma \ref{Lemma:HomSpaceAdachi} we have that
\begin{equation}\label{eq:nonzero hom for special case}
    a\leq c \leq b \leq d.
\end{equation}
Since $\Lambda$ is representation-directed and since diagonals do not intersect we have that $j\geq i$, and by the assumption $i\neq j$ we obtain that $j\geq i+1$. 

We first claim that if $1\leq m\leq j$ holds for some $m$ odd, then $c\geq s_m$. Indeed, if $j=m$ then $j$ is odd and so $c=s_j=s_m$. On the other hand, if $j\geq m+1$, then $s_j\geq s_{m+1}$ and using (\ref{eq:difference of consecutive simples}) we obtain
\[
c = s_j-y+1 \geq s_{m+1} - y+1 = \frac{d}{2}l+s_m-y+1 \geq l+ s_m -(l-1)+1 > s_m.
\]

Assume now towards a contradiction that $i$ is even. Then $b=s_i$ and $i+1$ is odd. Since $j\geq i+1$, the claim gives that $c\geq s_{i+1}$. Then (\ref{eq:nonzero hom for special case}) gives $s_{i+1}\leq c \leq b=s_{i}$ which is a contradiction. Hence $i$ is odd.

Assume now towards a contradiction that $j\geq i+2$. Since $i+2$ is odd, the claim gives that $c\geq s_{i+2}$. Then (\ref{eq:nonzero hom for special case}) gives $s_{i+2}\leq c\leq b=s_i$ which is a contradiction. Hence $j=i+1$.

Since $i$ is odd and $j=i+1$, we obtain using (\ref{eq:difference of consecutive simples}) that $c=s_{i+1}-y+1=s_i+\tfrac{d}{2}l-y+1$. Then, using (\ref{eq:nonzero hom for special case}) we have 
\[
s_i+\tfrac{d}{2}l-y+1 \leq s_{i}+x-1
\]
which we can rewrite as $\tfrac{d}{2}l\leq x+y-2$. Since $x,y\in [1,l-1]$, we obtain that $\tfrac{d}{2}l\leq 2l-4$. Rearranging this once more we obtain that
\[
8\leq (4-d)l.
\]
If $l=2$, then this inequality gives $0\leq -2d$ which is impossible. If $l=3$, then this inequality gives $4\geq 3d$ which is again impossible since $d\geq 2$. Hence $l>3$. Then $d$ must be even by Theorem \ref{thm:VasoClassifyAcyclicCluster}. But if $d\geq 4$ then the above inequality is again not satisfied. Hence $d=2$. 
\end{proof}
\label{Section:Preliminaries}

\section{Summand-maximal \texorpdfstring{$\td$}{td}-rigid pairs for \texorpdfstring{$\Lambda(n,l)$}{L(n,l)}}
\label{Section:strongly maximal td-rigid}
\subsubsection*{Aim.} In this section we fix an algebra $\Lambda=\Lambda(n,l)$ that satisfies the conditions of Theorem \ref{thm:VasoClassifyAcyclicCluster}, so that it admits a $d$-cluster tilting subcategory $\cC$. Our aim is to investigate the structure of summand-maximal $\td$-rigid pairs for $\Lambda$. Before we proceed, let us outline the basic steps of our strategy.

\subsubsection*{Describing $\td$-rigid pairs.} A basic $\td$-rigid pair $(M,P)$ can be described using the indecomposable summands of $M$ and $P$. By definition, all of these indecomposable summands lie in $\cC$. On the other hand, using (\ref{eq:description of C}), we have that we can separate the indecomposable modules in $\cC$ in two different collections: the indecomposable projectives and the modules in the diagonals $\diag{2},\ldots,\diag{p}$. Each indecomposable projectives can be described uniquely using a number from $1,\ldots,n$. Each indecomposable module in the diagonals can also be described uniquely by specifying in which diagonal it lies and its length. Thus to describe a basic $\td$-rigid pair $(M,P)$ we need a subset $\red$ of $[1,n]$ corresponding to the indecomposable projective summands of $M$, a subset $\blue$ of $[1,n]$ corresponding to the indecomposable (neccessarily projective) summands of $P$ and then a collection of indecomposable modules in $\diag{2},\ldots,\diag{p}$ corresponding to the non-projective summands of $M$.

We first introduce some notation to be able to describe our $\td$-rigid pairs using this point of view. The first part in our strategy is to answer the following question: if we have a collection of combinatorial data as described in the previous paragraph, under what additional conditions does it define a $\td$-rigid pair? In this way we characterize basic $\td$-rigid pairs in a purely combinatorial way. This is the aim of Section \ref{subsec:basic td-rigid pairs}.

\subsubsection*{Local restrictions on summand-maximal $\td$-rigid pairs.} The next part of our strategy is to observe that the combinatorial data described above has some important local properties. More specifically, assume that a basic $\td$-rigid pair $(M,P)$ is supported in some consecutive diagonals $\diag{i},\ldots,\diag{i+k}$. Then we define an appropriate interval $\Xi\subseteq [1,n]$ which lies ``close'' to the diagonals $\diag{i},\ldots,\diag{i+k}$. We may then count the amount of indecomposable summands of $(M,P)$ that lie in the diagonals $\diag{i},\ldots,\diag{i+k}$ and add to that number the amount of indecomposable projective summands of $(M,P)$ that are of the form $P(k)$ for some $k\in\Xi$. It turns out that this number of indecomposable summands of $(M,P)$ ``in the neighbourhood of'' $\diag{i},\ldots,\diag{i+k}$ cannot exceed $\abs{\Xi}$. In fact, with a few crucial exceptions, it cannot even exceed $\abs{\Xi}-1$. These observations are made precise in Section \ref{subsec:basic td-rigid pairs supported in consecutive diagonals}.

The main aim of Section \ref{subsec:admissible configurations} is to study how this local constraint relates to summand-maximal $\td$-rigid pairs. It turns out that the local gap described above cannot be patched by adding indecomposable summands outside of this local area. Together with some more obvious observations, we see that this property forces summand-maximal $\td$-rigid pairs to be supported on at most two consecutive diagonals (again, with one exception). Then we only need to check what kind of configurations are possible when one restricts to at most two consecutive diagonals. We provide a list that describes such configurations in Definition \ref{def:admissible configurations}.  

\subsubsection*{Gluing the local configurations to classify summand-maximal $\td$-rigid pairs.} With that done, the last part of the strategy in classifying summand-maximal $\td$-rigid pairs is to find how the local admissible configurations can be glued together so that the end result is not only a $\td$-rigid pair but is also summand-maximal. This is the aim of Section \ref{subsec:well-configured pairs}. This leads immediately to the notion of well-configured pairs, introduced in Definition \ref{def:well-configured}, which completely characterizes summand-maximal $\td$-rigid pairs in $\modfin{\Lambda}$ in an explicit combinatorial way. 

Another important result of this characterization is that summand-maximal $\td$-rigid pairs are also characterized by a simple counting argument: the number of indecomposable summands of $M\oplus P$ is equal to $\abs{\Lambda}=n$ if and only if $(M,P)$ is a summand-maximal $\td$-rigid pair. As this condition is significantly easier to check, we highlight this in examples even before showing its equivalence to the property of being a summand-maximal $\td$-rigid pair.

The following section is quite heavy on notation. We have tried to give many examples and use conventions that make it easier to remember them, but the interested reader is advised to make use of the index and the table at the end of the article to get helpful reminders.

\subsubsection*{A special case.} As it can be inferred from the explanation of our strategy, there are a few exceptions that need to be dealt separately. The following remark encapsulates the different cases and the essential reason for their differences.

\begin{remark}\label{rem:the case d=2 and l>3 is different}
If $d>2$ or $l\leq 3$, then Lemma \ref{Lemma:DiagZeroHom} gives that $\Hom{\Lambda}{\diag{i}}{\diag{j}}=0$ for all $1\leq i\neq j\leq p$. However, if $d=2$ and $l>3$, then this property does not hold. This makes the study of $\td$-rigid pairs for $d=2$ and $l>3$ require some more attention. This explains the need for considering this special case more closely in the sequel.   
\end{remark}

\subsection{Basic \texorpdfstring{$\td$}{td}-rigid pairs.} \label{subsec:basic td-rigid pairs}

The main aim of this section is to give a combinatorial characterization of $\td$-rigid pairs in $\modfin{\Lambda}$. To be able to do this, we introduce some notation that lets us describe pairs $(M,P)$ with $M\in\cC$ and $P$ projective.

\begin{definition}\label{def:C-pair}
A \emph{$\cC$-pair}\index[definitions]{$\cC$-pair} is a pair $(M,P)$ where $M\in\cC$ and $P\in\proj\Lambda$ are modules such that $M\oplus P$ is basic. 
\end{definition}

Let $(M,P)$ be a $\cC$-pair. Using the description of $\cC$ in (\ref{eq:description of C}), we may assume that indecomposable summands of $M$ and $P$ are indecomposable summands of $P([1,n])\oplus \bigoplus_{i=2}^{p}\diagm{i}$. Then we may write
\[
M = M_{\text{pr}}\oplus \bigoplus_{i=2}^{p}M_i
\]
where $M_{\text{pr}}\in\add{P[1,n]}$\index[symbols]{M@$M_\text{pr}$}\index[symbols]{M@$M_i$} and where for $i\in [2,p]$ we let $M_i$ be the direct sum of indecomposable summands of $M$ in the diagonal $\diag{i}$. In particular we have that indecomposable summands of $\bigoplus\limits_{i=2}^{p}M_i$ are never projective. For $i\not\in[2,p]$ we set $M_i=0$. We denote $\cM=\add{M}$, $\cM_i=\add{M_i}$\index[symbols]{M@$\cM$}\index[symbols]{M@$\cM_i$}, $m_i=\abs{M_i}$\index[symbols]{M@$m_i$} and $\cP=\add{P}$. Moreover, we have $M_{\text{pr}}\in\add{P[1,n]}$ and $P\in \add{P[1,n]}$ and so there exist unique sets $\red,\blue\subseteq [1,n]$ such that $M_{\text{pr}}= P(\red)$ and $P= P(\blue)$. We call $(\red,\blue)$\index[symbols]{Ra@$\red$}\index[symbols]{B@$\blue$} the \emph{non-diagonal component}\index[definitions]{component!non-diagonal} of $M$. In pictures and examples we use Red colour and dashed lines when denoting an indecomposable module which is a summand of $M$ and we use Blue colour and bold lines when denoting an indecomposable projective module which is a summand of $P(\blue)$. Notice finally that since $M\oplus P$ is basic, we obtain that $\red\cap\blue=\varnothing$, which we use freely in the sequel. 

Clearly, if $(M,P)$ is a basic $\td$-rigid pair, it is also a $\cC$-pair. Given a $\cC$-pair $(M,P)$, our next aim is to give necessary and sufficient conditions on the data $\red$, $\blue$, $M_2,\ldots, M_p$ for $(M,P)$ to be a $\td$-rigid pair. The following is the first step towards that direction.

\begin{lemma}\label{lem:blue then red}
Let $(M,P)$ be a $\cC$-pair with non-diagonal component $(\red,\blue)$. Then the following are equivalent.
\begin{enumerate}
    \item[(a)] $\Hom{\Lambda}{P(\blue)}{P(\red)}=0$.
    \item[(b)] For every $x\in \blue$ and $y\in \red$ such that $1\leq x < y\leq n$, there exists $z\in [1,n]$ with $x<z<\cdots<z+l-2<y$ such that $(\red\cup\blue)\cap [z,z+l-2]=\varnothing$.
\end{enumerate}
\end{lemma}

\begin{proof}
Assume first that (b) holds but (a) fails and we reach a contradiction. By (a) there exist $x\in \blue$ and $y\in \red$ such that $\Hom{\Lambda}{P(x)}{P(y)}\neq 0$. In particular, if $P(x)=\ind{a}{x}$ and $P(y)=\ind{b}{y}$ for some $a$ and $b$, then $a\leq b\leq x \leq y $ by Lemma \ref{Lemma:HomSpaceAdachi}. Since $\red\cap\blue=\varnothing$, we have $x<y$. By part (b) we conclude that $y-b\geq y-x\geq l$, which contradicts $\Lambda=\Lambda(n,l)$. 

Assume now that (a) holds and we show that (b) also holds. Let $y\in \red$. If there exists $x<y$, without loss of generality, we may assume that $x$ is maximal such that $x<y$ and $x\in \blue$. Then $\Hom{\Lambda}{P(x)}{P(x+i)}\neq 0$ for $1\leq i\leq l-1$ by Lemma \ref{Lemma:HomSpaceAdachi}. Since $\Hom{\Lambda}{P(\blue)}{P(\red)}=0$, we have that $x+i\not\in\red$ for $1\leq i\leq l-1$. Since $x<y$ and $y\in\red$, we conclude that $x+l-1<y$. By maximality of $x$, we have that $(x+i)\not\in\blue$ for $1\leq i\leq l-1$. The claim follows by setting $z=x+1$.
\end{proof}

Let $(M,P)$ be a $\cC$-pair and let $i\in [2,p]$. Assume that $M_i\neq 0$. Recall that $M_i$ is the direct sum of all indecomposable summands of $M$ in the diagonal $\diag{i}$. Since the diagonal $\diag{i}$ has exactly one indecomposable module of each length $1,\ldots,l-1$ (see also (\ref{eq:definition of diagonals})) we obtain that there exists a unique set $X_i\subseteq [1,l-1]$\index[symbols]{X@$X_i$} such that
\begin{equation}\label{eq:the definition of M_i}
M_i = \begin{cases}
    \bigoplus\limits_{x\in X_i} \ind{s_i}{s_i+x-1}, &\mbox{if $i$ is odd,} \\
    \bigoplus\limits_{x\in X_i
    } \ind{s_i-x+1}{s_i}, &\mbox{if $i$ is even.}
\end{cases}
\end{equation}
In particular, an indecomposable summand of $\diagm{i}$ is a direct summand of $M_i$ if and only if it is of length $x\in X_i$. We set 
\begin{equation}\label{eq:the definition of l_i and n_i}
l_i = \begin{cases}
    \min{X_i}, &\mbox{if $M_i\neq 0$,} \\
    l, &\mbox{if $M_i=0$}
\end{cases}\quad \text{ and }\quad n_i =  \begin{cases}
    \max{X_i}, &\mbox{if $M_i\neq 0$,} \\
    0, &\mbox{if $M_i=0$.}
\end{cases}
\end{equation}
The definition of $l_i$\index[symbols]{L@$l_i$}\index[symbols]{N@$n_i$} and $n_i$ in the case $M_i=0$ may seem a bit arbitrary but it is motivated by the fact that it allows us to state many results without assuming that $M_i\neq 0$. The following observation follows immediately by the definitions.

\begin{lemma}\label{lem:immediate use of mi li ni}
Let $(M,P)$ be a $\cC$-pair and let $i\in[2,p]$. Then $1\leq l_i\leq l$ and $0\leq n_i\leq l-1$. Moreover, $m_i\neq 0$ if and only if $m_i\leq n_i-l_i+1$ and $1\leq l_i \leq n_i\leq l-1$ both hold. 
\end{lemma}

The next lemma is our second step towards the main result of this paragraph.

\begin{lemma}\label{lem:connection between lis and nis}
Let $(M,P)$ be a $\cC$-pair and $i\in [2,p]$.
\begin{enumerate}
    \item[(a1)] If $i$ is odd, then $\Hom{\Lambda}{M_{i-1}}{\td(M_{i})}=0$ if and only if $n_{i-1}+n_i\leq l-1$.
    \item[(a2)] If $i$ is odd, then $\Hom{\Lambda}{M_i}{\td(M_{i+1})}=0$ if and only if $l+1\leq l_{i}+l_{i+1}$.
    \item[(b1)] If $i$ is even, then $\Hom{\Lambda}{M_i}{\td(M_{i+1})}=0$ if and only if $n_{i}+n_{i+1}\leq l-1$.
    \item[(b2)] If $i$ is even, then $\Hom{\Lambda}{M_{i-1}}{\td(M_i)}=0$ if and only if $l+1\leq l_{i-1}+l_{i}$.
    \item[(c1)] If $i$ is odd and $d=2$, then $\Hom{\Lambda}{M_i}{\td(M_{i+2})}=0$ if and only if $n_i\leq l_{i+2}+1$.
    \item[(c2)] If $i$ is odd and $d=2$, then $\Hom{\Lambda}{M_{i-2}}{\td(M_i)}=0$ if and only if $n_{i-2}\leq l_{i}+1$.
\end{enumerate}
\end{lemma}

\begin{proof}
\begin{enumerate}
    \item[(a1)] If $M_{i-1}=0$ respectively $M_i=0$, then $n_{i-1}=0$ respectively $n_i=0$. Since both $n_{i}\leq l-1$ and $n_{i-1}\leq l-1$ hold by Lemma \ref{lem:immediate use of mi li ni}, the claim follows. 
    
    Now assume that $M_{i-1}\neq 0$ and $M_i\neq 0$. We have that 
    \[
    \Hom{\Lambda}{M_{i-1}}{\td(M_i)}=0
    \]
    if and only if 
    \[
    \Hom{\Lambda}{L}{\td(N)}=0 \text{ for all indecomposable } L\in \cM_{i-1} \text{ and } N\in \cM_{i}.
    \]
    By Lemma \ref{Lemma:MorhphismInDiagonal}(a), and since $i-1$ is even, this is equivalent to 
    \[
    \length{\td(N)}>\length{L} \text{ for all indecomposable } L\in \cM_{i-1} \text{ and } N\in \cM_{i}.
    \]
    Using $\length{\td(N)}=l-\length{N}$ from Lemma \ref{lem:computation of taud}(a), we obtain that this is equivalent to 
    \[
    \length{L}+\length{N}\leq l-1 \text{ for all indecomposable } L\in \cM_{i-1} \text{ and } N\in \cM_{i}.
    \]
    Let $L'=\ind{s_{i-1}-n_{i-1}+1}{s_{i-1}}\in\cM_{i-1}$ and $N'=\ind{s_i}{s_i+n_i-1}\in\cM_i$. Since $\length{L}\leq \length{L'}=n_{i-1}$ for all indecomposable $L\in \cM_{i-1}$ and $\length{N}\leq \length{N'}=n_i$ for all indecomposable $N\in\cM_i$, the above is equivalent to $n_{i-1}+n_{i}\leq l-1$, as required.

    \item[(a2)] Similar to part (a1).

    \item[(b1)] Follows immediately by applying part (a1) to $i+1$.

    \item[(b2)] Follows immediately by applying part (a2) to $i-1$.

    \item[(c1)] If $M_i=0$ respectively $M_{i+2}=0$ then $n_i=0$ respectively $l_{i+2}=l$. Since both $n_i\leq l-1$ and $1\leq l_{i+2}$ hold by Lemma \ref{lem:immediate use of mi li ni}, the claim follows. 
    
    Now assume that $M_{i}\neq 0$ and $M_{i+2}\neq 0$. Then an indecomposable module in $\cM_{i}$ is of the form $L_x=\ind{s_i}{s_i+x-1}$ and an indecomposable module in $\cM_{i+2}$ is of the form $N_y=\ind{s_{i+2}}{s_{i+2}+y-1}$. We compute using Lemma \ref{lem:computation of taud}(a) that
    \[
    \td(N_y)=\ind{s_{i+2}+y-1-l}{s_{i+2}-2}\overset{(\ref{eq:difference of consecutive simples})}{=} \ind{s_i+y+1}{s_i+l}.
    \]
    We have that $\Hom{\Lambda}{M_{i}}{\td(M_{i+2})}=0$ if and only if 
    \[
    \Hom{\Lambda}{L}{\td(N)}=0 \text{ for all indecomposable } L\in \cM_{i} \text{ and } N\in \cM_{i+2}.
    \]
    By the above computation, this is equivalent to
    \[
    \Hom{\Lambda}{\ind{s_i}{s_i+x-1}}{\ind{s_i+y+1}{s_i+l}}=0 \text{ for all $x$ and $y$ such that } L_x\in \cM_{i} \text{ and } N_y\in \cM_{i+2}.
    \]
    By Lemma \ref{Lemma:HomSpaceAdachi} we obtain that this is equivalent to $s_i+x-1\not\in [s_{i}+y+1,s_i+l]$ or $s_{i}+y+1\not\in [s_i,s_i+x-1]$. Since $x\leq l-1$ and $1\leq y$, both of the above cases are equivalent to 
    \[
    x\leq y+1 \text{ for all $x$ and $y$ such that } L_x\in \cM_{i} \text{ and } N_y\in \cM_{i+2}.
    \]
    Since $L_{n_i}\in\cM_i$ and $N_{l_{i+2}}\in\cM_{i+2}$ and also $x\leq n_i$ and $l_{i+2}\leq y$ for all $x$ and $y$ such that $L_x\in\cM_i$ and $N_y\in\cM_{i+2}$, the above is equivalent to $n_{i}\leq l_{i+2}+1$, as required.

    \item[(c2)] Follows immediately by applying part (c1) to $i-2$. \qedhere
\end{enumerate}
\end{proof}

Before giving the third and last step towards the main result, we need to introduce some more notation. Let $(M,P)$ be a $\cC$-pair with non-diagonal component $(\red,\blue)$. For $i\in [2,p]$ we define
\begin{equation}\label{eq:the definition of notblue and notred}
\notblue{i} \coloneqq \begin{cases}
    [s_{i},s_{i}+n_{i}-1], &\mbox{if $i$ is odd,} \\
    [s_{i}-n_{i}+1,s_{i}], &\mbox{if $i$ is even,}
\end{cases}
\text{ and }
\notred{i} \coloneqq \begin{cases}
    [s_{i-1}-(l-l_i)+1,s_{i-1}], &\mbox{if $i$ is odd,} \\
    [s_{i-1}, s_{i-1}+(l-l_i)-1], &\mbox{if $i$ is even.}
\end{cases}
\end{equation}

Note that by definition we have that $\abs{\notred{i}}=l-l_i$\index[symbols]{Rb@$\notred{i}$}\index[symbols]{B@$\notblue{i}$} and $\abs{\notblue{i}}=n_i$; we use these two equalities freely throughout. The aim of the intervals $\notred{i}$ respectively $\notblue{i}$ is to encode indecomposable projective modules which can not be in $\red$ respectively $\blue$, as the following lemma shows.

\begin{lemma}\label{lem:the intervals Bi and Ri}
Let $(M,P)$ be a $\cC$-pair with non-diagonal component $(\red,\blue)$ and $i\in [2,p]$.
\begin{enumerate}
    \item[(a)] $\Hom{\Lambda}{P(\blue)}{M_i}=0$ if and only if $\blue\cap\notblue{i}=\varnothing$.
    \item[(b)] $\Hom{\Lambda}{P(\red)}{\td(M_i)}=0$ if and only if $\red\cap\notred{i}=\varnothing$. 
\end{enumerate}
\end{lemma}

\begin{proof}
\begin{enumerate}
    \item[(a)] If $M_i=0$, then $\notblue{i}=\varnothing$ and the claim follows trivially. Assume that $M_i\neq 0$. Let $M(a,b)\in\cM_i$ be the indecomposable summand of $M_i$ of largest length. Then $[a,b]=\notblue{i}$ by the definitions involved, that is (\ref{eq:the definition of M_i}) and (\ref{eq:the definition of l_i and n_i}). Then we have that $\Hom{\Lambda}{P(\blue)}{M_i}=0$ if and only if $\blue\cap [a,b]=\varnothing$ if and only if $\blue\cap \notblue{i}=\varnothing$, as required.
    \item[(b)] Similar to part (a) by considering the module $\td(M(c,d))$ where $M(c,d)$ is the indecomposable summand of $M_i$ of smallest length. \qedhere
\end{enumerate}    
\end{proof}

With this in mind we can give our main result for this paragraph.

\begin{proposition}\label{prop:description of td-rigid pairs}
    Let $(M,P)$ be a $\cC$-pair with non-diagonal component $(\red,\blue)$. Then $(M,P)$ is a basic $\td$-rigid pair if and only if the following conditions all hold.
    \begin{enumerate}
        \item[(a)] For every $x\in \blue$ and $y\in \red$ such that $1\leq x < y\leq n$, there exists $z\in [1,n]$ with $x<z<\cdots<z+l-2<y$ such that $(\red\cup\blue)\cap [z,z+l-2]=\varnothing$.
        \item[(b)] For all $i\in [2,p]$ we have that
        \begin{enumerate}
            \item[(b1)] if $i$ is odd, then $n_{i-1}+n_i\leq l-1$ and $l+1\leq l_i+l_{i+1}$ hold. If moreover $d=2$, then also $n_{i}\leq l_{i+2}+1$ and $n_{i-2}\leq l_{i}+1$ hold,
            \item[(b2)] if $i$ is even, then $n_i+n_{i+1}\leq l-1$ and $l+1\leq l_{i-1}+l_i$ hold, and
            \item[(b3)] $\red\cap\notred{i}=\varnothing$ and $\blue\cap\notblue{i}=\varnothing$.
        \end{enumerate}
    \end{enumerate}
\end{proposition}

\begin{proof}
The $\cC$-pair $(M,P)$ is $\td$-rigid if and only if 
\[
\begin{cases}
    \Hom{\Lambda}{P(\blue)}{P(\red)} =0, \\    
    \Hom{\Lambda}{M}{\td(M)}=0, \\
    \Hom{\Lambda}{P(\blue)}{M}=0, \\
    \Hom{\Lambda}{P(\red)}{\td(M)}=0
\end{cases}
\]
all hold. The first of these equalities is equivalent to condition (a) by Lemma \ref{lem:blue then red}. The second equality is equivalent to the conditions (b1) and (b2) by Lemma \ref{lem:connection between lis and nis}, keeping in mind Lemma \ref{Lemma:DiagZeroHom}. The third and fourth equalities are equivalent to condition (b3) by Lemma \ref{lem:the intervals Bi and Ri}.
\end{proof}

We finish this paragraph with some observations about basic $\td$-rigid pairs and some examples. First notice that if we have some $\td$-rigid pair $(M,P)$, and we know that some $k\in [1,n]$ is both in $\notred{i}$ and $\notblue{j}$ for some $i,j$, then we have that $k$ is not in $\red\cup\blue$ by Proposition \ref{prop:description of td-rigid pairs}(b3). In particular $P(k)$ is not a direct summand of $(M,P)$. This motivates the first observtation.

\begin{lemma}\label{lem:connection between notblue and notred}
Let $(M,P)$ be a basic $\td$-rigid pair with non-diagonal component $(\red,\blue)$. Let $i\in [2,p]$ be odd. If $M_i\neq 0$, then $\notred{i+1}\subsetneq \notblue{i}$ and $\notblue{i-1}\subsetneq \notred{i}$.  
\end{lemma}

\begin{proof}
We only show the strict inclusion $\notred{i+1}\subsetneq \notblue{i}$ as the proof for $\notblue{i-1}\subsetneq\notred{i}$ is similar. By (\ref{eq:the definition of notblue and notred}) we have that the interval $\notblue{i}$ is equal to $[s_i,s_i+n_{i}-1]$. Since $M_i\neq 0$, we have that $n_i>0$ and son in particular $\notblue{i}=[s_i,s_i+n_{i}-1]\neq \varnothing$. If $M_{i+1}=0$, then $\notred{i+1}=\varnothing$ and so $\notred{i+1}\subsetneq \notblue{i}$ holds. Assume that $M_{i+1}\neq 0$. Then (\ref{eq:the definition of notblue and notred}) gives $\notred{i+1}=[s_i,s_i+(l-l_{i+1})-1]$ and using Proposition \ref{prop:description of td-rigid pairs}(b1) and $l_i\leq n_i$ we have
\[
s_i+(l-l_{i+1}) - 1 \leq s_i+(l_i-1) - 1  < s_i+n_i-1,
\]
which shows that $\notred{i+1}\subsetneq \notblue{i}$.
\end{proof}

The second observation is related to the maximal number $k$ such that $M_{i},\ldots,M_{i+k}$ can all be nonzero. As we have mentioned in the introduction to this section, this is a crucial factor when investigating summand-maximal $\td$-rigid pairs. At this point we only deal with two special cases.

\begin{remark}\label{Remark: no three consecutive for l=3}
Let $(M,P)$ be a basic $\td$-rigid pair. Let $i\in [2,p]$. Assume that $M_{i}\neq 0$ and that $M_{i+1}\neq 0$. Then in the following cases we can show that $M_{i+2}= 0$.
\begin{enumerate}
\item[(a)] Assume first that $i$ is odd and $d=2$. Assume to a contradiction that $M_{i+2}\neq 0$. Then the number $m_{i+2}=\abs{M_{i+2}}$ is positive. Then we have
\begin{align*}
    m_{i+2} &\leq n_{i+2} - l_{i+2} + 1 &\text{(Lemma \ref{lem:immediate use of mi li ni})} \\
    &\leq (l-1-n_{i+1}) + (1-n_i) +1 &\text{(Proposition \ref{prop:description of td-rigid pairs}(b1))} \\
    &= l+1 -n_{i+1} -n_i &\text{(rearranging)} \\
    &\leq l_i+l_{i+1} -n_{i+1}-n_i &\text{(Proposition \ref{prop:description of td-rigid pairs}(b1))} \\
    &=(l_i-n_i) + (l_{i+1}-n_{i+1}) &\text{(rearranging)} \\ 
    &\leq 0 + 0 = 0, &\text{(Lemma \ref{lem:immediate use of mi li ni})},
\end{align*}
contradicting $m_{i+2}>0$. We conclude that $M_{i+2}=0$.

\item[(b)] Now assume that $l=3$. Let $i$ be even. Then by Proposition \ref{prop:description of td-rigid pairs}(b2) we have that $n_i+n_{i+1}\leq l-1=2$ and so we deduce that $n_i=l_i=1=n_{i+1}=l_{i+1}$. On the other hand, since $i+1$ is odd, by Proposition \ref{prop:description of td-rigid pairs}(b1) we obtain that 
\[
l+1\leq l_{i+1}+l_{i+2} = 1 +l_{i+2} \leq 1 +l
\]
from which we obtain that $l_{i+2}=l$ and so $M_{i+2}=0$. If $i$ is odd then a similar argument works.
\end{enumerate}
\end{remark}

Now we illustrate the notation we have introduced in the following example.

\begin{example}\label{example tau 4 rigid Lambda 23 4}
    Let us construct a basic $\tau_4$-rigid pair $({\color{rigid}M},{\color{support}P})$ for $\Lambda(23,4)$. The algebra $\Lambda(23,4)$ together with the $4$-cluster tilting subcategory $\cC$ is illustrated via the following picture:
    \[
    \includegraphics{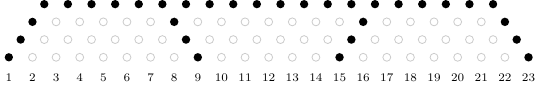}
    \]  
    We see that we have $4$ diagonals and so $p=4$. Moreover, we see that the simple modules in $\cC$ correspond to $s_1=1$, $s_2=9$, $s_3=15$, $s_4=23$. 
    
    For each of the diagonals $\diag{2}$, $\diag{3}$ and $\diag{4}$ we need to choose a module $M_2$, $M_3$ and $M_4$. We start by choosing, say, $\ind{s_3}{s_3+1}=\ind{15}{16}$ to be a summand of ${\color{rigid}M_3}$. This module has length $2$, and so $l_3\leq 2\leq n_3$ where $l_3$ respectively $n_3$ are the smallest and largest lengths of modules in $M_3$. By Proposition \ref{prop:description of td-rigid pairs}(b1) we obtain that $n_2\leq 1$ and $l_4\geq 3$. We encircle $\ind{15}{16}$ with a red dashed line and put red crosses to the modules which we cannot choose because of the restrictions $n_2\leq 1$ and $l_4\geq 3$.
    \[
    \includegraphics{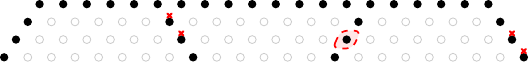}
    \]
    We can also observe that certain indecomposable projectives are excluded from ${\color{support}P}$ respectively ${\color{rigid}M}$, namely those whose top appears in the composition series of $\ind{15}{16}$ respectively $\tau_4(\ind{15}{16})=\ind{8}{9}$. These are precisely the modules that are encoded by $\notblue{3}$ and $\notred{3}$.
    \[
    \includegraphics{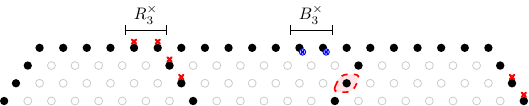}
    \]
    No further restrictions exist at this point and we may set ${\color{rigid}M_2}=\ind{8}{8}$ and ${\color{rigid}M_4}=\ind{21}{23}$. Then Proposition \ref{prop:description of td-rigid pairs}(b2) gives $n_3\leq 2\leq l_3$. We also introduce new restrictions encoded by $\notblue{2}$, $\notblue{4}$, $\notred{2}$ and $\notred{4}$: 
    \begin{equation*}
        \includegraphics{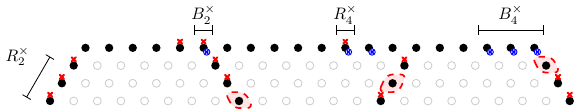}
    \end{equation*}
    As it can be seen from the above picture, we have now exhausted the amount of non-projective summands that we can include in ${\color{rigid}M}$, and are only left with including projective summands in ${\color{rigid}M}$ and ${\color{support}P}$. We may for example include every indecomposable projective without a red cross in ${\color{rigid}M}$
    \[
    \includegraphics{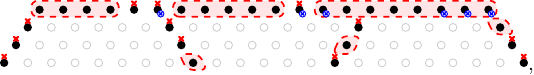}
    \]
    which gives $\abs{{\color{rigid}M}}+\abs{{\color{support}P}}=20<23=\abs{\Lambda}$. We may instead include every indecomposable projective without a blue cross in ${\color{support}P}$
    \[
    \includegraphics{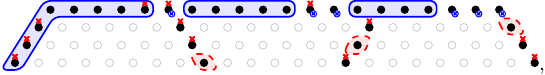}
    \]
    which also gives $\abs{{\color{rigid}M}}+\abs{{\color{support}P}}=20<23=\abs{\Lambda}$. Finally, we could include some indecomposable projective modules in ${\color{rigid}M}$ and some indecomposable projective modules in ${\color{support} P}$, for example
    \[
    \includegraphics{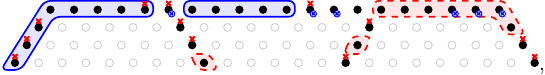}
    \]
    where we get more summands, but still shy of $\abs{\Lambda}$: $\abs{{\color{rigid}M}}+\abs{{\color{support}P}}=22<23=\abs{\Lambda}$. There are, of course, many different choices of indecomposable projective modules to include in ${\color{rigid} M}$ and ${\color{support}P}$, but, as it is shown in Proposition \ref{prop:to not lose need less than 2}, given that we have that $d>2$ and that ${\color{rigid} M_i}\neq 0$, ${\color{rigid} M_{i+1}}\neq 0$ and ${\color{rigid} M_{i+2}}\neq 0$ for some $i\in [2,p]$, any such choice gives $\abs{{\color{rigid}M}}+\abs{{\color{support}P}}<\abs{\Lambda}$. With this in mind, we encourage the reader to play around with this algebra and try to find a basic $\tau_4$-rigid pair with 23 summands, before venturing on. 
\end{example}

\begin{remark}
Let $A$ be any finite-dimensional algebra. If $d=1$, then there exists an analogue of \emph{Bongartz completion} \cite{BongCompl}. Namely, if $(M,P)$ is a basic $\tau$-rigid pair for $A$, then there exists a basic $\tau$-rigid pair $(M',P')$ with $\abs{M'}+\abs{P'}=\abs{A}$ such that $M$ is a direct summand of $M'$ and $P$ is a direct summand of $P'$, see \cite[Theorem 0.2]{adachi_-tilting_2014}. Example \ref{example tau 4 rigid Lambda 23 4} shows that the naive generalization of this result for $d>1$ fails.
\end{remark}

\subsection{Basic \texorpdfstring{$\td$}{td}-rigid pairs supported in consecutive diagonals}\label{subsec:basic td-rigid pairs supported in consecutive diagonals}

Let $(M,P)$ be a basic $\td$-rigid pair. As we have seen, the indecomposable summands of $M$ are either projective or they belong to $\diag{i}$ for some $i\in [2,p]$. It turns out that the existence of a sequence $M_{i},M_{i+1},\ldots,M_{i+k}$ where $M_{i+j}\neq 0$ for $j\in [0,k]$, can impose big restrictions on the rest of the modules appearing in $(M,P)$. The first such restriction is the following.

\begin{lemma}\label{lem:consecutive diagonals can have at most l-1}
Let $(M,P)$ be a basic $\td$-rigid pair and let $i\in [2,p]$. Assume that $M_{i+j}\neq 0$ for $0\leq j\leq k$. Then $\sum_{j=0}^k m_{i+j} \leq l-1$.    
\end{lemma}

\begin{proof}
Using Proposition \ref{prop:description of td-rigid pairs}(b1) and (b2) we can show the inequalities
\begin{equation*}
\sum_{j=0}^k n_{i+j} \leq \begin{cases}
    \left(\frac{k+3}{2}\right)(l-1), &\mbox{if $i$ is odd and $k$ is odd,} \\
    \left(\frac{k+1}{2}\right)(l-1), &\mbox{if $i$ is even and $k$ is odd,} \\
    \left(\frac{k+2}{2}\right)(l-1), &\mbox{if $k$ is even,}
\end{cases}    
\end{equation*}
and
\begin{equation*}
\sum_{j=0}^k l_{i+j} \geq \begin{cases}
    \left(\frac{k+1}{2}\right)(l+1), &\mbox{if $i$ 
 is odd and $k$ is odd,} \\
    \left(\frac{k-1}{2}\right)(l+1)+2, &\mbox{if $i$ is even and $k$ is odd,} \\
    \left(\frac{k}{2}\right)(l+1)+1, &\mbox{if $k$ is even.}
\end{cases}    
\end{equation*}
Let us demonstrate by showing the first inequality. We thus assume that $i$ and $k$ are both odd. Then using Proposition \ref{prop:description of td-rigid pairs}(b1) and (b2) we have
\begin{align*}
    \sum_{j=0}^{k}n_{i+j} &= n_i + \sum_{j=1}^{\frac{k-1}{2}}(n_{i+2j-1}+n_{i+2j}) +n_{i+k} \leq (l-1) + \sum_{j=1}^{\frac{k-1}{2}}(l-1) + (l-1)= \left(\frac{k+3}{2}\right)(l-1).
\end{align*}
Now using the above inequalities a direct computation shows that independently of the parity of $i$ and $k$ we have
\[
\sum_{j=0}^{k} n_{i+j} - \sum_{j=0}^{k}l_{i+j} \leq l-k-2.
\]
Using Lemma \ref{lem:immediate use of mi li ni} and the above we compute
\[
\sum_{j=0}^{k}m_{i+j} \leq \sum_{j=0}^{k}(n_{i+j}-l_{i+j}+1) = \sum_{j=0}^{k}(n_{i+j}-l_{i+j}) + \sum_{j=0}^{k}1 \leq (l-k-2) + (k+1) = l-1,
\]
as required.
\end{proof}

Let $(M,P)$ be a $\cC$-pair. Let $2\leq i\leq i+k \leq p$ and assume that $M_{i+j}\neq 0$ for $0\leq j\leq k$. Then these modules $M_{i+j}$ define intervals $\notred{i+j}$ and $\notblue{i+j}$ which affect which indecomposable projective modules can appear in $M$ and $P$, as we have seen. We are interested in the largest possible subinterval $\Xi$ of $[1,n]$ that contains all of these intervals $\notred{i+j}$ and $\notblue{i+j}$, as this encodes which indecomposable projective modules are allowed or not allowed to be summands of $M$ and $P$, given the sequence $M_{i},\ldots,M_{i+k}$. Thus we define
\begin{equation}\label{eq:the definition of Xi}
\Xi(i,i+k) = \begin{cases}
    [s_{i-1}-(l-l_i)+1, s_{i+k}], &\mbox{if $i$ is odd and $k$ is odd,} \\
    [s_{i-1}-(l-l_i)+1, s_{i+k}+n_{i+k}-1], &\mbox{if $i$ is odd and $k$ is even,} \\
    [s_{i-1}, s_{i+k}+n_{i+k}-1], &\mbox{if $i$ is even and $k$ is odd,} \\
    [s_{i-1},s_{i+k}], &\mbox{if $i$ is even and $k$ is even.}
\end{cases}
\end{equation}
That is $\Xi(i,i+k)$\index[symbols]{X@$\Xi(i,i+k)$}, is the interval starting at the start of $\notred{i}$ and ending at the end of $\notblue{i+k}$. We simply write $\Xi_i$ instead of $\Xi(i,i)$\index[symbols]{X@$\Xi(i)$}. In particular we have 
\[
\Xi_i = \begin{cases}
    [s_{i-1}-(l-l_i)+1,s_i+n_i-1], &\mbox{if $i$ is odd,} \\
    [s_{i-1},s_i], &\mbox{if $i$ is even,}
\end{cases}
\]
and a straightforward induction shows that $\Xi(i,i+k)=\bigcup_{j=0}^{k}\Xi_{i+j}$. We also define $Z_i=\Xi_i\setminus (\notred{i}\cup \notblue{i})$\index[symbols]{Z@$Z_i$}, that is
\begin{equation}\label{eq:definition of Z_i}
Z_i = \begin{cases}
    [s_{i-1}+1,s_i-1], &\mbox{if $i$ is odd,}\\
    [s_{i-1}+(l-l_i),s_i-n_i], &\mbox{if $i$ is even.}
\end{cases}
\end{equation}

\begin{example}
    In any of the $\tau_4$-rigid pairs $({\color{rigid}M},{\color{support}P})$ that appear in Example \ref{example tau 4 rigid Lambda 23 4}, we have
    \[
    \begin{split}
        \Xi_2&=[s_1,s_2]=[1,9],\\
        \Xi_3&=[s_2-(l-l_3)+1,s_3+n_3-1]=[8,16],\\
        \Xi_4&=[s_3,s_4]=[15,23],\\
        \Xi(2,4)&=[s_1,s_4]=[1,23]=\Xi_2\cup \Xi_3\cup \Xi_4.
    \end{split}
    \]
\end{example}

The main result of this paragraph counts the difference between the size of the set $\Xi(i,i+k)$ and the sum of the number of indecomposable summands of the modules $M_i,\ldots,M_{i+k}$ together with the number of indecomposable projective modules indexed in $\Xi(i,i+k)$. It shows that in most cases $\Xi(i,i+k)$ has strictly larger size. As we will see later, this is an obstruction to having a summand-maximal $\td$-rigid pair.

\begin{proposition}\label{prop:to not lose need less than 2}
Let $(M,P)$ be a basic $\td$-rigid pair with non-diagonal component $(\red,\blue)$. Let $2\leq i\leq i+k \leq p$ and assume that $M_{i+j}\neq 0$ for $0\leq j\leq k$. Set $\Xi=\Xi(i,i+k)$. 
\begin{enumerate}
    \item[(a)] If $d=2$ and $i$ is odd, then $k\leq 1$.
    \item[(b)] If $d=2$ and $i$ is even, then $k\leq 2$.
    \item[(c)] If $d>2$ and $k\geq 2$, then
    \begin{equation}\label{eq:bound on loss k large}
        \abs{(\red\cup \blue)\cap \Xi} + \sum_{j=0}^{k}m_{i+j} < \abs{\Xi}.
    \end{equation}
\end{enumerate}
\end{proposition}

\begin{proof}
Part (a) is Remark \ref{Remark: no three consecutive for l=3}(a). Part (b) follows immediately from part (a). It remains to show part (c). Thus we assume that $d>2$ and $k\geq 2$. Then we have that $M_i\neq 0$, $M_{i+1}\neq 0$ and $M_{i+2}\neq 0$. By Lemma \ref{lem:consecutive diagonals can have at most l-1} we obtain that $\sum_{j=0}^{k}m_{i+j}\leq l-1$, that is the number of nonisomorphic indecomposable summands appearing in all of $M_i,\ldots,M_{i+k}$ is at most equal to $l-1$. Hence it is enough to show that $\abs{(\red\cup \blue)\cap \Xi}\leq \Xi-l$. Therefore, it is enough to show that there exists a subset $X\subseteq \Xi$ such that 
\begin{enumerate}
    \item[(i)] $(\red\cup \blue)\cap X=\varnothing$, and
    \item[(ii)] $\abs{X} \geq l$.
\end{enumerate}
We consider the cases $i$ odd and $i$ even separately and in each case we construct such a set $X$.
    
Assume first that $i$ is odd. Set $X=\notred{i+1}\cup \notblue{i+1}$. Clearly $X\subseteq \Xi$. By Lemma \ref{lem:connection between notblue and notred} we have that
\[
\notred{i+1} \subsetneq \notblue{i} \text{ and } \notblue{i+1}\subsetneq \notred{i+2}.
\]
By Proposition \ref{prop:description of td-rigid pairs}(b3) we have that $\blue\cap \notblue{j}=\varnothing$ and $\red\cap\notred{j}=\varnothing$ for all $j\in [2,p]$. We conclude that 
\[
(\red\cup \blue)\cap \notred{i+1} = \varnothing \text{ and } (\red\cup \blue)\cap \notblue{i+1}=\varnothing,
\]
and so condition (i) holds. Moreover, by definition of the intervals $\notred{i+1}$ and $\notblue{i+1}$ we have that $\notred{i+1}\cap \notblue{i+1}=\varnothing$. Using Lemma \ref{lem:immediate use of mi li ni} we obtain
\[
\abs{X} =\abs{\notred{i+1}} + \abs{\notblue{i+1}} = l-l_{i+1} + n_{i+1} = l+(n_{i+1}-l_{i+1}) \geq l,
\]
and so condition (ii) holds too. This shows the claim when $i$ is odd.

Assume now that $i$ is even. Let $X_1=Z_{i+1}$. Using the definition of $Z_{i+1}$ in (\ref{eq:definition of Z_i}) and since $i$ is even, we obtain
\begin{equation}\label{eq:where we use d>2}
\abs{X_1} = s_{i+1}-1 - (s_i+1)+1 = \frac{d-2}{2}l+1 \geq l+1,
\end{equation}
where we used (\ref{eq:difference of consecutive simples}) and $d>2$. Hence $X_1$ satisfies condition (ii). We may thus assume that $X_1$ does not satisfy condition (i), otherwise there is nothing to prove. We then obtain that $(\red\cup \blue)\cap Z_{i+1}\neq \varnothing$. We consider the subcases $\red\cap Z_{i+1}\neq\varnothing$ and $\blue\cap Z_{i+1}\neq\varnothing$ separately.

\textbf{Subcase 1}: Assume that $\red\cap Z_{i+1}\neq \varnothing$. Then there exists $y\in \red\cap Z_{i+1}$. Now let $X_2=\notred{i}\cup \notred{i+1}\subseteq \Xi$. By definition of the intervals $\notred{i}$ and $\notred{i+1}$ we have that $\notred{i}\cap \notred{i+1}=\varnothing$. Using Lemma \ref{lem:immediate use of mi li ni} and Proposition \ref{prop:description of td-rigid pairs}(b2) we obtain
\[
\abs{X_2} = \abs{\notred{i}}+\abs{\notred{i+1}} = (l-l_i) + (l-l_{i+1}) \geq 2l -(n_i+n_{i+1}) \geq 2l-(l-1) = l+1.
\]
Hence $X_2$ also satisfies condition (ii). Again we may assume that $X_2$ does not satisfy condition (i). Since by Proposition \ref{prop:description of td-rigid pairs}(b3) we have that $\red\cap\notred{i}=\varnothing$ and $\red\cap\notred{i+1}=\varnothing$, the negation of condition (i) for $X_2$ gives that $\blue\cap (\notred{i}\cup \notred{i+1})\neq \varnothing$. Thus there exists $x\in \blue\cap (\notred{i}\cup \notred{i+1})$. By the definition of the intervals $\notred{i}$, $\notred{i+1}$ and $Z_{i+1}$ we have that 
\[
x\leq \max{(\notred{i}\cup \notred{i+1})} =  s_{i} < s_{i} +1 =\min{Z_{i+1}} \leq y. 
\]
Since $x<y$ and $x\in\blue$ and $y\in\red$ all hold, by Proposition \ref{prop:description of td-rigid pairs}(a) we conclude that there exists $x<z<\cdots <z+l-2<y$ such that for $Z=[z,z+l-2]$ we have $(\red\cup \blue)\cap Z=\varnothing$. Set $X = Z\cup \notred{i+2}$. By Lemma \ref{lem:connection between notblue and notred} we have that $\notred{i+2}\subsetneq \notblue{i+1}$. By Proposition \ref{prop:description of td-rigid pairs}(b3) we conclude that $\notred{i+2}\cap (\red\cup \blue)=\varnothing$. Hence $X$ satisfies condition (i). Moreover, we have that $Z\cap \notred{i+2}=\varnothing$ since $\max{Z}<y\leq\max{Z_{i+1}}<\min{\notred{i+2}}$. Then
\[
\abs{X} = \abs{[z,z+l-2]} + \abs{\notred{i+2}} =l-1+\abs{\notred{i+2}} \geq l,
\]
and so $X$ satisfies condition (ii) as well.

\textbf{Subcase 2}: Assume that $\blue\cap Z_{i+1}\neq \varnothing$. Then there exists $x'\in \blue\cap Z_{i+1}$. An argument similar to the previous subcase shows that $X_3=\notblue{i+2}\cup \notblue{i+1}$ satisfies condition (ii) and thus we may assume that $\red\cap(\notblue{i+1}\cup \notblue{i+2})\neq \varnothing$. Then there exists $y'\in \red\cap (\notblue{i+1}\cup \notblue{i+2})$. Applying Proposition \ref{prop:description of td-rigid pairs}(a) to $x'<y'$ similarly as above gives a set $Z'$ with $(\red\cup \blue)\cap Z'=\varnothing$ and $\abs{Z'}=l-1$. Setting $X'=Z'\cup \notblue{i}$ one can then show similarly that $X'$ satisfies conditions (i) and (ii). 
\end{proof}

\subsection{Admissible configurations}\label{subsec:admissible configurations}

We now want to describe summand-maximal $\td$-rigid pairs for $\Lambda$. Thus let $(M,P)$ be a summand-maximal $\td$-rigid pair. Assume that we have a sequence $M_{i},M_{i+1},\ldots,M_{i+k}$ such that $M_{i+j}$ is nonzero for all $0\leq j\leq k$. If $d=2$, then Proposition \ref{prop:to not lose need less than 2}(a) and (b) bound $k$ by $1$ if $i$ is odd and by $2$ if $i$ is even. If $d>2$, then we show later that Proposition \ref{prop:to not lose need less than 2}(c) implies that $k\leq 1$. This indicates that $M$ can intersect a very small number of consecutive diagonals $\diag{i},\ldots,\diag{i+k}$, namely at most $2$ if $d>2$ and at most $3$ in the very special case when $d=2$ and $i$ is even. This is a first reduction of our problem of classifying summand-maximal $\td$-rigid pairs, as we now need to only study configurations satisfying
\begin{align*}
&M_{i-1}=0,\;\; M_i\neq 0,\;\;\text{and } M_{i+1} =0,\text{ or} \\
&M_{i-1}=0,\;\; M_i\neq 0,\;\; M_{i+1}\neq 0,\;\;\text{and } M_{i+2}=0,\text{ or}\\
&M_{i-1}=0,\;\; M_i\neq 0,\;\; M_{i+1}\neq 0,\;\; M_{i+2}\neq 0,\;\; M_{i+3}=0,\;\; d=2,\;\;\text{and } \text{$i$ is even.}
\end{align*}
We can find certain restrictions on how these modules look like and these restrictions are encoded in Table \ref{tab:illustration admissible configurations} in an illustrative way and in Definition \ref{def:admissible configurations} in a formal way.

In addition to the modules on the diagonals, we can find a few more necessary conditions that need to be satisfied in such local configurations if $(M,P)$ is a summand-maximal $\td$-rigid pair. The crucial part of these conditions pertain to the interval $\Xi=\Xi(i,i+k)$ and impose restrictions on which indecomposable projectives $P(\xi)$ with $\xi\in \Xi$ can be direct summands of $M$ and of $P$. Since $P(\red)$ is the direct sum of all indecomposable projective summands of $M$ and $P(\blue)=P$, we can reformulate these conditions using the intersections $\red\cap \Xi$ and $\blue\cap \Xi$. The following definition is motivated by the need to describe these conditions.

\begin{definition}\label{def:support, rigid, support to rigid}
Let $(M,P)$ be a $\cC$-pair with non-diagonal component $(\red,\blue)$. Let $I\subseteq [1,n]$. We say that $I$ is
\begin{itemize}
    \item \emph{support}\index[definitions]{interval!support} if $\red\cap I =\varnothing$ and $\blue\cap I=I\setminus\left(\bigcup_{i=2}^{p} \notblue{i}\right)$,
    \item \emph{rigid}\index[definitions]{interval!rigid} if $\blue\cap I=\varnothing$ and $\red\cap I = I\setminus \left(\bigcup_{i=2}^{p} \notred{i}\right)$,
    \item \emph{support to rigid at $x\in I$}\index[definitions]{interval!support to rigid} if $\blue\cap I = \{y\in I \mid y\leq x\}\neq \varnothing$ and $\red\cap I = \{y\in I\mid x+l\leq y\}\neq \varnothing$,
    \item \emph{rigid to support at $x\in I$}\index[definitions]{interval!rigid to support} if $\red \cap I = \{y\in I\mid y\leq x\}\neq \varnothing$ and $\blue\cap I = \{y\in I\mid x+1\leq y\}\neq\varnothing$.
\end{itemize}
\end{definition}

Note that in Definition \ref{def:support, rigid, support to rigid} the properties of the interval $I$ are defined with respect to the given $\cC$-pair, which is suppressed in the naming. Note also that an empty interval is both rigid and support, and that is the only interval with this property. With this in mind we can now describe all configurations that are possible for a summand-maximal $\td$-rigid pair, as we prove later.

\begin{table}
        \centering
        \begin{tabular}{cc|cc}
            (I) - $i$ odd & \includegraphics[scale=.8]{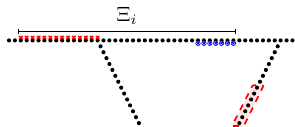} & (I) - $i$ even& \includegraphics[scale=.8]{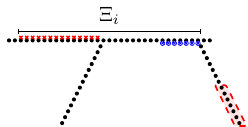}\\ \hline
            (II) - $i$ odd & \includegraphics[scale=.8]{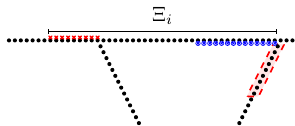} & (II) - $i$ even & \includegraphics[scale=.8]{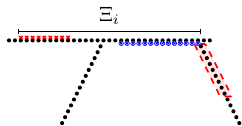} \\ \hline
            (III) - $i$ odd & \includegraphics[scale=.8]{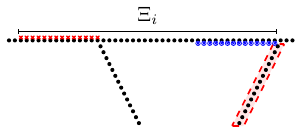} & (III) - $i$ even & \includegraphics[scale=.8]{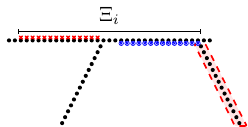} \\\hline
            (IV) & \multicolumn{2}{c}{\includegraphics[scale=.8]{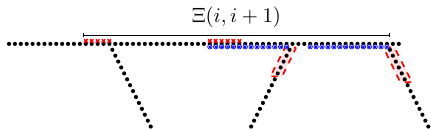}}&$m_i+m_{i+1}<l-1$\\\hline
            (V)& \multicolumn{2}{c}{\includegraphics[scale=.8]{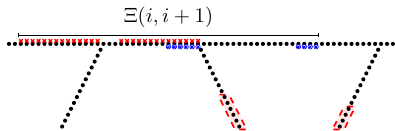}}&$m_i+m_{i+1}<l-1$\\\hline
            (VI)&\multicolumn{2}{c}{\includegraphics[scale=.8]{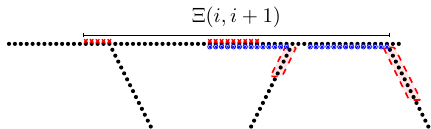}}&$m_i+m_{i+1}=l-1$\\\hline
            (VII)&\multicolumn{2}{c}{\includegraphics[scale=.8]{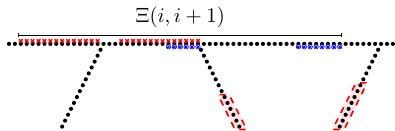}}&$m_i+m_{i+1}=l-1$\\\hline
            (VIII) &\multicolumn{2}{c}{\includegraphics[scale=.9]{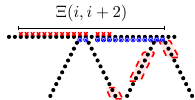}}&$\begin{array}{c}
                l_{i+1}=l-l_{i+2}+1    \\
                n_{i+1}=l-n_i-1\\
                m_i+m_{i+1}+m_{i+2}=l-1
            \end{array}$
            
        \end{tabular}
        \caption{Types of admissible configurations}
        \label{tab:illustration admissible configurations}
\end{table}

\begin{definition}\label{def:admissible configurations}
Let $(M,P)$ be a $\cC$-pair with non-diagonal component $(\red,\blue)$. Let $2\leq i\leq i+k \leq p$ and assume that $M_{i+j}\neq 0$ for $0\leq j\leq k$. Set $\Xi=\Xi(i,i+k)$. We say that $(M_i,\ldots,M_{i+k})$\index[symbols]{M@$(M_i,\ldots,M_{i+k})$} is an \emph{admissible configuration}\index[definitions]{admissible configuration} if one of the following holds.
\begin{enumerate}
    \item[(I)] $k=0$, $l_i=1$, $n_i<l-1$ and $\Xi$ is support.
    \item[(II)] $k=0$, $l_i>1$, $n_i=l-1$ and $\Xi$ is rigid.
    \item[(III)] $k=0$, $l_i=1$, $n_i=l-1$ and $\Xi$ is either support, rigid, or support to rigid at some $x\in\Xi\setminus \notblue{i}$.
    \item[(IV)] $k=1$, $i$ is odd, $n_i=n_{i+1}=l-1$, $m_i+m_{i+1}<l-1$, and $\Xi$ is rigid.
    \item[(V)] $k=1$, $i$ is even, $l_i=l_{i+1}=1$, $m_i+m_{i+1}<l-1$, and $\Xi$ is support.
    \item[(VI)] $k=1$, $i$ is odd, $n_i=n_{i+1}=l-1$, $m_i+m_{i+1} = l-1$ and $\Xi$ is rigid or support to rigid at $x\in [s_{i}-(l-m_{i+1}),s_{i}-1]$.
    \item[(VII)] $k=1$, $i$ is even, $l_i=l_{i+1}=1$, $m_i+m_{i+1}=l-1$, and $\Xi$ is support or support to rigid at $x\in [s_i-(l-1),s_i-m_i]$.
    \item[(VIII)] $k=2$, $d=2$, $i$ is even, $l_i=1$, $l_{i+1}=l-l_{i+2}+1$, $n_{i+1}=l-n_i-1$, $n_{i+2}=l-1$ and $\Xi$ is support to rigid at $x\in [s_{i+1}-l_{i+2},s_i-n_i]$.
\end{enumerate}
In each case (I)---(VIII) we call the corresponding roman numeral the \emph{type} of the admissible configuration. We say that an admissible configuration is \emph{full}\index[definitions]{admissible configuration!full} if $m_{i+j}=n_{i+j}-l_{i+j}+1$ holds for $0\leq j\leq k$.
\end{definition}

Before we give some examples, we collect some useful observations about admissible configurations.

\begin{remark}\label{rem:admissible configurations are td-rigid}
Let $(M,P)$ be a $\cC$-pair with non-diagonal component $(\red,\blue)$ and let $(M_i,\ldots,M_{i+k})$ be an admissible configuration. Let $\Xi=\Xi(i,i+k)$.
\begin{enumerate}
    \item[(a)] If $l=2$ then $(M_i,\ldots,M_{i+k})$ is necessarily of type (III).
    \item[(b)] An admissible configuration has the property of being ``locally'' $\td$-rigid. Indeed, assume that $M_{i-1}=0$ and $M_{i+k+1}=0$. Set $\red'=\red\cap\Xi$ and $\blue'=\blue\cap\Xi$. Let $i'\in [i,i+k]$. It is straightforward to check that independently of the type of the admissible configuration, conditions (a) and (b3) in Proposition \ref{prop:description of td-rigid pairs} hold with $\red'$ and $\blue'$ in place of $\red$ and $\blue$ and conditions (b1) and (b2) in Proposition \ref{prop:description of td-rigid pairs} hold with $i'$ in place of $i$.
    \item[(c)] Assume that the admissible configuration is full. We claim that
    \begin{equation}\label{eq:admissible configurations do not lose 1}
        \abs{(\red\cup\blue)\cap\Xi} +\sum_{j=0}^{k}m_i  = \abs{\Xi}.
    \end{equation}
    Indeed, we have that $\Xi$ is either support, rigid or support to rigid. Assume $\Xi$ is support. Then $(M_i,\ldots,M_{i+k})$ is of one of the types (I), (III), (V) or (VII). In all of these cases $l_{i+t}=1$ for $0\leq t\leq k$ and so $m_{i+t}=n_{i+t}$. Additionally, in cases (V) and (VII) it is easy to see that $\notblue{i}\cap\notblue{i+1}=\varnothing$. Hence, since $\Xi$ is support, we compute
    \[ 
    \abs{(\red\cup\blue)\cap\Xi} = \abs{\Xi}-\sum_{t=0}^{k} \abs{\notblue{i+t}}  = \abs{\Xi}-\sum_{t=0}^{q} n_{i+t} = \abs{\Xi}-\sum_{t=0}^{q} m_{i+t},
    \]
    which shows (\ref{eq:admissible configurations do not lose 1}). If $\Xi$ is rigid, then the proof is similar. If $\Xi$ is support to rigid, then $\abs{(\red\cup\blue)\cap\Xi}=\abs{\Xi}-(l-1)$ and $(M_i,\ldots,M_{i+k})$ is of one of the types (III), (VI), (VII) or (VIII). We claim that in all of these types we have $\sum_{t=0}^{k} m_{i+t}=l-1$, and so (\ref{eq:admissible configurations do not lose 1}) holds again. Indeed, it is only needed to show that $\sum_{t=0}^{k} m_{i+t}=l-1$ holds in type (VIII), as in the other cases it holds by definition. In that case we have $k=2$ and
    \[
    \sum_{t=0}^{2}m_{i+t}= \sum_{t=0}^{2}(n_{i+t}-l_{i+t}+1)=(n_i-1+1) + (l-n_i-1-(l-l_{i+2}+1)+1) +(l-1-l_{i+2}+1) = l-1,
    \]
    as required.    
    \item[(d)] Assume that the type of the configuration is (III). If $\Xi$ is support, then $\Xi$ is the disjoint union of two consecutive intervals $\blue\cap\Xi$ and $\notblue{i}$. If $\Xi$ is rigid, then $\Xi$ is the disjoint union of two consecutive intervals $\notred{i}$ and $\red\cap\Xi$. Furthermore, both intervals $\notblue{i}$ and $\notred{i}$ have length $l-1$.
\end{enumerate}
\end{remark}

\begin{example}\label{Example:Admissible configuration}
    Let us assume that we have a $\tau_4$-rigid module ${\color{rigid}M}$ of $\Lambda(23,4)$ which we can see encircled in the following picture:
    \[
    \includegraphics{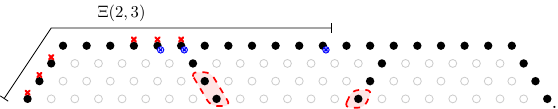}
    \]
    In the above picture we have also marked the sets $\notred{2}\cup\notred{3}$ with red crosses, the set $\notblue{2}\cup\notblue{3}$ with blue crosses and $\Xi(2,3)$. Assume that we want to find a projective module $\color{support} P$ such that $({\color{rigid}M},{\color{support}P})$ is a $\tau_4$-rigid pair and $({\color{rigid}M_2},{\color{rigid}M_3})$ is a full admissible configuration of type (VII). Then there are several choices for $\Xi(2,3)$. For example, $\Xi(2,3)$ can be support:
    \[
    \includegraphics{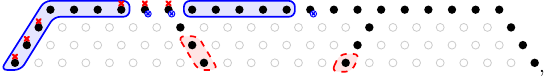}
    \]
    or, $\Xi(2,3)$ can be support to rigid at $7\in[s_2-(l-1),s_2-m_2]=[6,7]$:
    \[
    \includegraphics{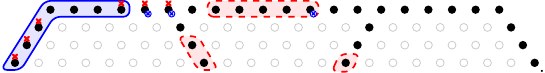}
    \]
    In each of these case we can construct $({\color{rigid}M},{\color{support}P})$ such that $\abs{{\color{rigid}M}}+\abs{{\color{support}P}}=\abs{\Lambda}=23$. In the first case this is done by letting $[1,23]\setminus\Xi(2,3)=[16,23]$ be support, and in the other case by letting $[16,23]$ be either rigid, rigid to support, or support.
\end{example}

\begin{example}\label{ex:Admissible type VIII}
    Now, let us look at the slightly special case of $d=2$. We consider the algebra $\Lambda(13,5)$ which admit a $2$-cluster tilting subcategory, and the $\tau_2$-rigid module ${\color{rigid}M}$ given by the following encircled summands:
    \[
    \includegraphics{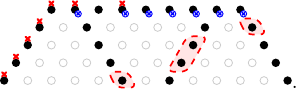}
    \]
    Now, if we want to expand $({\color{rigid}M_2},{\color{rigid}M_3},{\color{rigid}M_4})$ into an admissible configuration, it has to be of type (VIII). Hence $\Xi(2,4)=[1,13]$ needs to be support to rigid at some $x\in[s_3-4,s_2-1]=[4,5]=\mathcal{I}$. For example, the following expansion is a full admissible configuration of type (VIII)
    \[
    \includegraphics{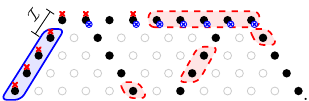}
    \]
    After counting, we see that $\abs{{\color{rigid}M}}+\abs{{\color{support}P}}=13=\abs{\Lambda}$.
\end{example}

Let $(M,P)$ be a basic $\td$-rigid pair with non-diagonal component $(\red,\blue)$. Let $2\leq i\leq i+k\leq p$ and assume that $M_{i+j}\neq 0$ for $0\leq j\leq k$. Set $\Xi=\Xi(i,i+k)$. As we have mentioned, we are now interested in studying $(M,P)$ locally in the ``neighbourhood'' of the sequence $M_{i},\ldots,M_{i+k}$. A way to define this neighbourhood is to think of all the indecomposable $\Lambda$-modules which are directly affected from $M_{i},\ldots,M_{i+k}$ with respect to being or not direct summands of $(M,P)$. As we have seen, the interval $\Xi$ is the largest subinterval of $[1,n]$ which contains all of $\notred{i+j}$ and $\notblue{i+j}$ for $0\leq j \leq k$ and thus is a good candidate for such a neighbourhood. Thus we can count how many $\xi\in \Xi$ belong to $\red$, how many $\xi\in\Xi$ belong to $\blue$ and also count the indecomposable summands of each $M_{i+j}$ which we have agreed to denote by $m_{i+j}$. Thus the number of indecomposable summands of $(M,P)$ in this neighbourhood of $M_{i},\ldots,M_{i+k}$ is
\begin{equation}\label{Eq:the number N}
N\coloneqq (\red\cup\blue)\cap\Xi +\sum_{j=0}^{k}m_{i+j}.
\end{equation}
We have seen in Proposition \ref{prop:to not lose need less than 2} that if $d>2$ and $k\geq 2$, then $N(i,i+k)$ cannot be greater than $\abs{\Xi}-1$. The main aim of this paragraph is to show that the greatest possible value for $N$ is $\abs{\Xi}$, and moreover that this value is attained if and only if $(M_i,\ldots,M_{i+k})$ is a full admissible configuration.

We should mention that at this point it is not clear why maximizing each of these numbers $N$ does maximize the total number of indecomposable summands of $(M,P)$ which is our aim. But the fact that the maximal possible value attained by $N$ is $\abs{\Xi}$ will be enough to show that indeed such a maximization has to happen. The details for this point belong to the next section.

In counting the number $N$ from above, we may notice that the intersection $(\red\cup\blue)\cap\Xi$ is closely related to the intervals $\notred{i}$ and $\notblue{i}$. Hence we first investigate how the intervals $\notred{i}$ and $\notred{i+1}$ interact when $M_i\neq 0$ and $M_{i+1}\neq 0$ (and similarly for $\notblue{i}$ and $\notblue{i+1}$).

\begin{lemma}\label{lem:the case k=1, i even and d=2}
Let $(M,P)$ be a $\td$-rigid pair with non-diagonal component $(\red,\blue)$. Let $2\leq i\leq i+1\leq p$ be such that $M_{i}\neq 0$ and $M_{i+1}\neq 0$.
\begin{enumerate}
    \item[(a)] If $i$ is even and $d=2$, then $\notred{i}\cup\notred{i+1}=[s_{i-1},s_{i}]$. Otherwise, we have $\notred{i}\cap\notred{i+1}=\varnothing$.
    \item[(b)] If $i$ is odd and $d=2$, then $\notblue{i}\cup\notblue{i+1}=[s_i,s_{i+1}]$. Otherwise, we have $\notblue{i}\cap\notblue{i+1}=\varnothing$.
\end{enumerate}
\end{lemma}

\begin{proof}
We only show part (a) as part (b) is similar. If $i$ is odd, then by (\ref{eq:the definition of notblue and notred}) we have that $\notred{i}=[s_{i-1}-(l-l_i)+1,s_{i-1}]$ and $\notred{i+1}=[s_i,s_{i}+(l-l_{i})-1]$ and so $\notred{i}\cap\notred{i+1=\varnothing}$ since $s_{i-1}<s_{i}$. 

Assume that $i$ is even. By (\ref{eq:the definition of notblue and notred}) we have that $\notred{i}=[s_{i-1},s_{i-1}+(l-l_i)-1]$ and $\notred{i+1}=[s_i-(l-l_{i+1})+1,s_i]$. If $d>2$, then using $s_i-s_{i-1}=\tfrac{d}{2}l$ by (\ref{eq:difference of consecutive simples}) and $1\leq l_i$ by Lemma \ref{lem:immediate use of mi li ni}, we obtain that
\[
s_{i-1}+(l-l_i)-1 = s_i-\frac{d}{2}l+(l-l_i)-1 < s_i-l+(l-l_i)-1 = s_i-l_i-1 \leq s_i-2,
\]
and so $\notred{i}\cap\notred{i+1}=\varnothing$. If $d=2$, then it is enough to show that 
\[
s_{i-1} \leq s_{i} - (l-l_{i+1}) +1 \leq \left(s_{i-1}+(l-l_i)-1\right)+1 \leq s_{i}.
\]
Equivalently, and using $s_i-s_{i-1}=l$ by (\ref{eq:difference of consecutive simples}), we need to show that
\[
0 \leq l_{i+1}+1\leq l-l_{i} \leq l.
\]
Thus only the middle inequality needs to be shown. But this inequality is equivalent to $l_{i}+l_{i+1}\leq l-1$, which holds since $l_{i}+l_{i+1}\leq n_{i}+n_{i+1}\leq l-1$ by Proposition \ref{prop:description of td-rigid pairs}(b2).
\end{proof}

In the definition of admissible configurations we see that the interval $\Xi$ is either support, rigid, or support to rigid. Assume that $\Xi$ is rigid; similar considerations hold when $\Xi$ is support. The interval $\Xi=\Xi(i,i+k)$ contains all of the intervals $\notred{i},\ldots,\notred{i+k}$ and exactly those as long as intervals of the form $\notred{j}$ are concerned. Therefore, $\Xi$ being rigid implies that $(\red\cup\blue)\cap\Xi$ is equal to $\Xi$ from which we have removed all of the intervals $\notred{i},\ldots,\notred{i+k}$. Then the number $N$ in (\ref{Eq:the number N}) is equal to 
\[
\abs{\Xi}-\abs{\bigcup_{j=0}^{k}\notred{i+j}}+\sum_{j=0}^{k}m_{i+j}.
\]
Thus to find the maximal value of $N$ we need to consider the difference between the last two terms above. The following technical lemma deals with this case.

\begin{lemma}\label{lem:the intervals X}
Let $(M,P)$ be a $\td$-rigid pair with non-diagonal component $(\red,\blue)$. Let $k\in\{0,1\}$ be such that $2\leq i\leq i+k \leq p$ and assume that $M_{i+j}\neq 0$ for $0\leq j\leq k$. Set $\Xi=\Xi(i,i+k)$, $\notred{}(i,i+k)=\bigcup_{j=0}^{k}\notred{i+j}$, and $\notblue{}(i,i+k)=\bigcup_{j=0}^{k}\notblue{i+j}$. Then the following hold.
\begin{enumerate}
    \item[(a)] $\abs{\notred{}(i,i+k)}\geq \sum_{j=0}^{k}m_{i+j}$. Moreover, if $k=1$ and $i$ is even, then the inequality is strict.
    \item[(b)] $\abs{\notblue{}(i,i+k)}\geq \sum_{j=0}^{k}m_{i+j}$. Moreover, if $k=1$ and $i$ is odd, then the inequality is strict.
    \item[(c)] Let $X=\notred{}(i,i+k)$ or $X=\notblue{}(i,i+k)$. If
    \begin{equation}\label{eq:condition for maximal interval}
    (\red\cup \blue)\cap \Xi=\Xi\setminus X \text{ and } \abs{X}=\sum_{j=0}^{k}m_{i+j}
    \end{equation}
    hold, then $(M_i,\ldots,M_{i+k})$ is a full admissible configuration.
\end{enumerate}
\end{lemma}

\begin{proof}
\begin{enumerate}
    \item[(a)] If $k=1$ and $i$ is even and $d=2$, then using Lemma \ref{lem:the case k=1, i even and d=2}(a) and Proposition \ref{prop:description of td-rigid pairs}(b2) we have
    \[
    \abs{\notred{}(i,i+1)} = s_i-s_{i-1} = l > n_{i}+n_{i+1} \geq m_{i}+m_{i+1},
    \]
    as required. Otherwise $\notred{i}\cap\notred{i+1}=\varnothing$ holds by Lemma \ref{lem:the case k=1, i even and d=2}(a) and so
    \[
    \abs{\notred{}(i,i+k)} = \sum_{j=0}^{k}\abs{\notred{i+j}} = \sum_{j=0}^{k}(l-l_{i+j}) \geq \sum_{j=0}^{k}(n_{i+j}+1-l_{i+j}) = \sum_{j=0}^{k}m_{i+j},
    \]
    as required. Moreover, if $k=1$ and $i$ is even, then we have $l\geq n_i+n_{i+1}+1$ by Proposition \ref{prop:description of td-rigid pairs}(b2). Using this we obtain
    \begin{align*}
    \abs{\notred{}(i,i+k)} &= \abs{\notred{i}}+\abs{\notred{i+1}} = l-l_i+l-l_{i+1}\geq 2(n_i+n_{i+1}+1)-l_i-l_{i+1} \\
    &= m_{i}+m_{i+1}+n_{i}+n_{i+1}>m_{i}+m_{i+1},
    \end{align*}
    and so the inequality is strict.
    \item[(b)] Similar to part (a).
    \item[(c)] We only show the claim when $X=\notred{}(i,i+k)$ as the proof when $X=\notblue{}(i,i+k)$ is similar. By definition of the intervals we have that $s_{i+k}\in\notblue{i+k}$ and $s_{i+k}\not\in\notred{}(i,i+k)$, see (\ref{eq:the definition of notblue and notred}). By  the first part of (\ref{eq:condition for maximal interval}) we have
    \[
    s_{i+k}\in \Xi\setminus \notred{}(i,i+k) = (\red\cup\blue)\cap \Xi=(\red\cap\Xi)\cup (\blue\cap\Xi).
    \]
    Since $s_{i+k}\in\notblue{i+k}$ and $\blue\cap\notblue{i+k}=\varnothing$, we conclude that $s_{i+k}\in \red\cap\Xi$. In particular, we have that $s_{i+k}\in\red\cap\notblue{i+k}$. Set $y=s_{i+k}$.

    We claim that $\blue\cap(\Xi\setminus X)=\varnothing$. Indeed, assume to a contradiction that there exists $x\in\blue\cap(\Xi\setminus X)$. In particular $x\not\in\notblue{i+k}$ and $x\not\in\notred{i}$. Recall that $\Xi$ is the interval starting at $\notred{i}$ and ending at $\notblue{i+k}$. Since $y\in\notblue{i+k}$, we obtain that $x<y$. Then by Proposition \ref{prop:description of td-rigid pairs}(a) there exists an interval $Z=[z,z+l-2]$ with $x<z<z+l-2<y$ such that $(\red\cup\blue)\cap Z=\varnothing$. It follows that
    \[
    (\red\cup\blue) \cap \Xi = (\red\cup\blue)\cap(\Xi\setminus Z) \subseteq \Xi\setminus Z.
    \]
    By the first part of (\ref{eq:condition for maximal interval}) we conclude that $\Xi\setminus X\subseteq \Xi\setminus Z$ or equivalently that $Z\subseteq X=\notred{i}\cup\notred{i+k}$. Since $x<z$ and $x\not\in\notred{i}$, we have that $\notred{i}$ and $Z$ are disjoint (as both intervals $\notred{i}$ and $Z$ lie inside $\Xi$, and $\Xi$ starts at $\notred{i}$). Hence $k=1$ and $Z\subseteq \notred{i+1}$. Then the second part of (\ref{eq:condition for maximal interval}) together with part (a) of this lemma imply that $i$ is odd. Hence Lemma \ref{lem:connection between notblue and notred} implies that
    \[
    Z\subseteq \notred{i+1} \subsetneq \notblue{i}
    \]
    and so $l-1 =\abs{Z} < \abs{\notblue{i}} = n_i\leq l-1$, which is a contradiction. Therefore, our claim that $\blue\cap(\Xi\setminus X)=\varnothing$ holds. 

    Now notice that by the first part of
    (\ref{eq:condition for maximal interval}) we have
    \[
    \varnothing=\blue\cap (\Xi\setminus X) = \blue\cap ((\red\cup\blue)\cap \Xi) = (\blue\cap (\red\cup\blue))\cap\Xi = \blue\cap\Xi.
    \]
    Hence $\blue\cap\Xi=\varnothing$ and so the first part of (\ref{eq:condition for maximal interval}) gives that $\Xi$ is rigid. By the second part of (\ref{eq:condition for maximal interval}) we have that $\abs{\notred{}(i,i+k)}=\sum_{j=0}^{k}m_{i+j}$ holds. By part (a) of this lemma we obtain that $k=0$ or $i$ is even. In any case, Lemma \ref{lem:the case k=1, i even and d=2} gives that $\abs{\notred{}(i,i+k)}=\sum_{j=0}^{k}\abs{\notred{i+j}}$. Hence we obtain that
    \[
    \sum_{j=0}^{k}(l-l_{i+j})=\sum_{j=0}^{k}\abs{\notred{i+j}} = \sum_{j=0}^{k}m_{i+j} \leq \sum_{j=0}^{k}(n_{i+j}-l_{i+j}+1).
    \]
    Since $n_{i+j}\leq l-1$ for $0\leq j\leq k$, we conclude that the last inequality must be an equality and that $n_{i+j}=l-1$ for $0\leq j\leq k$. Therefore, if $k=0$, then $(M_i)$ is a full admissible configuration of type (II) or (III), while if $k=1$, then we have seen that $i$ is odd and so $(M_i,M_{i+1})$ is a full admissible configuration of type (IV) or (VI). \qedhere
\end{enumerate}    
\end{proof}

We can now show the main result for this paragraph.

\begin{proposition}\label{prop:to not lose need admissible configuration}
Let $(M,P)$ be a $\td$-rigid pair with non-diagonal component $(\red,\blue)$. Let $2\leq i\leq i+k \leq p$ and assume that $M_{i+j}\neq 0$ for $0\leq j\leq k$. Set $\Xi=\Xi(i,i+k)$. Then
\begin{equation}\label{eq:bound on loss k small}
    \abs{(\red\cup \blue)\cap \Xi} + \sum_{j=0}^{k}m_{i+j} \leq \abs{\Xi}.
\end{equation}
Moreover equality holds if and only if $(M_i,\ldots,M_{i+k})$ is a full admissible configuration.
\end{proposition}

\begin{proof}
If $(M_i,\ldots,M_{i+k})$ is a full admissible configuration, then equality holds in (\ref{eq:bound on loss k small}) by Remark \ref{rem:admissible configurations are td-rigid}(c). We now show the other direction.

If $d>2$ and $k\geq 2$ then (\ref{eq:bound on loss k small}) holds by Proposition \ref{prop:to not lose need less than 2}(c). Hence we may assume that $d=2$ or $k\leq 1$. We consider the two cases separately.

\textbf{Case $k\leq 1$}. In this case we work similarly to the proof of Proposition \ref{prop:to not lose need less than 2}. That is, it is enough to show that there exists a subset $X\subseteq \Xi$ such that 
\begin{enumerate}
    \item[(i)] $(\red\cup \blue)\cap X=\varnothing$, 
    \item[(ii)] $\abs{X} \geq \sum_{j=0}^{k}m_{i+j}$, and
    \item[(iii)] if (\ref{eq:condition for maximal interval}) holds, then $(M_i,\ldots,M_{i+k})$ is a full admissible configuration.
\end{enumerate}
Indeed, conditions (i) and (ii) give
\[
\abs{(\red\cup \blue)\cap\Xi} + \sum_{j=0}^{k}m_{i+j} = \abs{(\red\cup\blue)\cap (\Xi\setminus X)} + \sum_{j=0}^{k}m_{i+j} \leq \abs{\Xi}-\abs{X}+\sum_{j=0}^{k}m_{i+j} \leq \abs{\Xi},
\]
which shows (\ref{eq:bound on loss k small}). Moreover, if equality holds in (\ref{eq:bound on loss k small}), then the above line shows that (\ref{eq:condition for maximal interval}) holds and so $(M_i,\ldots,M_{i+k})$ is a full admissible configuration by condition (iii).

Notice that by Lemma \ref{lem:the intervals X} conditions (ii) and (iii) hold for $\notred{}(i,i+k)=\bigcup_{j=1}^{k}\notred{i+j}$ and for $\notblue{}(i,i+k)=\bigcup_{j=1}^{k}\notblue{i+j}$. Hence we may assume that condition (i) fails for these two sets, otherwise the claim follows. Thus we assume that $(\red\cup\blue)\cap \notred{}(i,i+k)\neq\varnothing$, from which we obtain that $\blue\cap\notred{}(i,i+k)\neq \varnothing$, and we assume that $(\red\cup\blue)\cap\notblue{}(i,i+k)\neq \varnothing$, from which we obtain that $\red\cap\notblue{}(i,i+k)\neq \varnothing$. Let $x\in\blue\cap \notred{}(i,i+k)$ and $y\in\red\cap\notblue{}(i,i+k)$. 

We claim that $x<y$. Recall that $\Xi$ is the interval starting at $\notred{i}$ and ending at $\notblue{i+k}$. By assumption, we have that $x\in\notred{i}\cup\notred{i+k}$. Moreover, $y\in \red$ and so $y\not\in\notred{i}$. Hence if $x\in\notred{i}$, since $\Xi$ starts at $\notred{i}$, we have that $x<y$. Similarly, if $y\in\notblue{i+k}$, then we again obtain $x<y$. Hence it is enough to show that if $k=1$ and $x\in \notred{i+1}$ and $y\in\notblue{i}$, then we reach a contradiction. If $i$ is odd, then $\notred{i+1}\subsetneq \notblue{i}$ by Lemma \ref{lem:connection between notblue and notred}, contradicting $x\in\blue\cap\notred{i+1}$. Hence $i$ is even. But then $i+1$ is odd and Lemma \ref{lem:connection between notblue and notred} gives $\notblue{i}\subsetneq \notred{i+1}$, contradicting $y\in\red\cap\notblue{i}$. This shows that $x<y$.

Since $x<y$, by Proposition \ref{prop:description of td-rigid pairs}(a) there exists an interval $Z=[z,z+l-2]$ with $x<z<z+l-2<y$ such that $(\red\cup\blue)\cap Z=\varnothing$. Hence the interval $Z\subseteq \Xi$ satisfies condition (i). Moreover, since $\abs{Z}=l-1$, it also satisfies condition (ii) by Lemma \ref{lem:consecutive diagonals can have at most l-1}. Hence it remains to show that $Z$ satisfies condition (iii). Assume thus that (\ref{eq:condition for maximal interval}) holds for $Z$ and we need to show that $(M_i,\ldots,M_{i+k})$ is a full admissible configuration. Let $x'=\min(\notred{i})$ and note that
\[
x' = \min(\notred{i}) = \min(\Xi) \leq x < z = \min(Z),  
\]
and so $x'\not\in Z$. By the first part of (\ref{eq:condition for maximal interval}) we then have
\[
x'\in \Xi\setminus Z = (\red\cup\blue)\cap\Xi= (\red\cap\Xi)\cup (\blue\cap\Xi).
\]
Since $\red\cap\notred{i}=\varnothing$ and $x'\in\notred{i}$, we conclude that $x'\in \blue\cap\Xi$. We now claim that $\{\xi\in\Xi\mid \xi<z\}\cap \red=\varnothing$. Assume to a contradiction that there exists $\xi\in\Xi$ with $\xi<z$ and $\xi\in \red$. Then $x'=\min(\Xi)\leq z$ and since $x'\in\blue$ and $\xi\in\red$, we have $x'<\xi$. Hence by Proposition \ref{prop:description of td-rigid pairs}(a) there exists an interval $Z'$ of length $l-1$ between $x'$ and $\xi$ such that $(\red\cup\blue)\cap Z'=\varnothing$. In particular, $Z'\cap Z=\varnothing$ since $\xi<z=\min(Z)$. But then by the first part of (\ref{eq:condition for maximal interval}) we have
\[
Z' = Z' \setminus Z \subseteq \Xi\setminus Z = (\red\cup\blue)\cap \Xi.
\]
Hence $Z'\subseteq \red\cup\blue$ and together with $(\red\cup\blue)\cap Z'=\varnothing$ we obtain $Z'=\varnothing$, contradicting $\abs{Z'}=l-1$. We conclude that indeed $\{\xi\in\Xi\mid\xi<z\}\cap\red=\varnothing$. Then
\[
\{\xi \in \Xi \mid \xi < z\} \setminus Z \subseteq \Xi\setminus Z = (\red\cap\Xi)\cup(\blue\cap\Xi)
\]
implies that $\{\xi \in\Xi\mid \xi<z\} \subseteq \blue\cap\Xi$. A similar argument shows that $\{\xi\in\Xi\mid \xi>z+l-2\}\subseteq \red\cap\Xi$. Then we have 
\[
\{\xi \in \Xi \mid \xi < z\}\sqcup \{\xi \in \Xi \mid \xi > z+l-2\} = \Xi\setminus Z = (\red\cap \Xi)\sqcup (\blue\cap \Xi),
\]
and together with the aforementioned inclusions, we conclude that $\{\xi\in\Xi\mid\xi < z\}=\blue \cap\Xi$ and $\{\xi\in\Xi\mid \xi>z+l-2\}=\red\cap\Xi$. Hence $\Xi$ is support to rigid at $z-1$. At this point we consider the subcases $k=0$ and $k=1$ separately.

\textbf{Subcase $k=0$}. If $k=0$, then the second condition of (\ref{eq:condition for maximal interval}) gives $l-1=\abs{Z}=m_i$, from which we obtain that $n_i=l-1$, $l_i=1$ and $m_i=n_i-l_i+1$. Hence in this case $(M_i)$ is a full admissible configuration of type (III). 

\textbf{Subcase $k=1$}. Now assume that $k=1$. Then the first part of (\ref{eq:condition for maximal interval}) gives
\[
(\notred{i+1}\cap\notblue{i}) \cap (\Xi\setminus Z) = (\notred{i+1} \cap \notblue{i}) \cap (\red\cup\blue)\cap \Xi = \varnothing.
\]
Hence $\notred{i+1}\cap\notblue{i}\subseteq Z=[z,z+l-2]$. Therefore, if $\notred{i+1}\cap\notblue{i} = [\alpha,\beta]$, then
\begin{equation}\label{eq:limits for z}
    z-1 \leq \alpha-1 \text{ and } \beta+1\leq z+l-1.
\end{equation}
Also the second part of (\ref{eq:condition for maximal interval}) gives $l-1=\abs{Z}=m_{i}+m_{i+1}$. We consider the cases $i$ odd and $i$ even separately.

Assume that $i$ is odd. Then using Lemma \ref{lem:immediate use of mi li ni} and Proposition \ref{prop:description of td-rigid pairs}(b1) we have
\begin{align*}
l-1 &= m_{i}+m_{i+1} \leq (n_i-l_i+1)+(n_{i+1}-l_{i+1}+1) = n_i+n_{i+1}+2-(l_i+l_{i+1}) \\
&\leq n_i+n_{i+1}+2 -(l+1) \leq (l-1)+(l-1) +2 -(l+1) = l-1.
\end{align*}
We conclude that all inequalities above must be equalities and so $m_{i+1}=n_{i+1}-l_{i+1}+1$ and $n_{i+1}=l-1$. On the other hand, using (\ref{eq:the definition of notblue and notred}) we compute that $[\alpha,\beta]=\notred{i+1}\cap\notblue{i}=[s_i,s_i+(l-l_{i+1})-1]$. Hence in this case (\ref{eq:limits for z}) becomes
\[
z-1 \leq s_i-1 \text{ and } s_i+(l-l_{i+1})-1+1 \leq z+l-1.
\]
Equivalently, and using $m_{i+1}=n_{i+1}-l_{i+1}+1$ and $n_{i+1}=l-1$, we obtain
\[
z-1 \leq s_i-1 \text{ and } s_i-(l-m_{i+1}) \leq z-1.
\]
Hence $z-1\in [s_i-(l-m_{i+1}),s_i-1]$, which shows that $(M_i,M_{i+1})$ is a full admissible configuration of type (VI).

Assume that $i$ is even. Then using Lemma \ref{lem:immediate use of mi li ni} and Proposition \ref{prop:description of td-rigid pairs}(b2) we have
\[
l-1 = m_{i}+m_{i+1} \leq n_i+n_{i+1}+2 - (l_i+l_{i+1}) \leq (l-1)+2 -(l_i+l_{i+1})\leq (l-1)+2-(1+1) = l-1.
\]
We conclude that all inequalities above must be equalities and so $m_i=n_i-l_1+1$ and $l_i=1$. On the other hand, using (\ref{eq:the definition of notblue and notred}) we compute that $[\alpha,\beta]=\notred{i+1}\cap\notblue{i}=[s_i-n_i+1,s_i]$. Hence in this case (\ref{eq:limits for z}) becomes
\[
z-1 \leq s_i-n_i+1-1 \text{ and } s_i+1 \leq z+l-1.
\]
Equivalently, and using $m_{i}=n_{i}-l_{i}+1$ and $l_{i}=1$, we obtain
\[
z-1 \leq s_i-m_i \text{ and } s_i-(l-1) \leq z-1.
\]
Hence $z-1\in [s_i-(l-1),s_i-m_i]$, which shows that $(M_i,M_{i+1})$ is an admissible configuration of type (VII).

\textbf{Case $d=2$}. By Proposition \ref{prop:to not lose need less than 2}(a) and (b) we obtain that $k\leq 2$. Since the case $k\leq 1$ has already been dealt with, we may assume that $k=2$. Then Proposition \ref{prop:to not lose need less than 2}(a) implies that $i$ is even. Hence we have that $d=2$ and $i$ is even and $k=2$. Using (\ref{eq:the definition of notblue and notred}) and (\ref{eq:difference of consecutive simples}) we may immediately compute that
\[
\notred{i}=[s_{i}-l,s_{i}-2] \text{ and } \notred{i+1}=[s_{i}-(l-l_{i+1})+1,s_i].
\]
We claim that 
\[
\notred{}(i,i+1)=\notred{i}\cup\notred{i+1}=[s_i-l,s_i].\]
Indeed, it is enough to show that
\[
s_{i}-(l-l_{i+1})+1\leq s_i-1,
\]
or equivalently that $1+l_{i+1}\leq l-1$. Notice that since since $m_i>0$, we have that $n_i\geq 1$ and so Proposition \ref{prop:description of td-rigid pairs}(b2) gives $n_{i+1}\leq l-2$. In particular $l_{i+1}\leq n_{i+1}\leq l-2$ which shows the required inequality $1+l_{i+1}\leq l-1$. This shows that $\notred{}(i,i+1)=[s_i-l,s_i]$. Similarly we may show that $\notblue{}(i+1,i+2)=[s_{i+2}-l,s_{i+2}]$.

Assume now that $\blue\cap \notred{}(i,i+1)=\varnothing$. Then, since $\abs{\notred{}(i,i+1)}=l+1$, and using Lemma \ref{lem:consecutive diagonals can have at most l-1}, we obtain that
\[
\abs{(\red\cup\blue)\cap \Xi} + \sum_{j=0}^{2}m_{i+j} \leq \abs{\Xi}-(l+1)+(l-1) = \abs{\Xi}-2 < \abs{\Xi},
\]
which shows (\ref{eq:bound on loss k small}). Hence we may assume that there exists $x\in\blue\cap\notred{}(i,i+1)$. A similar argument shows that we may assume that there exists $y\in \red\cap \notblue{}(i+1,i+2)$. By our computations we have that $x\in [s_i-l,s_i]$ and $y\in[s_{i+2}-l,s_{i+2}]$ and so $x<y$ by (\ref{eq:difference of consecutive simples}). Hence by Proposition \ref{prop:description of td-rigid pairs}(a) there exists an interval $Z=[z,z+l-2]$ with $x<z<z+l-2<y$ such that $(\red\cup\blue)\cap Z=\varnothing$. Then using Lemma \ref{lem:consecutive diagonals can have at most l-1} again we have
\begin{equation}\label{eq:bound for the case d=2 i is even and k=2}
\abs{ (\red\cup\blue)\cap \Xi} + \sum_{j=0}^{2}m_{i+j} \leq \abs{\Xi\setminus Z} + (l-1) = \abs{\Xi}-\abs{Z}+(l-1)=\abs{\Xi},
\end{equation}
and so (\ref{eq:bound on loss k small}) holds again. Moreover, in the case of equality in (\ref{eq:bound on loss k small}), we obtain that all inequalities in (\ref{eq:bound for the case d=2 i is even and k=2}) are equalities. Hence $m_{i+j}=n_{i+j}-l_{i+j}+1$ for $j\in \{0,1,2\}$ and also $(\red\cup\blue)\cap \Xi=\Xi\setminus Z$. Then, as in the previous case, we can show that $z-1\in [s_{i+1}-l_{i+2},s_{i}-n_i]$. Therefore in this case $(M_{i},M_{i+1},M_{i+2})$ is a full admissible configuration of type (VIII).
\end{proof}

\subsection{Well-configured \texorpdfstring{$\cC$}{C}-pairs}\label{subsec:well-configured pairs}

Admissible configurations play an important role in the characterization of summand-maximal $\td$-rigid modules. Indeed, they restrict the local behaviour of a summand-maximal $\td$-rigid pair $(M,P)$. The other restriction comes from what happens between two admissible configurations. To explain this other restriction, we introduce the following notion.

\begin{definition} \label{def:the diagonal partition}
Let $(M,P)$ be a $\cC$-pair. The \emph{diagonal component of $M$}\index[definitions]{component!diagonal} is defined to be the set $T=T(M)=\{i\in [2,p] \mid M_i\neq 0\}$\index[symbols]{T@$T(M)$}. By ordering the elements of $T$ there is a unique way to write $T$ as a union of disjoint intervals $T_1,\ldots,T_k$:
\[
T = T_1 \sqcup T_2 \sqcup \cdots \sqcup T_k
\]
such that $\max(T_j)<\min(T_{j+1})-1$ for $1\leq j\leq k-1$. We call $(T_1,\ldots,T_k)$\index[symbols]{T@$(T_1,\ldots,T_k)$} the \emph{diagonal partition of $M$}\index[definitions]{diagonal partition}. 
\end{definition}
In other words, the diagonal partition of a $\cC$-pair $(M,P)$ encodes which sequences of consecutive diagonals $\diag{i},\ldots,\diag{i+k}$ are intersected by $M$.

Let $(M,P)$ be a $\cC$-pair with diagonal partition $(T_1,\ldots,T_k)$. For $1\leq j\leq k$ we write $T_j=[t_{j,1},t_{j,2}]$ and we set $\Xi(T_j) \coloneqq \Xi(t_{j,1},t_{j,2})$\index[symbols]{X@$\Xi(T_j)$}. For $0\leq j\leq k$ we define \index[symbols]{Tzzz@$\Theta_i$}
\begin{equation}\label{eq:definition of the intervals Theta}
    \Theta_j \coloneqq \begin{cases}
        [1,\min(\Xi(T_1))-1], &\mbox{if $j=0$,} \\
        [\max(\Xi(T_j))+1,\min(\Xi(T_{j+1}))-1], &\mbox{if $1\leq j\leq k-1$,} \\
        [\max(\Xi(T_j))+1,n], &\mbox{if $j=k$.}
    \end{cases}
\end{equation}
The interval $\Xi(T_j)$ plays the same role as the interval $\Xi(i,i+k)$ which we saw in the previous paragraph. The interval $\Theta_j$ is exactly the interval between the intervals $\Xi(T_j)$ and $\Xi(T_{j+1})$ as the following examples demonstrate and as we show soon.

\begin{example}
\begin{enumerate}
    \item[(a)] The diagonal partition of the $\tau_4$-rigid pair in Example \ref{example tau 4 rigid Lambda 23 4} is $T_1=[2,4]$ with $\Xi(T_1)=[1,23]$ and $\Theta_0=\varnothing=\Theta_1$.
    \item[(b)] The diagonal partition of the $\tau_4$-rigid pair in Example \ref{Example:Admissible configuration} is $T_1=[2,3]$ with $\Xi(T_1)=[1,15]$, $\Theta_0=\varnothing$ and $\Theta_1=[16,23]$.
\end{enumerate}    
\end{example}

\begin{example}
    \label{Example:diagonal partition with two intervals}
    Let $({\color{rigid} M},{\color{support} P})$ be a $\td$-rigid pair of $\Lambda(n,l)$. 
    \begin{enumerate}
    \item[(a)] If $d=4,\ n=37,\ l=4$ and ${\color{rigid} M}$ is given by
    \[
    \includegraphics{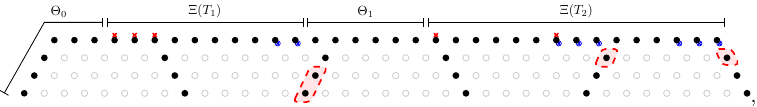}
    \]
    then the diagonal partition of ${\color{rigid} M}$ is $(T_1,T_2)=([3,3],[5,6])$. In addition, we have
    \[
        \Xi(T_1)=[7,16],\ \Xi(T_2)=[23,37], \text{ and }\Theta_0=[1,6],\  \Theta_1=[17,22],\ \Theta_2=\varnothing.
    \]
    \item[(b)] If $d=2,\ n=34,\ l=5$ and ${\color{rigid} M}$ is given by
    \[
    \includegraphics{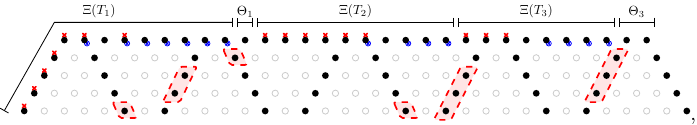}
    \]
    then the diagonal partition is $(T_1,T_2,T_3)=([2,4],[6,7],[9,9])$. In addition, we see that
    \[
    \Xi(T_1)=[1,13],\ \Xi(T_2)=[15,24],\  \Xi(T_3)=[25,32],
    \]
    and
    \[
    \Theta_0=\varnothing,\ \Theta_1=[14,14],\  \Theta_2=\varnothing,\ \Theta_3=[33,34].
    \]
    \end{enumerate}
\end{example}

If $j=0$ or $j=k$ then, as the above examples show, it is possible that $\Theta_j$ is empty. Moreover, in the above examples we see that $\Xi(T_j)\cap\Xi(T_{j+1})=\varnothing$ for $1\leq j\leq k-1$. This always holds for $\td$-rigid pairs, and even for $\cC$-pairs there is only one possible case where this property does not hold as recorded in the following lemma.

\begin{lemma}\label{lem:the intervals Theta}
Let $(M,P)$ be a $\cC$-pair with diagonal partition $(T_1,\ldots,T_k)$. 
\begin{enumerate}
    \item[(a)] Let $1\leq j\leq k-1$ and set $T_j=[\alpha_1,\alpha_2]$ and $T_{j+1}=[\beta_1,\beta_2]$. Then the following are equivalent
    \begin{enumerate}
        \item[(i)] $\Xi(T_j)\cap\Xi(T_{j+1})\neq\varnothing$.
        \item[(ii)] $d=2$ and $\beta_1=\alpha_2+2$ and $\alpha_2$ and $\beta_1$ are both odd and $n_{\alpha_2}> l_{\alpha_2+2}+1$ all hold.
    \end{enumerate}
    \item[(b)] Assume that $\Xi(T_j)\cap\Xi(T_{j+1})=\varnothing$ for $1\leq j\leq k-1$. Then the sequence
    \begin{equation}\label{eq:the partition of [1,n]}
    \Theta_0, \Xi(T_1), \Theta_1, \Xi(T_2),\ldots, \Theta_{k-1}, \Xi(T_k),\Theta_k
    \end{equation}
    is a sequence of consecutive intervals which forms a partition of $[1,n]$.
\end{enumerate}
\end{lemma}

\begin{proof}
\begin{enumerate}
    \item[(a)] From the definition of a diagonal partition we have that $\beta_1=\alpha_2+q$ for some $q\geq 2$. From (\ref{eq:the definition of Xi}) and using Lemma \ref{lem:immediate use of mi li ni} we see that
    \[
    \max(\Xi(T_j)) \leq \begin{cases}
        s_{\alpha_2}+l-2, &\text{ if }\alpha_2 \text{ is odd,}\\
        s_{\alpha_2}, &\text{ if } \alpha_2 \text{ is even,}
    \end{cases}
    \]
    and 
    \[
    \min(\Xi(T_{j+1})) \geq  \begin{cases}
        s_{\beta_1-1}-l+2, &\text{ if }\beta_1 \text{ is odd,}\\
        s_{\beta_1-1}, &\text{ if }\beta_1 \text{ is even.}
    \end{cases} 
    \]
    On the other hand, (\ref{eq:the definition of Xi}) also gives that $\min(\Xi(T_j))\leq s_{\alpha_1-1}$. Together with $\alpha_1\leq \alpha_2=\beta_1-q\leq \beta_1-2$, we obtain
    \[
    \min(\Xi(T_j)) \leq s_{\alpha_1-1} \leq s_{\alpha_2-1} \overset{(\ref{eq:difference of consecutive simples})}{=
    } s_{\alpha_2+1} -(d-1)l+2 \leq s_{\beta_1-1}-l+2 \leq \min(\Xi(T_j+1)).
    \]
    We conclude that $\Xi(T_j)\cap\Xi(T_{j+1})=\varnothing$ if and only if 
    \begin{equation}\label{eq:consecutive Tis do not intersect in Xi}
    \max(\Xi(T_j))<\min(\Xi(T_{j+1})),
    \end{equation} 
    holds. From (\ref{eq:difference of consecutive simples}) we obtain the following computation
    \[
    s_{\beta_1-1}-s_{\alpha_2}= \begin{cases}
        \frac{q-1}{2}[(d-1)l+2],&\text{ if }\alpha_2 \text{ is odd and }\beta_1 \text{ is even,}\\
        \frac{q-1}{2}[(d-1)l+2],&\text{ if }\alpha_2 \text{ is even and }\beta_1 \text{ is odd,}\\
        \frac{q-2}{2}[(d-1)l+2]+\frac{d}{2}l,&\text{ if }\alpha_2 \text{ and }\beta_1\text{ are both odd},\\
        \frac{q-2}{2}[(d-1)l+2]+\frac{d-2}{2}l+2,&\text{ if }\alpha_2 \text{ and }\beta_1\text{ are both even}.
    \end{cases}
    \]
    Notice that if one of $\alpha_2$ and $\beta_1$ is even and the other is odd, then it follows that $q\geq 3$. Using this and the above inequalities and equalities it is straightforward to verify that (\ref{eq:consecutive Tis do not intersect in Xi}) holds, except for when $d=2$ and $q=2$ and $\alpha_2$ and $\beta_1$ are both odd. 

    Now assume that $\Xi(T_j)\cap\Xi(T_{j+1})\neq\varnothing$. Then this is equivalent to $\max(\Xi(T_j)) \geq \min(\Xi(T_{j+1}))$. Therefore (\ref{eq:consecutive Tis do not intersect in Xi}) does not hold and so we may assume that $d=2$ and $\beta_1=\alpha_2+2$ and $\alpha_2$ and $\beta_1$ are both odd. Using (\ref{eq:the definition of Xi}) we have that
    \[
    s_{\alpha_2}+n_{\alpha_2}-1 = \max(\Xi(T_j)) \geq \min(\Xi(T_{j+1}))  = s_{\beta_1-1}-(l-l_{\beta_1})+1,
    \]
    or, equivalently, using $\beta_1=\alpha_2+2$, that
    \[
    n_{\alpha_2} \geq s_{\alpha_2+1}-s_{\alpha_2}-l+l_{\alpha+2}+2.
    \]
    Since $s_{\alpha_2+1}-s_{\alpha_2}=l$, this is equivalent to $n_{\alpha_2}> l_{\alpha_2+2}+1$, as required.
    \item[(b)] Follows immediately by part (a) and the definition of $\Theta_j$.\qedhere
    \end{enumerate}
\end{proof}

\begin{remark}\label{rem:we have consecutive intervals}
Notice that if $(M,P)$ is a $\td$-rigid pair, then condition (ii) in Lemma \ref{lem:the intervals Theta}(a) never occurs by Proposition \ref{prop:description of td-rigid pairs}(b1). Hence for $\td$-rigid pairs, the condition in Lemma \ref{lem:the intervals Theta}(b) is always satisfied.
\end{remark}

We are now ready to give the definition of a well-configured $\cC$-pair, which serves as the combinatorial description of summand-maximal $\td$-rigid pairs. First let $(M,P)$ be a $\cC$-pair with diagonal partition $(T_1,\ldots,T_k)$. Assume that for each $j\in[1,k]$ the sequence $(M_{t_{j,1}},\ldots,M_{t_{j,2}})$ is an admissible configuration. When this sequence is an admissible configuration of type (III), we have that $T_j=\{t_j\}$ and so $t_{j,1}=t_{j,2}=t_j$, and also that $M_{t_j} = \diagm{t_j}$. This extreme case plays a special role which is recorded in the following definition.

\begin{definition}\label{def:well-configured}
Let $(M,P)$ be a $\cC$-pair with non-diagonal component $(\red,\blue)$. Let $(T_1,\ldots,T_k)$ be the diagonal partition of $(M,P)$. We say that $(M,P)$ is \emph{well-configured}\index[definitions]{well-configured} if any of the following conditions hold:
\begin{enumerate}
    \item[(a)] $T=\varnothing$, $\red=[1,x]$ and $\blue=[x+1,n]$ for some $0\leq x\leq n$, or
    \item[(b)] $T\neq \varnothing$ and the following hold:
    \begin{enumerate}
        \item[(i)] for $1\leq j\leq k$ we have that $(M_{t_{j,1}},\ldots,M_{t_{j,2}})$ is a full admissible configuration,
        \item[(ii)] for $1\leq j\leq k-1$ we have that $\Xi(T_j)\cap\Xi(T_{j+1})=\varnothing$,
        \item[(iii)] for $0\leq j\leq k$ we have that $\Theta_j$ is either rigid, support, or rigid to support,
        \item[(iv)] for $1\leq j\leq k$ we have that if $\Xi(T_j)$ is rigid, then $\Theta_{j-1}$ is rigid or $(M_{t_{j,1}},\ldots,M_{t_{j,2}})$ is a full admissible configuration of type (III), and
        \item[(v)] for $1\leq j\leq k$ we have that if $\Xi(T_j)$ is support, then $\Theta_{j}$ is support or $(M_{t_{j,1}},\ldots,M_{t_{j,2}})$ is a full admissible configuration of type (III).
    \end{enumerate}
\end{enumerate}
\end{definition}

There is admittedly a lot of data in Definition \ref{def:well-configured}, but almost all of it is collected in the much more compact Table \ref{table:well-configured}. The only exception is condition (ii) in Definition \ref{def:well-configured}(b) which happens only under the extremely specific circumstances described in Lemma \ref{lem:the intervals Theta}(a)(ii) and needs to be checked separately. Instead of giving an explicit but cumbersome general algorithm of how to use that table, we illustrate with three representative examples.

\begin{table}[h]
\begin{tabular}{c m{2em}|c|c|c|c|c|c|c|c|c}
\multicolumn{2}{l|}{\multirow{2}{*}{\diagbox{$T_i$}{$T_{i+1}$}}} & \multirow{2}{*}{I}  &\multirow{2}{*}{II}  &\multirow{2}{*}{III}  &\multirow{2}{*}{IV}&\multirow{2}{*}{V}&\multicolumn{2}{c|}{VI}&\multirow{2}{*}{VII}&\multirow{2}{*}{VIII}  \\
&&&&&&&R&SR&&
\\\hline
\multicolumn{2}{l|}{I} & S  &\cellcolor{black}  &S  &\cellcolor{black}&S&\cellcolor{black}&S&S&S  \\\hline
\multicolumn{2}{l|}{II} & A  &R  &A  &R&A&R&A&A&A  \\\hline
\multicolumn{2}{l|}{III} & A  &R  &A  &R&A&R&A&A&A  \\\hline
\multicolumn{2}{l|}{IV} & A  &R  &A  &R&A&R&A&A&A  \\\hline
\multicolumn{2}{l|}{V} & S  &\cellcolor{black}  &S  &\cellcolor{black}&S&\cellcolor{black}&S&S&S  \\\hline
\multicolumn{2}{l|}{VI} & A  &R  &A  &R&A&R&A&A&A  \\\hline
\multirow{2}{*}{VII}&S & S  &\cellcolor{black}  &S  &\cellcolor{black}&S&\cellcolor{black}&S&S&S  \\\cline{2-11}
&SR & A  &R  &A  &R&A&R&A&A&A  \\\hline
\multicolumn{2}{l|}{VIII} & A  &R  &A  &R&A&R&A&A &A 
\end{tabular}
\caption{In the table we show the possible direct neighbouring admissible types of a well-configured pair as non-blacked out cells. The rows are the possible type of $T_i$ and the columns of $T_{i+1}$. For $T_i$ of type VII and $T_{i+1}$ of type VI also the form of $\Xi(T_\bullet)$ is of importance, hence they are specified. The internal cells contain a label denoting the form of $\Theta_i$. The labels are given by: "R" -- Rigid, "S" -- Support, "SR" -- Support to Rigid, "A" -- Any of Rigid, Support and Rigid to Support.\label{table:well-configured}}
\end{table}

\begin{example}\label{Example:Not well-configured pair}
\begin{enumerate}
    \item[(a)] Let $({\color{rigid}M},{\color{support}P})$ be a $\tau_4$-rigid pair of $\Lambda(37,4)$. Assume that ${\color{rigid}M}$ is as in Example \ref{Example:diagonal partition with two intervals}(a), that is
    \[
    \includegraphics{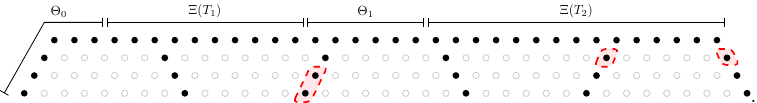}
    \]
    Assume that $\Xi(T_1)$ and $\Xi(T_2)$ are such that $({\color{rigid}M_3})$ is a full admissible configuration of type (I) and $({\color{rigid}M_5},{\color{rigid}M_6})$ is a full admissible configuration of type (IV). Table \ref{table:well-configured} indicates that this is impossible for a well-configured pair as a full admissible configuration of type (IV) can not come immediately after a full admissible configuration of type (I). We conclude that $({\color{rigid}M},{\color{support}P})$ is not well-configured.
    \item[(b)] Let us construct a well-configured $\cC$-pair of $\Lambda(37,4)$ by slightly adjusting ${\color{rigid}M}$ from the previous example. There are many ways to achieve this. For example, we may add a summand to ${\color{rigid}M_6}$ and change the configuration $({\color{rigid}M_5},{\color{rigid}M_6})$ to be a full admissible configuration of type (VI). This gives the following picture:
    \[
    \includegraphics{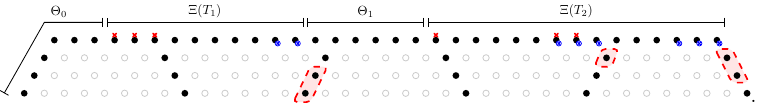}
    \]
    Consulting Table \ref{table:well-configured} we see all the possibilities for the sets $\red$ and $\blue$. In particular, since we have an admissible configuration of type (I) followed by an admissible configuration of type (VI), we see that $\Xi(T_2)$ needs to be support to rigid and that $\Theta_1$ has to be rigid. Moreover, since $({\color{rigid} M_3})$ is an admissible configuration of type (I), we have that $\Xi(T_1)$ needs to be support. To find the choices for $\Theta_0$ we may consult the column of Table \ref{table:well-configured} corresponding to the type of the first admissible configuration we have. As long as the option ``A'' appears there, $\Theta_0$ can be anything among rigid, rigid to support at any point in $\Theta_0$, and support; otherwise it has to be rigid. In our case we see that we may pick $\Theta_0$ to be rigid to support at, say, $4\in \Theta_0=[1,6]$. Then we obtain the following well-configured $\cC$-pair: 
    \[
    \includegraphics{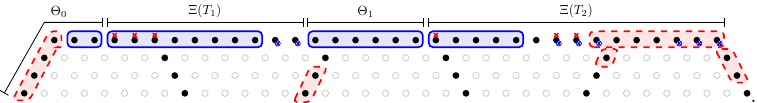}
    \]
    \item[(c)] Let $d=2$ and let us now construct a well-configured $\cC$-rigid pair of $\Lambda(34,5)$. In this case we have that $p=10$. We may pick any $i\in [2,10]$ for our first admissible configuration to start. Say we pick $i=2$. Then, we may pick any admissible configuration type from the first column of Table \ref{table:well-configured}. Say we pick an admissible configuration of type (VIII). This covers three diagonals as can be seen in Definition \ref{def:admissible configurations} and so defines $T_1=[2,4]$ and also $l_2=1$, $l_3=6-l_4$, $n_3=4-n_2$, $n_4=4$. We may choose, say $n_2=1$, $n_3=3$, $l_3=2$, $l_4=4$ which defines $({\color{rigid}M_2},{\color{rigid}M_3},{\color{rigid}M_4})$.

    Next, we must have that ${\color{rigid}M_5}=0$ and so we may pick $i\in [6,10]$ for our second admissible configuration to occur, or stop here. We may pick $i=6$ and our second admissible configuration to be of type (VII) with, say $l_6=l_7=1$, $n_6=1$ and $n_7=3$. This defines $T_2=[6,7]$ and $({\color{rigid}M_6},{\color{rigid}M_7})$. Then for the next diagonal we must have ${\color{rigid}M_8}=0$ and so we may choose $i\in[9,10]$ for the next admissible configuration to occur, or again stop here. 
    
    Say we pick $i=9$. Table \ref{table:well-configured} reveals that we may pick the next admissible configuration to be of type (III).  However, since we are in the case $d=2$, we have to be slightly more careful. Condition (ii) in Definition \ref{def:well-configured}(b) is the only condition not encoded in Table \ref{table:well-configured} and we have to check that it is not violated. This can be done easily via Lemma \ref{lem:the intervals Theta}(a). An admissible configuration of type (III) gives $T_3=[9,9]$, $l_9=1$ and $n_9=4$. Then we observe that $n_7=3>2=l_9+1$, which implies by Lemma \ref{lem:the intervals Theta} that $\Xi(T_2)\cap\Xi(T_3)\neq \varnothing$. This means that this choice does not lead to a well-configured pair and so we may not pick an admissible configuration of type (III) at this point. We may instead choose the last configuration to be of type (II). Consulting Table \ref{table:well-configured} as in the previous example reveals the possible nature for the intervals $\Xi(T_1)$, $\Theta_1$, $\Xi(T_2)$, $\Theta_2$, $\Xi(T_3)$ and $\Theta_3$. A possible well-configured $\cC$-pair obtained this way is the following:
    \[
    \includegraphics{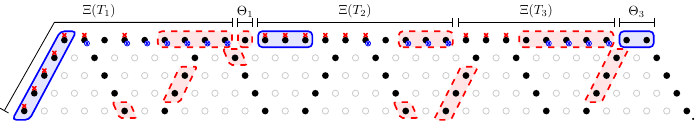}
    \]
\end{enumerate}
\end{example}

In the above examples all well-configured $\cC$-pairs are maximal $\td$-rigid pairs with $n=\abs{\Lambda}$ summands. Our main aim is to prove that well-configured $\cC$-pairs are characterized by this property and, furthermore, this property characterizes summand-maximal $\td$-rigid pairs. We start by showing that well-configured $\cC$-pairs are always $\td$-rigid.

\begin{proposition}\label{prop:well-configured implies td-rigid}
Let $(M,P)$ be a well-configured $\cC$-pair. Then $(M,P)$ is a basic $\td$-rigid pair.
\end{proposition}

\begin{proof}
We need to check that the conditions in Proposition \ref{prop:description of td-rigid pairs} hold. Let $(T_1,\ldots,T_k)$ be the diagonal partition of $(M,P)$. If $T=\varnothing$ then we are in case (a) of Definition \ref{def:well-configured} and clearly $(M,P)=(P(\red),P(\blue))$ is $\td$-rigid. 

Assume now that $T\neq \varnothing$. Then for $1\leq j\leq k$ we have that $(M_{t_{j,1}},\ldots,M_{t_{j,2}})$ is a full admissible configuration by condition (i) in Definition \ref{def:well-configured}(b). Let $i\in [2,p]$. If $i\in T_j$ for some $j$, then by Remark \ref{rem:admissible configurations are td-rigid}(b) we have that conditions (b1) and (b2) of Proposition \ref{prop:description of td-rigid pairs} hold. If $i\not\in T_j$, then $M_{i}=0$ and conditions (b1) and (b2) of Proposition \ref{prop:description of td-rigid pairs} hold trivially. Hence conditions (b1) and (b2) of Proposition \ref{prop:description of td-rigid pairs} hold for all $i\in [2,p]$.

For $0\leq q\leq 2k$ we set $X_q = \Theta_{\frac{q}{2}}$ for $q$ even and $X_q=\Xi(T_{\frac{q+1}{2}})$ for $q$ odd. We then set $\red_q=\red\cap X_q$ and $\blue_q=\blue\cap X_q$. By condition (ii) in Definition \ref{def:well-configured}(b) and Lemma \ref{lem:the intervals Theta} we have that $\red$ is the disjoint union of the sets $\red_0,\red_1,\ldots,\red_{2k}$ and similarly for $\blue_q$. We first claim that condition (b3) of Proposition \ref{prop:description of td-rigid pairs} hold with $\red_q$ respectively $\blue_q$ in place of $\red$ respectively $\blue$. Indeed, if $q$ is even, then this follows by Remark \ref{rem:admissible configurations are td-rigid}(b), while if $q$ is odd, this follows since $\Theta_j$ intersects no interval of the form $\notblue{i}$ or $\notred{i}$ by Lemma \ref{lem:the intervals Theta}(b). This shows that condition (b3) in Proposition \ref{prop:description of td-rigid pairs} holds.

Notice that all of the intervals $X_q$ are either rigid, support, support to rigid, or rigid to support. It remains to show that condition (a) of Proposition \ref{prop:description of td-rigid pairs} holds. Thus let $x\in \blue$ and $y\in \red$ such that $1\leq x < y \leq n$ and we need to find an interval of length $l-1$ between them that intersects neither $\red$ nor $\blue$. Then $x\in \blue_q$ and $y\in\red_{q'}$ for some $q\leq q'$. If there exists an interval $X_{q''}$ with $q\leq q''\leq q'$ and such that $X_{q''}$ is support to rigid, then we are done. Hence we may assume that for all $q''$ with $q\leq q'' \leq q'$ we have that $X_{q''}$ is not support to rigid. Since $X_q\cap \blue\neq \varnothing$, we conclude that $X_q$ is support or rigid to support. Similarly we obtain that $X_{q'}$ is rigid or rigid to support. Since $x<y$ with $x\in\blue$ and $y\in\red$, we conclude that $X_q\neq X_{q'}$. It follows that there exists $r$ such that $q\leq r <r+1\leq q'$ and such that $X_r$ is support or rigid to support, while $X_{r+1}$ is rigid or rigid to support. Assume that $r$ is odd. Then $X_r=\Xi(T_{\frac{r+1}{2}})$ is support, since no interval $\Xi(T_j)$ can be rigid to support by part (v) of Definition \ref{def:well-configured}(b) and the definition of admissible configurations. Similarly we obtain that $X_{r+1}=\Theta_{\frac{r+1}{2}}$ is rigid. Hence by condition (v) in Definition \ref{def:well-configured} we obtain that the full admissible configuration corresponding to $T_{\frac{r+1}{2}}$ is of type (III). By Remark \ref{rem:admissible configurations are td-rigid}(d) we have that $\Xi(T_{\frac{r+1}{2}})$ is the disjoint union of two consecutive intervals of the form $\blue_r$ and $\notblue{i}$, with the last one having length $l-1$. Since $X_r$ is support, it follows that $\notblue{i}\cap \red=\varnothing$. This shows the claim and so condition (a) of Proposition \ref{prop:description of td-rigid pairs} holds if $r$ is odd; the case where $r$ is even is similar.
\end{proof}

The following statement establishes a crucial bound on the number of indecomposable summands for a basic $\td$-rigid pair.

\begin{proposition}\label{prop:bound for M with partition}
Let $(M,P)$ be a basic $\td$-rigid pair with non-diagonal component $(\red,\blue)$. Then
\begin{equation}\label{eq:absolute upper bound for summands}
\abs{\red} + \abs{\blue} + \sum_{i=2}^{p}m_i \leq n.    
\end{equation}
Moreover, if equality holds in (\ref{eq:absolute upper bound for summands}), then $(M,P)$ is well-configured.
\end{proposition}

\begin{proof}
Let $T$ be the diagonal component of $M$. If $T=\varnothing$, then 
\[
\abs{\red}+\abs{\blue} + \sum_{i=2}^{p}m_i = \abs{\red\cup\blue} = \abs{(\red\cup\blue)\cap [1,n]} \leq n.
\]
Moreover, equality holds only if $(\red\cup\blue)\cap [1,n]=[1,n]$. It is easy to see that this implies that case (a) of Definition \ref{def:well-configured} must be satisfied and so $(M,P)$ is well-configured.
 
Assume that $T\neq \varnothing$. Let $(T_1,\ldots,T_k)$ be the diagonal partition of $(M,P)$. By Remark \ref{rem:we have consecutive intervals} we can apply Lemma \ref{lem:the intervals Theta}(b). In particular, condition (ii) of Definition \ref{def:well-configured}(b) holds. Moreover, we have that for $i\neq j$ we have $\Xi(T_i)\cap \Xi(T_j)=\varnothing$ and $\Theta_i\cap \Theta_j=\varnothing$. We then set
\[
\Xi = \bigsqcup_{j=1}^{k} \Xi(T_j) \text{ and } \Theta = \bigsqcup_{j=0}^{k} \Theta_j.
\]
By Lemma \ref{lem:the intervals Theta}(b) we also have that $[1,n]=\Xi\sqcup \Theta$. Using Proposition \ref{prop:to not lose need admissible configuration} we have that
\begin{align*}
\abs{(\red\cup\blue)\cap \Xi} &= \sum_{j=1}^{k} \abs{(\red\cup\blue)\cap \Xi(T_j)} 
\leq\sum_{j=1}^{k} \left(\abs{\Xi(T_j)} - \sum_{i\in T_j} m_i\right) \\
&= \sum_{j=1}^{k}\abs{\Xi(T_j)} - \sum_{i=2, m_i\neq 0}^{p} m_i 
= \abs{\Xi} - \sum_{i=2}^{p}m_i,
\end{align*}
where if equality holds, then $(M_{t_{j,1}},\ldots, M_{t_{j,2}})$ is a full admissible configuration for all $j\in [1,k]$. Hence we have
\begin{equation}\label{eq:inequality for the whole Xi}
\abs{(\red\cup\blue)\cap\Xi} \leq \abs{\Xi}-\sum_{i=2}^{p}m_i,    
\end{equation}
and equality holds only if condition (i) of Definition \ref{def:well-configured}(b) holds. Using this and $[1,n]=\Xi\sqcup\Theta$, we have
\begin{equation}\label{eq:showing 3.10}
\abs{\red}+\abs{\blue} = \abs{(\red\cup\blue)\cap [1,n]} = \abs{(\red\cup\blue)\cap\Xi} + \abs{(\red\cup\blue)\cap\Theta} \leq \abs{\Xi} - \sum_{i=2}^{p}m_i + \abs{\Theta} = n - \sum_{i=2}^{p}m_i,
\end{equation}
which shows (\ref{eq:absolute upper bound for summands}). Moreover, equality in (\ref{eq:showing 3.10}) holds only if, in addition to equality in (\ref{eq:inequality for the whole Xi}), we have that 
\begin{equation}\label{eq:equality in Theta}
    \abs{(\red\cup\blue)\cap\Theta} = \abs{\Theta}.
\end{equation}
Hence it is enough to show that if (\ref{eq:equality in Theta}) holds, then conditions (iii), (iv) and (v) of Definition \ref{def:well-configured} are satisfied. Let $0\leq j \leq k$ and we first claim that there exists no interval $I\subseteq \Theta_j$ which is support to rigid. Indeed, applying Proposition \ref{prop:description of td-rigid pairs}(a) to such an interval $I$ gives the existence of a subset $Z\subsetneq \Theta_{j-1}$ of length $l-1$ with $(\red\cup\blue)\cap Z=\varnothing$, which contradicts (\ref{eq:equality in Theta}). Since (\ref{eq:equality in Theta}) holds and $\Theta_j$ contains no interval which is support to rigid, it follows that $\Theta_j$ is rigid, support or support to rigid. This shows condition (iii) of Definition \ref{def:well-configured}. Let us now only show that condition (iv) of Definition \ref{def:well-configured} is satisfied, as the proof for condition (v) is similar.

Assume then that (\ref{eq:inequality for the whole Xi}) holds, that $\Xi(T_j)$ is rigid and that $(M_{t_{j,1}},\ldots,M_{t_{j,2}})$ is a full admissible configuration of any type other than (III) and we show that $\Theta_{j-1}$ is rigid. Assume to a contradiction that $\Theta_{j-1}$ is not rigid. In particular, $\Theta_{j-1}$ is non-empty since the empty interval is rigid. Since we have shown that condition (iii) of Definition \ref{def:well-configured} holds, it follows that $\Theta_{j-1}$ is either support or rigid to support. In any case, we have that $\max(\Theta_{j-1})\in\blue$. Now consider the interval $\Xi(T_j)$. By construction, it starts at the interval $\notred{t_{j,1}} = [\rho_1,\rho_2]$. Since $\Xi(T_j)$ is rigid, we have that $\rho_2+1\in\red$. By Lemma \ref{lem:the intervals Theta}(b) we have that $\max(\Theta_{j-1})=\rho_1-1$. Since $\max(\Theta_{j-1})\in \blue$, we conclude that $\rho_1-1\in \blue$. Then $\Hom{\Lambda}{P(\rho_1-1)}{P(\rho_2+1)}=0$, and so we conclude that $\rho_2+1 - (\rho_1-1) \geq l$. Then 
\[
l-l_{t_{j,1}}=\abs{\notred{t_{j,1}}} = \rho_2-\rho_1 + 1 \geq l-1,
\]
and so $l_{t_{j,1}}\leq 1$. Since $1\leq l_{t_{j,1}}$ holds by Lemma \ref{lem:immediate use of mi li ni}, we have that $(M_{t_{j,1}},\ldots,M_{t_{j,2}})$ is a full admissible configuration of type (I), (V), (VII) or (VIII) (since we have assumed it cannot be of type (III)). But in none of these cases is the interval $\Xi(T_j)$ rigid. Hence we reach a contradiction and we conclude that $\Theta_{j-1}$ is rigid as required.
\end{proof}

\begin{corollary}\label{cor:bound on summands}
Let $(M,P)$ be a basic $\td$-rigid pair. Then $\abs{M}+\abs{P}\leq \abs{\Lambda}$.    
\end{corollary}

\begin{proof}
Using Proposition \ref{prop:bound for M with partition} we have
\[
\abs{M}+\abs{P} = \abs{P(\red)}+\abs*{\left(\bigoplus_{i=2}^{p}M_i\right)}+\abs{P(\blue)} = \abs{\red} + \sum_{i=2}^p m_i+ \abs{\blue}\leq n=\abs{\Lambda}.
\qedhere\]
\end{proof}

We can now formulate our main result for this section.

\begin{theorem}\label{thrm:taud tilting is well-configured}
Let $(M,P)$ be a $\cC$-pair. Then the following are equivalent.
\begin{enumerate}
    \item[(a)] $(M,P)$ is a summand-maximal $\td$-rigid pair.
    \item[(b)] $(M,P)$ is a $\td$-rigid pair and $\abs{M}+\abs{P}=\abs{\Lambda}$.
    \item[(c)] $(M,P)$ is well-configured.
\end{enumerate}
\end{theorem}

\begin{proof}
(a) implies (b): If (a) holds, then by the definition of summand-maximal $\td$-rigid pairs we have that $\abs{\Lambda}\leq \abs{M}+\abs{P}$. By Corollary \ref{cor:bound on summands} we obtain that $\abs{\Lambda}=\abs{M}+\abs{P}$.

(b) implies (a): Let $(M',P')$ be any $\td$-rigid pair. Then by Corollary \ref{cor:bound on summands} we have that $\abs{M'}+\abs{P'}\leq \abs{\Lambda}=\abs{M}+\abs{P}$. Since $(M',P')$ was an arbitrary $\td$-rigid pair, it follows that $(M,P)$ is a summand-maximal $\td$-rigid pair.

(b) implies (c): If (b) holds, then we have equality in (\ref{eq:absolute upper bound for summands}). It follows that $(M,P)$ is well-configured by Proposition \ref{prop:bound for M with partition}.

(c) implies (b): Assume now that $(M,P)$ is well-configured. Then $(M,P)$ is a $\td$-rigid pair by Proposition \ref{prop:well-configured implies td-rigid}. It remains to show that $\abs{M}+\abs{P}=\abs{\Lambda}$. Let $T$ be the diagonal component of $M$. If $T=\varnothing$, then by Definition \ref{def:well-configured}(a) we have $(M,P)=(P(\red),P(\blue))$ and $\abs{\red}+\abs{\blue}=n=\abs{\Lambda}$, as required. 

Assume now that $T\neq 0$ and let $(T_1,\ldots, T_k)$ be the diagonal partition of $(M,P)$. Then conditions (i)---(v) of Definition \ref{def:well-configured}(b) are satisfied. Let
\[
\Xi = \bigsqcup_{j=1}^{k} \Xi(T_j) \text{ and } \Theta = \bigsqcup_{j=0}^{k} \Theta_j.
\]
Then condition (ii) of Definition \ref{def:well-configured}(b) allows us to apply Lemma \ref{lem:the intervals Theta}(b), 
so that $[1,n]=\Xi\sqcup\Theta$. By the definition of the intervals $\Xi$ and $\Theta$, we have that $\Theta_{j}\cap \notred{i}=\varnothing=\Theta_{j}\cap\notblue{i}$ for any $i\in [2,p]$. Hence condition (iii) of Definition \ref{def:well-configured}(b) gives that \begin{equation}\label{eq:well-configured is maximal on Theta}
(\red\cup\blue)\cap\Theta=\Theta.
\end{equation}

Now let $1\leq j\leq k$. By Remark \ref{rem:admissible configurations are td-rigid}(c) we have that
\begin{equation}\label{eq:admissible configurations do not lose 2}
\abs{(\red\cup\blue)\cap \Xi(T_j)} +\sum_{i\in T_j}m_i  = \abs{\Xi(T_j)}.
\end{equation}
Using (\ref{eq:admissible configurations do not lose 2}) we have
\[
\abs{(\red\cup\blue)\cap\Xi} = \sum_{j=1}^{k} \abs{(\red\cup\blue)\cap\Xi(T_j)} = \sum_{j=1}^{k}\left(\abs{\Xi(T_j)}-\sum_{i\in T_j}m_i\right) = \abs{\Xi} - \sum_{i=2}^{p} m_i.
\]
Using this and (\ref{eq:well-configured is maximal on Theta}) we have
\[
\abs{M}+\abs{P} = 
\abs{(\red\cup\blue)\cap [1,n]} + \sum_{i=2}^{p}m_i = \abs{(\red\cup\blue)\cap \Theta} + \abs{(\red\cup\blue)\cap \Xi} + \sum_{i=2}^{p}m_i = \abs{\Theta}+\abs{\Xi}=n,
\]
which shows that (b) holds.
\end{proof}

\subsubsection{The case \texorpdfstring{$d=\gldim(\Lambda)$}{d=gldimL}}
We finish this section by applying Theorem \ref{thrm:taud tilting is well-configured} to count the summand-maximal $\td$-rigid pairs in the special case when $d=\gldim(\Lambda)$.

\begin{remark} \label{rem: Counting strongly maximal when d-rep finite}
\begin{enumerate}
    \item[(a)] Assume that $d=\gldim(\Lambda(n,l))$ and that $\Lambda(n,l)$ admits a $d$-cluster tilting subcategory $\cC$. Then $p=2$ by Remark \ref{rem:d=gldim(Lambda) for Nakayama}, and consequently the diagonal partition of a summand-maximal $\td$-rigid pair $(M,P)$ is either empty or consists only of $T_1=\{2\}$. In the first case we have that $[1,n]=\Phi_0$ is either rigid, support or rigid to support, providing $2+(n-1)$ possible pairs. In the second case we have three possible subcases:
    \begin{enumerate}
    \item[(1)] $(M_2)$ is full admissible of type (I) with $\Xi_2=[1,n]$ rigid,
    \item[(2)] $(M_2)$ is full admissible of type (II) with $\Xi_2=[1,n]$ support, and
    \item[(3)] $(M_2)$ is full admissible of type (III) with $\Xi_2=[1,n]$ rigid, support or support to rigid at some $x\in [1,n-l]$.
\end{enumerate}
Each of (1) and (2) provides $l-2$ different pairs, and (3) provides $2+(n-l)$ pairs. In total we then have $2n+l-1$ different summand-maximal $\td$-rigid pairs of $\Lambda(n,l)$ when $d=\gldim(\Lambda)$. 

\item[(b)] To illustrate the exponential rate of growth in the number of summand-maximal $\td$-rigid pairs, let us note that through a bit more work in the same spirit as in part (a) we can calculate that for $p=4$ and either $(d>2 $ and $ l>2)$ or $(d=2$ and $l=3)$ there exists exactly
\[
\frac{1}{27}\left(-19l^3+18l^2n+144l^2-6ln^2-36ln-126l+2n^3+36n^2+72n-135\right)
\]
summand-maximal $\td$-rigid pairs.
\end{enumerate}
\end{remark}
\label{Section:Structure of tau_d rigid}

\section{\texorpdfstring{$d$}{d}-torsion classes}
\label{Section:d-torsion}

\subsubsection*{Aim.} In this section we once again fix an algebra $\Lambda=\Lambda(n,l)$ that satisfies the conditions of Theorem \ref{thm:VasoClassifyAcyclicCluster} and so admits a $d$-cluster tilting subcategory $\cC$. Our aim now is to classify $d$-torsion classes of $\Lambda$. These are subcategories $\cU\subseteq\cC$ of $\Lambda$ which serve as a higher analogue of torsion classes and are defined by certain homological properties.

\subsubsection*{Local restrictions on $d$-torsion classes.} Our basic strategy is in principle similar to our strategy for summand-maximal $\td$-rigid pairs. That is, we first find some local conditions that must be satisfied for a subcategory $\cU\subseteq\cC$ to be a $d$-torsion class, and then we find how these local conditions must be joined together. Before we proceed, let us give a brief overview of this strategy.

To describe the local condition, assume that we know that $\cU\subseteq\cC$ is a $d$-torsion class and let us consider the subcategory $\cU_i=\cU\cap\diag{i}$ consisting of all modules in $\cU$ which lie in the diagonal $\diag{i}$. First we can notice that $\cU_i$ must have a specific form: $\cU_i$ must be one of $0$, $\downdiag{i}{x}$, $\updiag{i}{y}$ or $\diag{i}$ (with some additional restrictions depending on the parity of $i$). Additionally, knowing the type of $\cU_i$ places restrictions on the type of $\cU_{i+1}$ as well. Finally, the information of the type of $\cU_i$ and $\cU_{i+1}$ restricts which indecomposable projective-injective modules with top between $s_i+a$ and $s_{i+1}+b$ can be included in $\cU$ (the $a$ and $b$ can be made precise and depend only on the parity of $i$).

\subsubsection*{Using the local restrictions to construct $d$-torsion classes.} Given this local behaviour, we can construct a $d$-torsion class by choosing a possible $\cU_1$, then a suitable $\cU_2$ and a compatible collection of indecomposable projective-injective modules in-between the diagonals $\diag{1}$ and $\diag{2}$, then a suitable $\cU_3$ (which depends only on $\cU_2$) and so on. We then collect all of this information about which type of diagonal may follow which one in the graph (\ref{eq:multigraph1 for d-torsion}), where the labels of the arrows inform us which choices we have for the indecomposable projective-injective modules in between two diagonals. Thus, paths in the graph (\ref{eq:multigraph1 for d-torsion}) with $p-1$ arrows (and so going throug $p$ vertices, which correspond to $p$ diagonals) which start in a suitable orientation for $\cU_1$ give $d$-torsion classes, and we prove that each $d$-torsion class may be constructed in this way.

\subsubsection*{A special case.} Along the way there are again a few special considerations when $d=2$. Furthermore, the case $l=2$ is quite degenerate as all of $\diag{i}$, $\updiag{i}{x}$ and $\downdiag{i}{y}$ are the same and since the parity of $i$ does not matter and so it is dealt with in a separate small section.

Although again there is some notation in this section, it is much lighter than Section \ref{Section:strongly maximal td-rigid}. Still the reader is reminded to make use of the index and the pictures at the end of the article in case it is needed.

\subsection{Preliminaries}
Let $d\geq 1$ be a positive integer. We start by recalling some notions related to $d$-torsion classes from \cite{august2023characterisation}. Let $\mathcal{A}$ be an additive category. Let $f:A\to B$ be a morphism in $\mathcal{A}$. Then $f$ is called \emph{left minimal}\index[definitions]{left minimal morphism} if any endomorphism $g$ of $B$ satisfying $g\circ f=f$ is an isomorphism. A morphism $g:B\to C$ in $\mathcal{A}$ is called a \emph{weak cokernel of $f$}\index[definitions]{weak cokernel} if $g\circ f=0$ and for any $h:B\to D$ with $h\circ f=0$, there exists a morphism $j:C\to D$ such that $j\circ g=h$; if moreover the morphism $j$ is unique, then $g:B\to C$ is called a \emph{cokernel}\index[definitions]{cokernel} and $C$ is unique up to an isomorphism. We say that a morphism $g:B\to C$ is a \emph{weak cokernel in $\mathcal{A}$} if there exists another morphism $f:A\to B$ in $\mathcal{A}$ such that $g$ is a weak cokernel of $f$. A \emph{$d$-cokernel}\index[definitions]{$d$-cokernel} of a morphism $f_0:A_0\to A_1$ is a sequence of morphisms
\[
\begin{tikzcd}
	A_1 & A_2 & \cdots & A_d & A_{d+1}\mathrlap{,}
	  \arrow["f_1",from=1-1, to=1-2]
        \arrow["f_2", from=1-2, to=1-3]
        \arrow["f_{d-1}", from=1-3, to=1-4]
        \arrow["f_{d}", from=1-4, to=1-5]
\end{tikzcd}
\]
such that $f_i$ is a weak cokernel of $f_{i-1}$ for $i\in\{1,\ldots,d-1\}$ and $f_d$ is a cokernel of $f_{d-1}$. The notions of \emph{right minimal}\index[definitions]{right minimal morphism}, \emph{weak kernel}\index[definitions]{weak kernel}, \emph{kernel}\index[definitions]{kernel} and \emph{$d$-kernel}\index[definitions]{$d$-kernel} are defined dually. A sequence
\[
\begin{tikzcd}
	0 & A_0 & A_1 & \cdots & A_d & A_{d+1} & 0 
	  \arrow[from=1-1, to=1-2]
        \arrow["f_0", from=1-2, to=1-3]
        \arrow["f_1", from=1-3, to=1-4]
        \arrow["f_{d-1}", from=1-4, to=1-5]
        \arrow["f_d", from=1-5, to=1-6]
        \arrow[from=1-6, to=1-7]
\end{tikzcd}
\]
is called a \emph{$d$-extension}\index[definitions]{$d$-extension} if $(f_1,\ldots,f_{d})$ is a $d$-cokernel of $f_0$ and $(f_0,\ldots,f_{d-1})$ is a $d$-kernel of $f_d$. Such a $d$-extension is said to be \emph{equivalent}\index[definitions]{$d$-extension, equivalent} to a $d$-extension 
\[
\begin{tikzcd}
	0 & A_0 & B_1 & \cdots & B_d & A_{d+1} & 0\mathrlap{,} 
        \arrow[from=1-1, to=1-2]
        \arrow["g_0", from=1-2, to=1-3]
        \arrow["g_1", from=1-3, to=1-4]
        \arrow["g_{d-1}", from=1-4, to=1-5]
        \arrow["g_d", from=1-5, to=1-6]
        \arrow[from=1-6, to=1-7]
\end{tikzcd}
\]
if there exists a commutative diagram
\[
\begin{tikzcd}
	0 & A_0 & A_1 & \cdots & A_d & A_{d+1} & 0 \\
	0 & A_0 & B_1 & \cdots & B_d & A_{d+1} & 0\mathrlap{.} 
        \arrow[from=1-1, to=1-2]
        \arrow["f_0", from=1-2, to=1-3]
        \arrow["f_1", from=1-3, to=1-4]
        \arrow["f_{d-1}", from=1-4, to=1-5]
        \arrow["f_d", from=1-5, to=1-6]
        \arrow[from=1-6, to=1-7]
        \arrow[from=2-1, to=2-2]
        \arrow["g_0", from=2-2, to=2-3]
        \arrow["g_1", from=2-3, to=2-4]
        \arrow["g_{d-1}", from=2-4, to=2-5]
        \arrow["g_d", from=2-5, to=2-6]
        \arrow[from=2-6, to=2-7]
        \arrow[from=1-3, to=2-3]
        \arrow[from=1-5, to=2-5]
        \arrow[equal, from=1-2, to=2-2]
        \arrow[equal, from=1-6, to=2-6]
\end{tikzcd}
\]
Let $\Lambda$ be a finite-dimensional algebra and let $\cC\subseteq \modfin{\Lambda}$ be a $d$-cluster tilting subcategory for some $d\geq 1$. Then the notion of equivalence of $d$-extensions in $\cC$ is an equivalence relation by \cite[Proposition 4.10]{jasso2016n}. Generalizing the classical notion of torsion classes to higher homological algebra, Jørgensen introduced the following notion.

\begin{definition}\cite[Definition 1.1]{Jorgensen2014}\label{def:d-torsion class}
We say that $\cU\subseteq \cC$ is a \emph{$d$-torsion class}\index[definitions]{$d$-torsion class} if for every $C\in\cC$ there exists a $d$-extension
\[
\begin{tikzcd}
	0 & U & C & C_1 & \cdots & C_d & 0 
	  \arrow[from=1-1, to=1-2]
        \arrow["u", from=1-2, to=1-3]
        \arrow["c_0", from=1-3, to=1-4]
        \arrow["c_1", from=1-4, to=1-5]
        \arrow["c_{d-1}", from=1-5, to=1-6]
        \arrow[from=1-6, to=1-7]
\end{tikzcd}
\]
such that $U\in\cU$ and, for every $U'\in\cU$, the induced sequence
\[
\begin{tikzcd}
	0 & \Hom{\Lambda}{U'}{C_1} & \cdots & \Hom{\Lambda}{U'}{C_d} & 0 
	  \arrow[from=1-1, to=1-2]
        \arrow[from=1-2, to=1-3]
        \arrow[from=1-3, to=1-4]
        \arrow[from=1-4, to=1-5]
\end{tikzcd}
\]
is exact.
\end{definition}

Note that a $d$-torsion class is closed under direct sums and summands by \cite[Lemma 2.7(iii)]{Jorgensen2014}.

Recently, an equivalent characterization of $d$-torsion classes was given in \cite{august2023characterisation}. We introduce some more notions necessary to recall that result.

Let $\cU\subseteq \cC$ be a subcategory. We say that $\cU$ is \emph{closed under $d$-quotients}\index[definitions]{closure under $d$-quotients} if for any morphism $f:M\to U$ with $M\in\cC$ and $U\in\cU$, there exists a $d$-cokernel 
\[
\begin{tikzcd}
	M & U & U_1 & \cdots & U_d & 0 
	  \arrow["f", from=1-1, to=1-2]
        \arrow["f_1", from=1-2, to=1-3]
        \arrow["f_2", from=1-3, to=1-4]
        \arrow["f_{d}", from=1-4, to=1-5]
        \arrow[from=1-5, to=1-6]
\end{tikzcd}
\]
of $f$, where $U_1,\ldots,U_d\in\cU$. We recall the following results from \cite{august2023characterisation}.

\begin{lemma}\cite[Construction 2.9]{august2023characterisation}\label{lem:how to get weak cokernels}
    Let $f:X\to Y$ be a morphism in $\cC$. Let $c:Y\to M$ be the cokernel of $f$ in $\modfin{\Lambda}$. Let $g:M\to C$ be a minimal left $\cC$-approximation of $M$. Then $g\circ c$ is a left minimal weak cokernel of $f$.
\end{lemma}

\begin{corollary}\label{cor:the usual weak cokernel}
    Assume that $X\in\cC$. Let $\pi_X:X\to X/\Socle{X}$ be the canonical epimorphism and let $g_0:X/\Socle{X}\to C$ be the minimal left $\cC$-approximation of $X/\Socle{X}$. Then $g=g_0\circ \pi_X$ is a left minimal weak cokernel in $\cC$.
\end{corollary}

\begin{proof}
    Let $\pi:P\to \Socle{X}$ be the projective cover of $\Socle{X}$ and $\iota:\Socle{X}\to X$ be the canonical inclusion. Then $\pi_X:X\to X/\Socle{X}$ is the cokernel of $\iota\circ\pi$, and so the claim follows by Lemma \ref{lem:how to get weak cokernels}.
\end{proof}

\begin{lemma}\cite[Lemma 3.14]{august2023characterisation}\label{lem:closed under minimal weak cokernel}
    Let $\mathcal{U}\subseteq \C$ be a subcategory which is closed under $d$-quotients. Let $g:M\to N$ be a left minimal weak cokernel in $\C$. If $M\in \cU$, then $N\in\cU$.
\end{lemma}

We say that $\cU$ is \emph{closed under $d$-extensions}\index[definitions]{closure under $d$-extensions} if for any $d$-extension
\[
\begin{tikzcd}
	0 & U & C_1 & \cdots & C_d & V & 0 
	  \arrow[from=1-1, to=1-2]
        \arrow["f_0",from=1-2, to=1-3]
        \arrow["f_1",from=1-3, to=1-4]
        \arrow["f_{d-1}",from=1-4, to=1-5]
        \arrow["f_d",from=1-5, to=1-6]
        \arrow[from=1-6, to=1-7]
\end{tikzcd}
\]
in $\cC$ with $U,V\in\cU$ there exists an equivalent $d$-extension
\[
\begin{tikzcd}
	0 & U & U_1 & \cdots & U_d & V & 0 
	  \arrow[from=1-1, to=1-2]
        \arrow["g_0",from=1-2, to=1-3]
        \arrow["g_1",from=1-3, to=1-4]
        \arrow["g_{d-1}",from=1-4, to=1-5]
        \arrow["g_d",from=1-5, to=1-6]
        \arrow[from=1-6, to=1-7]
\end{tikzcd}
\]
where $U_1,\ldots,U_d\in \cU$. With this, we can now recall the following from \cite{august2023characterisation}.

\begin{theorem}\cite[Theorem 1.1]{august2023characterisation}\label{thm:d-torsion by AHJKPT} 
A subcategory $\cU\subseteq \cC$ is a $d$-torsion class if and only if it is closed under $d$-quotients, $d$-extensions and direct summands.
\end{theorem}

In fact, we want to use a slightly modified but equivalent description of $d$-torsion classes. A $d$-extension
\[
\begin{tikzcd}
	0 & U & C_1 & \cdots & C_d & V & 0 
	  \arrow[from=1-1, to=1-2]
        \arrow["f_0",from=1-2, to=1-3]
        \arrow["f_1",from=1-3, to=1-4]
        \arrow["f_{d-1}",from=1-4, to=1-5]
        \arrow["f_d",from=1-5, to=1-6]
        \arrow[from=1-6, to=1-7]
\end{tikzcd}
\]
in $\cC$ is called \emph{minimal}\index[definitions]{$d$-extension!minimal} if $f_i\in \mathrm{Rad}_{\modfin{\Lambda}}{(C_i,C_{i+1})}$ for $i=1,\ldots,d-1$, where
\[
\mathrm{Rad}_{\modfin{\Lambda}}{(M,N)}= \{f\in\Hom{\Lambda}{M}{N} \mid 1_X-g\circ f \text{ is invertible for any $g\in\Hom{\Lambda}{N}{M}$}\}
\]
defines the \emph{Jacobson radical}\index[definitions]{Jacobson radical of module category} of $\modfin{\Lambda}$. We say that $\cU$ is \emph{closed under minimal $d$-extensions}\index[definitions]{closure under minimal $d$-extensions} if for any minimal $d$-extension
\[
\begin{tikzcd}
	0 & U & C_1 & \cdots & C_d & V & 0 
	  \arrow[from=1-1, to=1-2]
        \arrow["f_0",from=1-2, to=1-3]
        \arrow["f_1",from=1-3, to=1-4]
        \arrow["f_{d-1}",from=1-4, to=1-5]
        \arrow["f_d",from=1-5, to=1-6]
        \arrow[from=1-6, to=1-7]
\end{tikzcd}
\]
in $\cC$ with $U,V\in\cU$, we have that $C_1,\ldots,C_d\in\cU$. We then have the following.

\begin{proposition}\label{prop:description of d-torsion classes}
    $\cU\subseteq\cC$ is a $d$-torsion class if and only if it is closed under  $d$-quotients, minimal $d$-extensions and direct summands.
\end{proposition}

\begin{proof}
By \cite[Proposition 2.4]{Herschend-Jorgensen} we have that every $d$-extension in $\cC$ is equivalent to a minimal $d$-extension and, moreover, that this minimal $d$-extension appears as a direct summand in any $d$-extension in its equivalence class. The claim follows.
\end{proof}

For further details about $d$-torsion classes we refer to \cite{Jorgensen2014} and \cite{august2023characterisation}.

For the rest of this section we fix a finite-dimensional algebra $\Lambda=\Lambda(n,l)$ which admits a $d$-cluster tilting subcategory $\cC\subseteq \modfin{\Lambda}$ for some integer $d\geq 2$. Our aim is to classify the $d$-torsion classes of $\Lambda$. As $d$-torsion classes are closed under direct sums and summands, for the rest of this section all subcategories are closed under direct sums. In practice, this means that we describe subcategories using the indecomposable modules lying in them.

We find the classification of $d$-torsion classes to be somewhat less complicated in comparison to our previous classification of $\td$-tilting pairs. Moreover, many of the computations are of a similar flavor to the ones already performed in this article, and so we leave some details to the reader.

We may notice that any nonzero endomorphism of an indecomposable module in $\modfin{\Lambda}$ is an isomorphism by Lemma \ref{Lemma:HomSpaceAdachi}. It follows that any nonzero morphism $f:M\to N$ with $M$ and $N$ indecomposable is left minimal. We use this fact throughout.

\subsection{Weak cokernels in \texorpdfstring{$\cC$}{C}}
\label{subsec:weak cokernels in C}
In this small section we compute some weak cokernels in $\cC$. 

\begin{lemma}\label{lem:nonzero morphisms in C}
    Let $M=\ind{a}{b}\in\modfin{\Lambda}$ and $C\in\cC$ be indecomposable. Assume that $\Hom{\Lambda}{M}{C}\neq 0$. Then one of the following holds.
    \begin{enumerate}
        \item[(a)] $C$ is isomorphic to one of the indecomposable injective modules $I(a+j)$ for $0\leq j\leq b-a$.
        \item[(b)] $C\in\diag{i}$ for some odd $i$ and $s_{i}-l+1\leq a\leq s_i \leq b\leq s_i+l-1$.
        \item[(c)] $C\in\diag{i}$ for some even $i$ and $s_{i}-l+1\leq a\leq b\leq s_i$.
    \end{enumerate}
\end{lemma}

\begin{proof}
    This follows by a direct computation using the description of $\cC$ (see for example (\ref{eq:description of C})).
\end{proof}

\begin{lemma}\label{lem:minimal left approximation in C}
    Let $M=\ind{a}{b}\in\modfin{\Lambda}$ be indecomposable and assume that $\ind{a}{b}\not\in\cC$. Let $I=\ind{a}{a+l-1}$ be the injective envelope of $M$ and $\iota:M\to I$ be the canonical inclusion. Let $i\in [1,p]$ be minimal such that $\Hom{\Lambda}{M}{\diag{i}}\neq 0$, if such $i$ exists, and $i=0$ otherwise. Then exactly one of the following hold.
    \begin{enumerate}
        \item[(a)] $i=0$. Then $\iota:M\to I$ is a minimal left $\cC$-approximation.
        \item[(b)] $i\neq 0$ is odd. Then $N=\ind{s_i}{b}$ is in $\diag{i}$ and $\iota\oplus\pi_N:  M\to I\oplus N$ is a minimal left $\cC$-approximation, where $\pi_N:M\to N$ is the canonical projection.
        \item[(c)] $i\neq 0$ is even and $N=\ind{a}{s_i}$ is in $\diag{i}$. Then the canonical inclusion $\iota_N:M\to N$ is a minimal left $\cC$-approximation.
        \item[(d)] $i\neq 0$ is even and $\ind{a}{s_i}$ is not in $\diag{i}$. Then $\iota:M\to I$ is a minimal left $\cC$-approximation.
    \end{enumerate}
\end{lemma}

\begin{proof}
    This is a direct computation using Lemma \ref{lem:nonzero morphisms in C} and the description of $\cC$.
\end{proof}

\tikzcdset{scale cd/.style={every label/.append style={scale=#1},
    cells={nodes={scale=#1}}}}

\begin{lemma}\label{lem:n-extensions in C}
    \begin{enumerate}
        \item[(a)] Let $2\leq i\leq p$ be even and $M=\ind{a}{s_i}\in\diag{i}$. Let $j\in\{1,\ldots,s_i-a+1\}$ and
        \[
        N = \ind{s_i+\tfrac{d-2}{2}l+2}{a+\tfrac{d}{2}l-1+j}.
        \]
        Then there exists in $\cC$ a minimal $d$-extension of the form

\[
\begin{tikzpicture}[every node/.style={scale=.85},xscale=.85,yscale=.75]
    \node (zero1) at (-.2,0) []{$0$};
    \node (M1) at (1,0) [] {$M$};
    \node (M2) at (3.5,0) [] {$I(a)\oplus \ind{a+j}{s_i}$};
    \node (M3) at (6.5,0) [] {$I(a+j)$};
    \node (M4) at (8.7,0) [] {$I(a+l)$};
    \node (M5) at (11,0) [] {$I(a+l+j)$};
    \node (dots) at (13,0) [] {$\dots$};
    \node (M6) at (15.3,0) [] {$I(a+\tfrac{d-4}{2}l+j)$};
    \node (placeholder) at (1,-1) {{}};
    \node (M7) at (3.2,-1) [] {$I(a+\tfrac{d-2}{2}l+j)$};
    \node (M8) at (9,-1) [] {$I(a+\tfrac{d-2}{2}l+j)\oplus\ind{s_i+\tfrac{d-2}{2}l+2}{a+\tfrac{d}{2}l-1}$};
    \node (M9) at (13.5,-1) [] {$N$};
    \node (zero2) at (14.6,-1) []{$0$};

    \draw[right hook->] (zero1)--(M1);
    \draw[->] (M1)--(M2)node[midway,above,scale=.8]{$f_0$};
    \draw[->] (M2)--(M3)node[midway,above,scale=.8]{$f_1$};
    \draw[->] (M3)--(M4)node[midway,above,scale=.8]{$f_2$};
    \draw[->] (M4)--(M5)node[midway,above,scale=.8]{$f_3$};
    \draw[->] (M5)--(dots)node[midway,above,scale=.8]{$f_4$};
    \draw[->] (dots)--(M6)node[midway,above,scale=.8]{$f_{d-3}$};
    \draw[->] (placeholder)--(M7)node[midway,above,scale=.8]{$f_{d-2}$};
    \draw[->] (M7)--(M8)node[midway,above,scale=.8]{$f_{d-1}$};
    \draw[->] (M8)--(M9)node[midway,above,scale=.8]{$f_d$};
    \draw[->>] (M9)--(zero2);
\end{tikzpicture}
\]
    In particular, if $j=1$, then $\tau_d^{-}(M)=\ind{s_i+\tfrac{d-2}{2}l+2}{a+\tfrac{d}{2}l}$ and so this gives a minimal $d$-extension between $M$ and $\tau_d^{-}(M)$.
    \item[(b)] Let $1\leq i\leq p-1$ be odd and $M=\ind{s_i}{b}\in\diag{i}$. Let $j\in\{1,\ldots,l-1+s_i-b\}$ and
    \[
    N={\ind{b+\tfrac{d-2}{2}l+1+j}{s_i+\tfrac{d}{2}l}}.
    \] 
    Then there exists in $\cC$ a minimal $d$-extension of the form
    \[
\begin{tikzpicture}[every node/.style={scale=.85},xscale=.85,yscale=.8]
    \node (zero1) at (-.3,0)[] {$0$};
    \node (M1) at (.9,0) [] {$M$};
    \node (M2) at (3,0) [] {$\ind{s_i}{b+j}$};
    \node (M3) at (5.5,0) [] {$I(b+1)$};
    \node (M4) at (8,0) [] {$I(b+j+1)$};
    \node (M5) at (10.8,0) [] {$I(b+l+1)$};
    \node (dots) at (13,0) [] {$\dots$};
    \node (placeholder) at (.9,-1) {{}};
    \node (M6) at (3.2,-1) [] {$I(a+\tfrac{d-4}{2}l+1)$};
    \node (M7) at (7,-1) [] {$I(a+\tfrac{d-4}{2}l+j+1)$};
    \node (M8) at (11.5,-1) [] {$\ind{b+\tfrac{d-2}{2}l+1}{s_i+\tfrac{d}{2}l}$};
    \node (M9) at (14.6,-1) [] {$N$};
    \node (zero2) at (15.7,-1) [] {$0$};

    \draw[right hook->] (zero1)--(M1);
    \draw[->] (M1) -- (M2)node[midway,above,scale=0.8]{$f_0$};
    \draw[->] (M2) -- (M3)node[midway,above,scale=0.8]{$f_1$};
    \draw[->] (M3) -- (M4)node[midway,above,scale=0.8]{$f_2$};
    \draw[->] (M4) -- (M5)node[midway,above,scale=0.8]{$f_3$};
    \draw[->] (M5) -- (dots)node[midway,above,scale=0.8]{$f_4$};
    \draw[->] (placeholder) -- (M6)node[midway,above,scale=0.8]{$f_{d-3}$};
    \draw[->] (M6) -- (M7)node[midway,above,scale=0.8]{$f_{d-2}$};
    \draw[->] (M7) -- (M8)node[midway,above,scale=0.8]{$f_{d-1}$};
    \draw[->] (M8) -- (M9)node[midway,above,scale=0.8]{$f_d$};
    \draw[->>] (M9)--(zero2);
\end{tikzpicture}
\]
    In particular, if $j=1$, then $\tau_d^{-}(M)=\ind{b+\tfrac{d-2}{2}l+2}{s_i+\tfrac{d}{2}l}$ and so this gives a minimal $d$-extension between $M$ and $\tau_d^{-}(M)$. 
    \end{enumerate}
\end{lemma}

\begin{proof}
    We only show part (a) as part (b) can be shown similarly. 
    Notice that all the modules appearing in the sequence in (a) belong to $\cC$. Indeed, the injective modules clearly belong to $\cC$, the modules $M$ and $\ind{a+j}{s_i}$ belong to $\diag{i}$, while the modules $N$ and $\ind{s_i+\tfrac{d-2}{2}l+2}{a+\tfrac{d}{2}l-1}=\ind{s_{i+1}}{a+\tfrac{d}{2}l-1}$ belong to $\diag{i+1}$. Consider the following sequence of morphisms
 
    \[
    \begin{tikzpicture}[every node/.style={scale=.6},scale=.8]
        \node (leftEnd) at (-.6,0) []{$0$};
        \node (M1) at (1,0) [] {$M=\ind{a,s_i}$};
        \node (I1) at (2.5,1) [] {$I(a)$};
        \node (M2) at (2.5,-1) [] {$\ind{a+j}{s_i}$};
        \node (M3) at (4,0) [] {$\ind{a+j}{a+l-1}$};
        \node (I2) at (6.5,0) [] {$I(a+j)$};
        \node (M4) at (8,-1) [] {$\ind{a+l}{a+l+j-1}$};
        \node (I3) at (9.5,0) [] {$I(a+l)$};
        \node (M5) at (11.5,-1) [] {$\ind{a+l+j}{a+2l-1}$};
        \node (I4) at (13,0) [] {$I(a+l+j)$};
        \node (dots) at (15,0) [] {$\cdots$};

        \node (dots2) at (0,-4) [] {$\cdots$};
        \node (I5) at (2,-4) [] {$I(a+\tfrac{d-4}{2}l+j)$};
        \node (M6) at (4,-5) [] {$\ind{a+\tfrac{d-2}{2}l}{a+\tfrac{d-2}{2}l+j-1}$};
        \node (I6) at (5.5,-4) [] {$I(a+\tfrac{d-2}{2}l)$};
        \node (M7) at (9,-4) [] {$\ind{a+\tfrac{d-2}{2}l+j}{a+\tfrac{d}{2}l-1}$};
        \node (M8) at (11,-5) [] {$\ind{s_i+\tfrac{d-2}{2}l+2}{a+\tfrac{d}{2}l-1}$};
        \node (I7) at (11,-3) [] {$I(a+\tfrac{d-2}{2}l+j)$};
        \node (N) at (13,-4) [] {$N$};
        \node (N2) at (14,-4) [] {$0$};

        \draw[right hook->] (leftEnd)--(M1);
        
        \draw[right hook->] (M1)--(I1)node[midway,anchor=south east]{$-\iota_0$};
        \draw[->>] (M1)--(M2) node[midway,anchor=north east]{$\pi_0$};
        \draw[->>] (I1)--(M3);
        \draw[right hook->] (M2)--(M3);
        \draw[right hook->] (M3)--(I2);
        \draw[->] (I1) to[bend left]node[midway,above]{$f_1^1$} (I2);
        \draw[->] (M2) to[bend right]node[midway,below]{$f_1^2$} (I2);
        
        \draw[->>] (I2)--(M4);
        \draw[right hook->] (M4)--(I3);
        \draw[->] (I2)--(I3)node[midway, above]{$f_2$};

        \draw[->>] (I3)--(M5);
        \draw[right hook->] (M5)--(I4);
        \draw[->] (I3)--(I4)node[midway, above]{$f_3$};
        \draw[->] (I4)--(dots)node[midway,above]{$f_4$};
        
        \draw[->] (dots2)--(I5)node[midway,above] {$f_{d-3}$};
        \draw[->>] (I5)--(M6);
        \draw[right hook->] (M6)--(I6);
        \draw[->] (I5)--(I6)node[midway,above]{$f_{d-2}$};

        \draw[->>] (I6)--(M7);
        \draw[->] (I6) to[bend left]node[midway,above]{$f_{d-1}^1$} (I7);
        \draw[->] (I6) to[bend right]node[midway,below]{$f_{d-1}^2$} (M8);
        \draw[right hook->] (M7)--(I7);
        \draw[->>] (M7)--(M8);
        \draw[right hook->] (I7)--(N)node[midway,anchor=south west]{$-\pi_d$};
        \draw[->>] (M8)--(N) node[midway,anchor=south east]{$\iota_d$};
        \draw[->>] (N)--(N2);
    \end{tikzpicture}
    \]
    where each morphism not labeled by $f$ is a canonical (up to the noted signs) inclusion or projection, and the morphisms labeled by $f$ are the compositions of these morphisms. Set $f_0=(-\iota_0)\oplus \pi_0$, set $f_1=f_{1}^{1}\oplus f_{1}^{2}$, set $f_{d-1}=f_{d-1}^{1}\oplus f_{d-1}^{2}$ and set $f_d=(-\pi_d)\oplus \iota_d$. Using Lemma \ref{lem:minimal left approximation in C} it is easy to check that for $1\leq i\leq d$, each $f_i$ decomposes as the composition of the cokernel of $f_{i-1}$ with a minimal left $\cC$-approximation of this cokernel. By Lemma \ref{lem:how to get weak cokernels} we have that $(f_1,\dots,f_d)$ is a $d$-cokernel of $f_0$. Dually it can be shown that $(f_1,\dots,f_{d-1})$ is a $d$-kernel of $f_d$. This shows that the given sequence is a $d$-extension. Moreover, there are no nonzero morphism in the given sequence going the opposite direction, which shows that each $f_i$ is a radical morphism. This shows that this $d$-extension is minimal. 
\end{proof}

\begin{example}\label{Example:minimal d-extensions} Let $d=4$ and $l=4$. The following picture give examples of minimal $d$-extensions between consecutive diagonals in this case which can be obtained via Lemma \ref{lem:n-extensions in C}. 
\[
\includegraphics{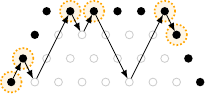}
\quad
\includegraphics{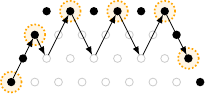}
\quad
\includegraphics{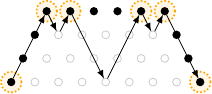}
\]
\[
\includegraphics{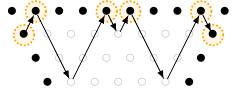}
\quad
\includegraphics{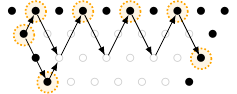}
\quad
\includegraphics{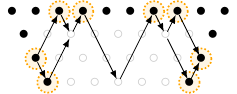}
\]
The modules appearing in the minimal $d$-extensions are given as yellow-encircled points  and the morphisms are obtained by composing the morphisms in the pictures.
\end{example}

\subsection{Classification of \texorpdfstring{$d$}{d}-torsion classes}

Let $\cU\subseteq \cC$ be an (additive and full) subcategory. Let $1\leq i\leq p$. We set $\cU_i=\cU\cap\diag{i}$\index[symbols]{U@$\cU_i$}. We further set $Q_{\cU}\subseteq [1,n-(l-1)]$\index[symbols]{Q@$Q_{\cU}$} to be the index set of the projective-injective indecomposable modules in $\cU$, that is, for $q\in [1,n-l+1]$ we have that $I(q)\in\cU$ if and only if $q\in Q_{\cU}$. Then we have that
\[
\cU = \add{\cU_1,\ldots,\cU_p,I(Q_{\cU})}.
\]
Hence a subcategory $\cU\subseteq \cC$ is described uniquely via the data $\cU_1,\ldots,\cU_p, Q_{\cU}$. Using the results of paragraph \ref{subsec:weak cokernels in C}, we can obtain a lot of information about the data $\cU_1,\ldots,\cU_p,Q_{\cU}$ when $\cU$ is a $d$-torsion class. The following three lemmas collect all of this information.

For the convenience of the reader, we recall that $\downdiag{i}{h}$ (respectively $\updiag{i}{h}$) is the additive closure of all indecomposable modules in $\diag{i}$ with length at most (respectively at least) $h$. Recall also that if $i$ is odd (respectively even), then $\ind{s_i}{s_i+h-1}$ (respectively $\ind{s_i-h+1}{s_i}$) is the unique indecomposable module in $\diag{i}$ of length $h$.

\begin{lemma}\label{lem:closure under d-quotients} Let $\cU$ be a $d$-torsion class in $\cC$ and let $1\leq i\leq p$.
\begin{enumerate}
    \item[(a)] Assume that $i$ is even and let $h\in\{1,\ldots,l-1\}$. If $\ind{s_i-h+1}{s_i}\in\cU_i$, then $\downdiag{i}{h}\subseteq \cU_i$.
    \item[(b)] Assume that $i$ is odd and let $I(q)$ be an indecomposable injective $\Lambda$-module with $s_i\leq q \leq s_{i+1}-(l-1)$. If $I(q)\in\cU$, then $[q,s_{i+1}-(l-1)]\subseteq Q_{\cU}$ and $\cU_{i+1}=\diag{i+1}$.
    \item[(c)] Assume that $i$ is even and let $I(q)$ be an indecomposable injective $\Lambda$-module with $s_i-(l-2)\leq q\leq s_{i+1}-1$. If $I(q)\in\cU$, then $[q,s_{i+1}-1]\subseteq Q_{\cU}$. If moreover $s_{i+1}-(l-1)\leq q$, then $\updiag{i+1}{h}\subseteq \cU_{i+1}$ where $h=l+q-s_{i+1}$.
    \item[(d)] Assume that $i$ is odd and let $\ind{s_i}{b}\in\diag{i}$ be indecomposable with $b>s_{i}$. If $\ind{s_i}{b}\in\cU_i$, then $[s_{i},s_{i+1}-(l-1)]\subseteq Q_{\cU}$ and $\cU_{i+1}=\diag{i+1}$.
\end{enumerate}
\end{lemma}

\begin{proof}
\begin{enumerate}
    \item[(a)] If $l=2$ there is nothing to show. Otherwise let $a=s_i-h+1$ and it is enough to show that $\ind{a+1}{s_i}\in\cU_i$. We have that $\ind{a+1}{s_i}=\ind{a}{s_i}/\Socle{\ind{a}{s_i}}$. Moreover, since $\ind{a+1}{s_i}\in\cC$, we have that the minimal left $\cC$-approximation of $\ind{a+1}{s_i}$ is the identity morphism. It follows by Corollary \ref{cor:the usual weak cokernel} that the canonical epimorphism $\ind{a}{s_i}\to \ind{a+1}{s_i}$ is a left minimal weak cokernel in $\cC$. We conclude that $\ind{a+1}{s_i}\in\cU_i$ by Lemma \ref{lem:closed under minimal weak cokernel}.

    \item[(b)] We first show the claim for $Q_{\cU}$. It is enough to show that $I(q+1)\in\cU$ when $q+1\leq s_{i+1}-(l-1)$. We have that $I(q)/\Socle{I(q)}=\ind{q+1}{q+l-1}$. Let $\pi:I(q)\to \ind{q+1}{q+l-1}$ be the canonical epimorphism. Using $q+1\leq s_{i+1}-(l-1)$ it is easy to see that $\ind{q+1}{q+l-1}$ satisfies the assumptions of either (a) or (d) of Lemma \ref{lem:minimal left approximation in C}. Hence Lemma \ref{lem:minimal left approximation in C} gives that the canonical inclusion $\iota:\ind{q+1}{q+l-1}\to I(q+1)$ is a minimal left $\cC$-approximation of $\ind{q+1}{q+l-1}$. It follows by Corollary \ref{cor:the usual weak cokernel} that $g=\iota\circ\pi$ is a left minimal weak cokernel in $\cC$. We conclude that $I(q+1)\in\cU$ by Lemma \ref{lem:closed under minimal weak cokernel}.
    
    Next notice that $I(s_{i+1}-(l-1))$ admits an epimorphism to any indecomposable module in $\diag{i+1}$ and so, similarly to part (a), we obtain that $\diag{i+1}\subseteq\cU_{i+1}$. Since $\cU_{i+1}\subseteq \diag{i+1}$ always holds, the claim follows.
    
    \item[(c)] We first show the claim for $Q_{\cU}$. If $q< s_{i+1}-(l-1)$, then the proof is similar to the previous case. Hence we may assume that 
    \begin{equation}\label{eq:the condition on q}
    s_{i+1}-(l-1)\leq q < q+1 \leq s_{i+1}-1.
    \end{equation}
    Again we let $\pi:I(q)\to \ind{q+1}{q+l-1}$ be the canonical epimorphism of $I(q)$ to $I(q)/\Socle{I(q)}$ and we let $\iota:\ind{q+1}{q+l-1}\to I(q+1)$ be the canonical inclusion. Using (\ref{eq:the condition on q}) it is easy to see that $\ind{q+1}{q+l-1}$ satisfies the assumptions of (b) of Lemma \ref{lem:minimal left approximation in C}. Let $N=\ind{s_{i+1}}{q+l-1}$ and let $\pi_N:\ind{q+1}{q+l-1}\to N$ be the canonical projection. Then Lemma \ref{lem:minimal left approximation in C} gives that $\iota\oplus \pi_N: \ind{q+1}{q+l-1}\to I(q+1)\oplus N$ is a minimal left $\cC$-approximation of $\ind{q+1}{q+l-1}$. We conclude that $I(q+1)\in\cU$ by Lemma \ref{lem:closed under minimal weak cokernel}. 
    
    Next notice that the above also shows that $\ind{s_{i+1}}{q'+l-1}\in\cU$ for $q'\in [q,s_{i+1}-1]$ and so $\updiag{i+1}{l+q-s_{i+1}}\subseteq \cU_{i+1}$ as well. 
    
    \item[(d)] Since $b>s_i$, we have that $\ind{s_i}{b}$ is not simple. Hence it is easy to see using Lemma \ref{lem:minimal left approximation in C} that $I(s_i+1)$ is a minimal left $\cC$-approximation of $\ind{s_i}{b}\to \ind{s_i}{b}/S(s_i)$. Thus $I(s_{i}+1)$ is a weak cokernel in $\cC$ and so $I(s_{i}+1)\in\cU$ by Lemma \ref{lem:closed under minimal weak cokernel}. By part (b) we obtain that $[s_{i}+1,s_{i+1}-(l-1)]\subseteq Q_{\cU}$ and that $\cU_{i+1}=\diag{i+1}$. It remains to show that $I(s_{i})=\ind{s_i}{s_i+l-1}\in\cU$. Set $j=l-1+s_i-b$. Then by Lemma \ref{lem:n-extensions in C}(b) we obtain a minimal $d$-extension between $M=\ind{s_i}{b}$ and $N=\ind{b+\tfrac{d-2}{2}l+1+j}{s_i+\tfrac{d}{2}l}$. By Proposition \ref{prop:description of d-torsion classes} we have that all middle terms in this minimal $d$-extension are in $\cU$, and so we obtain that $\ind{s_i}{b+j}=\ind{s_i}{s_i+l-1}=I(s_i)\in\cU$, as required.\qedhere
\end{enumerate}
\end{proof}

\begin{lemma}\label{lem:effect of M and its taud inverse}
    Let $\cU$ be a $d$-torsion class and let $1\leq i\leq p-1$. 
    \begin{enumerate}
        \item[(a)] Assume that $i$ is even. Assume also that $M=\ind{s_i-h+1}{s_i}\in\cU_i$ for some $h\in\{1,\ldots,l-1\}$ and that $\tau_d^{-}(M)\in\cU_{i+1}$. Then $\downdiag{i}{h}\subseteq\cU_i$ and $[s_i-h+1,s_{i+1}-1]\subseteq Q_{\cU}$. Moreover, we also have that $\updiag{i+1}{u}\subseteq \cU_{i+1}$ where
        \[
        u = \begin{cases}
            1, &\mbox{if $d>2$ or $h\in\{l-2,l-1\}$,} \\
            l-h-1, &\mbox{if $d=2$ and $h\leq l-3$.}
        \end{cases}
        \]
        \item[(b)] Assume that $i$ is odd. Assume also that $M=\ind{s_i}{s_i+h-1}\in\cU_i$ for some $h\in\{1,\ldots,l-1\}$ and that $M'=\ind{s_{i+1}-h'+1}{s_{i+1}}\in \cU_{i+1}$ for some $h'\in \{1,\ldots,l-h\}$. Then $\updiag{i}{h}\subseteq \cU_i$, $\cU_{i+1}=\diag{i+1}$ and $[s_i,s_{i+1}-(l-1)]\subseteq Q_{\cU}$.
    \end{enumerate}
\end{lemma}

\begin{proof}
    \begin{enumerate}
        \item[(a)] By Lemma \ref{lem:closure under d-quotients}(a) we have that $\downdiag{i}{h}\subseteq \cU_i$. Since $M,\tau_d^{-}(M)\in\cU$ and $\cU$ is closed under minimal $d$-extensions by Proposition \ref{prop:description of d-torsion classes}, Lemma \ref{lem:n-extensions in C}(a) gives that $I(s_i-h+1)\in \cU$. By Lemma \ref{lem:closure under d-quotients}(c) we conclude that $[s_i-h+1,s_{i+1}-1]\subseteq Q_{\cU}$. 
        
        Now let 
        \[
        q = \begin{cases}
            s_{i+1}-(l-1), &\mbox{if $d>2$ or $h\in\{l-2,l-1\}$,} \\
            s_i-h+1, &\mbox{if $d=2$ and $h\leq l-3$.}
        \end{cases}
        \]
        We claim that $s_i-h+1\leq q$. Clearly this holds in the second case. In the first case, using (\ref{eq:difference of consecutive simples}), we have that $s_i-h+1\leq q$ if and only if
        \[
        h+2+\frac{d-2}{2}l\geq l.
        \]
        This inequality clearly holds if $d\geq 4$ or $h=l-2$ or $h=l-1$. It also holds if $d=2$ since in this case $l=2$. Hence the claim is shown and $s_i-h+1\leq q$. Similarly we can show that $s_{i+1}-(l-1)\leq q$. Then the second part of Lemma \ref{lem:closure under d-quotients}(c) gives that $\updiag{i+1}{l+q-s_{i+1}}\subseteq \cU_{i+1}$. Noticing that $u=l+q-s_{i+1}$ completes the proof.
        \item[(b)] Set $b=s_i+h-1$ and $j=s_i-b+l-h'$. Using $1\leq h'\leq l-h$ we can show that $1\leq j\leq s_i-b+l-1$. By Lemma \ref{lem:n-extensions in C}(b) there exists a minimal $d$-extension between $\ind{s_i}{b}$ and 
        \[
        \ind{b+\tfrac{d-2}{2}l+1+j}{s_i+\tfrac{d}{2}l}=
        \ind{s_{i+1}-h'+1}{s_{i+1}}=M'.
        \]
        Since $M,M'\in\cU$ and $\cU$ is closed under minimal $d$-extensions by Proposition \ref{prop:description of d-torsion classes}, we obtain using the minimal $d$-extension in \ref{lem:n-extensions in C}(b) that $\ind{s_i}{b+j}=\ind{s_i}{s_i+l-h'}\in \cU_i$. Since $1\leq h\leq l-h'$, we have that $s_i+l-h'>s_i$ and so Lemma \ref{lem:closure under d-quotients}(d) gives that $[s_{i},s_{i+1}-(l-1)]\subseteq Q_{\cU}$ and $\cU_{i+1}=\diag{i+1}$. 
        
        It remains to show that $\updiag{i}{h}\subseteq\cU_i$. Let $\ind{s_i}{s_i+h+k-1}\in\updiag{i}{h}$ for some $k\in\{0,1,\ldots,l-1-h\}$ and it is enough to show that $\ind{s_i}{s_i+h-1+k}\in\cU_i$. Then $M''=\ind{b+\tfrac{d-2}{2}l+1+k}{s_i+\tfrac{d}{2}l}\in\diag{i+1}=\cU_{i+1}$. By Lemma \ref{lem:n-extensions in C}(b) we have a minimal $d$-extension between $M$ and $M''$ and so by Proposition \ref{prop:description of d-torsion classes} we obtain that $\ind{s_i}{b+k}=\ind{s_i}{s_i+h-1+k}\in\cU_i$, as required. \qedhere
    \end{enumerate}
\end{proof}

\begin{lemma}\label{lem:possible Uis}
    Let $\cU\subseteq \cC$ be a $d$-torsion class and let $1\leq i\leq p$.
    \begin{enumerate}
        \item[(a)] Assume $i$ is even. Then $\cU_i=\downdiag{i}{h}$ for some $h\in\{0,1,\ldots,l-1\}$.
        \item[(b)] Assume $i$ is odd. Then $\cU_i=\downdiag{i}{0}$ or $\cU_i=\downdiag{i}{1}$ or $\cU_i=\updiag{i}{h}$ for some $h\in\{1,\ldots,l-1\}$.
    \end{enumerate}
\end{lemma}

\begin{proof}
    \begin{enumerate}
        \item[(a)] If $\cU_i=\{0\}$, then $\cU_i=\downdiag{i}{0}$. Assume now that $\cU_{i}\neq \{0\}$. Let $\ind{s_i-h+1}{s_i}\in\cU_i$ be the indecomposable module of largest length $h$ in $\cU_i$. Then $h\in\{1,\ldots,l-1\}$ and $\cU_i\subseteq \downdiag{i}{h}$. By Lemma \ref{lem:closure under d-quotients}(a) we also have that $\downdiag{i}{h}\subseteq \cU_i$ and so $\cU_i=\downdiag{i}{h}$.
        
        \item[(b)] If $\cU_i=\{0\}$, then $\cU_i=\downdiag{i}{0}$. Assume now that $\cU_{i}\neq \{0\}$. Let $\ind{s_i}{s_i+h-1}\in\cU_i$ be the indecomposable module of smallest length $h$ in $\cU_i$. Then $h\in\{1,\ldots,l-1\}$ and $\cU_i\subseteq \updiag{i}{h}$. We consider the cases $h>1$ and $h=1$ separately.
        
        If $h>1$, then by Lemma \ref{lem:closure under d-quotients}(d) we also have that $\cU_{i+1}=\diag{i+1}$. In particular, we obtain that $\tdo(\ind{s_i}{s_i+h-1})\in\cU$. Then Lemma \ref{lem:effect of M and its taud inverse}(b) gives that $\updiag{i}{h}\subseteq \cU_i$ and so $\cU_i=\updiag{i}{h}$. 
        
       If $h=1$, then $\ind{s_i}{s_i}\in\cU_i$. If $\cU_i=\downdiag{i}{1}$ then we are done. Now assume that there exists $h'>1$ such that $\ind{s_i}{s_i+h'-1}\in\cU_i$. Then Lemma \ref{lem:closure under d-quotients}(d) gives that $\diag{i+1}=\cU_{i+1}$. In particular, $\tdo(\ind{s_i}{s_i})\in\diag{i+1}=\cU_{i+1}$. Since we also have $\ind{s_i}{s_i}\in\cU_i$, Lemma \ref{lem:effect of M and its taud inverse}(b) gives that $\updiag{i}{1}\subseteq \cU$, and so we conclude that $\cU_i=\updiag{i}{1}$. \qedhere
    \end{enumerate}
\end{proof}

We illustrate how we can use the previous lemmas via the following example.

\begin{example}\label{Example:how to use lemmas for d-torsion}
Let $\Lambda=\Lambda(23,4)$ and $d=4$ so that $\Lambda$ admits a $4$-cluster tilting subcategory $\cC$. Let us construct the smallest $4$-torsion class $\cU\subseteq \cC$ that contains the following encircled modules:
\[
\includegraphics{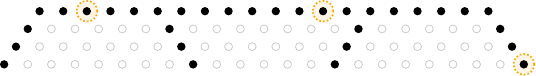}
\]
Applying Lemma \ref{lem:closure under d-quotients}(b) and (c) we obtain that the following modules are also included in $\cU$:
\[
\includegraphics{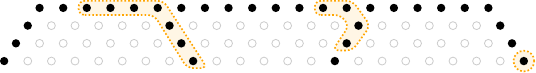}
\]
Next, applying Lemma \ref{lem:closure under d-quotients}(d) and then Lemma \ref{lem:closure under d-quotients}(c)
again we can further enlarge our collection of modules in $\cU$ as follows:
\[
\includegraphics{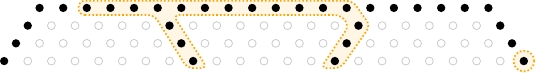}
\]
Next using the two rightmost diagonals we can apply Lemma \ref{lem:effect of M and its taud inverse}(b) to obtain the following collection of modules which must be in $\cU$:
\[
\includegraphics{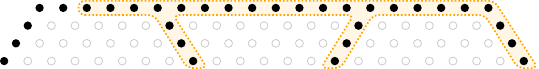}
\]
We can not get any more modules that must be included in $\cU$ using our previous results. Indeed, we can notice that the encircled part is a $4$-torsion class and thus it is equal to $\cU$.
\end{example}

Our aim is to show that we can compute all $d$-torsion classes in a way similar to Example \ref{Example:how to use lemmas for d-torsion}. The cases $l=2$ is a very degenerate version of the case $l>2$, but for technical reasons we study them separately. 

\subsubsection{The case $l>2$.}\label{subsubsec:the case l>2} Our first aim is to describe all possible pairs $(\cU_i,\cU_{i+1})$ given that $\cU$ is a $d$-torsion class. We have already done most of the work in the lemmas \ref{lem:closure under d-quotients}, \ref{lem:effect of M and its taud inverse} and \ref{lem:possible Uis}, and so our work here consists mostly of collecting all of the information in a cleaner way.

Note first that the pairs depend on the parity of $i$ and so we have two separate cases. Additionally, we also want to describe the intersection $Q_{\cU}\cap [s_i,s_{i+1}-(l-1)]$ when $i$ is odd and the intersection $Q_{\cU}\cap [s_i-(l-2),s_{i+1}-1]$ when $i$ is even, consisting of all indices $q$ such that $I(q)\in \cU$. The following two lemmas contain all of this information.

\begin{lemma}\label{lem:arrows starting at odd vertex}
    Assume that $l>2$. Let $\cU$ be a $d$-torsion class and let $i\in\{1,\ldots,p-1\}$ be odd. Then exactly one of the following holds.
    \begin{enumerate}
        \item[(a)] $(\cU_i,\cU_{i+1})\in\{(\downdiag{i}{0},\downdiag{i+1}{0}), (\downdiag{i}{0},\downdiag{i+1}{h}),(\downdiag{i}{1},\downdiag{i+1}{0}),\mid 1\leq h\leq l-2\}$
        and $Q_{\cU}\cap[s_i,s_{i+1}-(l-1)]=\varnothing$.
        \item[(b)] $(\cU_i,\cU_{i+1})=(\downdiag{i}{0},\diag{i+1})$ and there exists some $y\in\{0,1,\ldots,\tfrac{d-2}{2}l+2\}$ such that $ Q_{\cU}\cap [s_i,s_{i+1}-(l-1)]=[s_{i+1}-(l-1)-(y-1),s_{i+1}-(l-1)]$.
        \item[(c)] $(\cU_i,\cU_{i+1})=(\updiag{i}{h},\diag{i+1})$ for some $h\in\{1,\ldots,l-1\}$
        and $Q_{\cU}\cap [s_i,s_{i+1}-(l-1)]=[s_i,s_{i+1}-(l-1)]$. 
    \end{enumerate}
\end{lemma}

\begin{proof}
By Lemma \ref{lem:possible Uis} we have that $\cU_i=\downdiag{i}{0}$ or $\cU_i=\downdiag{i}{1}$ or $\cU_i=\updiag{i}{h}$ for some $h\in\{1,\ldots,l-1\}$ and that $\cU_{i+1}=\downdiag{i+1}{h'}$ for some $h'\in\{0,1,\ldots,l-1\}$. We consider three cases.

Assume first that $\cU_i=\downdiag{i}{0}$. If $\cU_{i+1}\neq \diag{i+1}$, then Lemma \ref{lem:closure under d-quotients}(b) gives that $Q_{\cU}\cap [s_i,s_{i+1}-(l-1)]=\varnothing$, and we are in case (a). Assume that $\cU_{i+1}=\diag{i+1}$. If $Q_{\cU}\cap [s_i,s_{i+1}-(l-1)]=\varnothing$, then we set $y=0$ and we are in case (b). Otherwise, set $m=\min\{Q_{\cU}\cap[s_i,s_{i+1}-(l-1)]\}$ and set $y=s_{i+1}-(l-1)-(m-1)$. It easily follows that $1\leq y\leq \tfrac{d-2}{2}l+2$. By Lemma \ref{lem:closure under d-quotients}(b) we have that $Q_{\cU}\cap [s_i,s_{i+1}-(l-1)]=[m,s_{i+1}-(l-1)]=[s_{i+1}-(l-1)-(y-1),s_{i+1}-(l-1)]$ and so we are in case (b) again.

Assume now that $\cU_i=\downdiag{i}{1}$. Assume to a contradiction that $\cU_{i+1}=\downdiag{i+1}{h'}$ for some $h'\in\{1,\ldots,l-1\}$. Then both $\ind{s_i}{s_i}\in\cU_i$ and $\ind{s_{i+1}}{s_{i+1}}\in\cU_{i+1}$ hold. Then Lemma \ref{lem:effect of M and its taud inverse}(b) gives that $\updiag{i}{1}\subseteq\cU_i$, which contradicts $\cU_i=\downdiag{i}{1}$ since $l>2$. Hence $\cU_{i+1}=\downdiag{i+1}{0}$ and so $Q_{\cU}\cap [s_i,s_{i+1}-(l-1)]=\varnothing$ by Lemma \ref{lem:closure under d-quotients}(b). Therefore we are in case (a).

Finally assume that $\cU_i=\updiag{i}{h}$ for some $1\leq h\leq l-1$. By Lemma \ref{lem:closure under d-quotients}(d) we have that $\cU_{i+1}=\diag{i+1}$. Since $\ind{s_i}{s_i+l-2}\in\cU_{i}$ by assumption and since $\tau_d^{-}(\ind{s_i}{s_i+l-2})\in\diag{i+1}=\cU_{i+1}$, we obtain by Lemma \ref{lem:effect of M and its taud inverse}(b) that $Q_{\cU}\cap [s_i,s_{i+1}-(l-1)]=[s_i,s_{i+1}-(l-1)]$ and so we are in case (c).
\end{proof}

\begin{lemma}\label{lem:arrows starting at even vertex}
    Assume that $l>2$. Let $\cU$ be a $d$-torsion class and let $i\in\{1,\ldots,p-1\}$ be even. Then exactly one of the following holds.
    \begin{enumerate}
        \item[(a)] $(\cU_i,\cU_{i+1})\in\{(\downdiag{i}{0},\downdiag{i+1}{0}),(\downdiag{i}{0},\downdiag{i+1}{1}),(\downdiag{i}{h},\downdiag{i+1}{0}),(\downdiag{i}{h},\downdiag{i+1}{1}),(\diag{i},\downdiag{i+1}{0})\mid 1\leq h\leq l-2\}$ and $Q_{\cU}\cap [s_i-(l-2),s_{i+1}-1]=\varnothing$.
        \item[(b)] $(\cU_i,\cU_{i+1})=(\downdiag{i}{0},\updiag{i+1}{h})$ for some $h\in \{2,\ldots,l-1\}$ and there exists some $y\in\{0,1,\ldots,l-h\}$ such that $Q_{\cU}\cap [s_i-(l-2),s_{i+1}-1]=[s_{i+1}-1-(y-1),s_{i+1}-1]$.
        \item[(c)] $(\cU_i,\cU_{i+1})=(\downdiag{i}{0},\diag{i+1})$ and there exists some $y\in\{0,1,\ldots,\tfrac{d}{2}l\}$ such that $Q_{\cU}\cap [s_i-(l-2),s_{i+1}-1]=[s_{i+1}-1-(y-1),s_{i+1}-1]$.
        \item[(d)] $(\cU_i,\cU_{i+1})=(\downdiag{i}{h},\diag{i+1})$ for some $h\in\{1,\ldots,l-2\}$ and there exists some $y\in\{0,1,\ldots,l-h-1\}$ such that $Q_{\cU}\cap [s_i-(l-2),s_{i+1}-1]=[s_i-h-(y-1),s_{i+1}-1]$.
        \item[(e)] $(\cU_i,\cU_{i+1})=(\diag{i},\diag{i+1})$ and $Q_{\cU}\cap [s_i-(l-2),s_{i+1}-1]=[s_i-(l-2),s_{i+1}-1]$.
        \item[(f)] $d=2$, $(\cU_i,\cU_{i+1})=(\downdiag{i}{h},\updiag{i+1}{u})$ for some $h\in\{1,\ldots,l-3\}$ and $u\in\{2,\ldots,l-h-1\}$ and there exists some $y\in\{0,1,\ldots,l-(h+u)-1\}$ such that $Q_{\cU}\cap [s_i-(l-2),s_{i+1}-1]=[s_{i}-h-(y-1),s_{i+1}-1]$.
    \end{enumerate}
\end{lemma}

\begin{proof}
By Lemma \ref{lem:possible Uis} we have that $\cU_i=\downdiag{i}{h}$ for some $h\in\{0,1,\ldots,l-1\}$ and that $\cU_{i+1}=\downdiag{i+1}{0}$ or $\cU_{i+1}=\downdiag{i+1}{1}$ or $\cU_{i+1}=\updiag{i+1}{h'}$ for some $h'\in\{1,\ldots,l-1\}$. We consider three cases.

Assume first that $\cU_i=\downdiag{i}{0}$. If $\cU_{i+1}\neq \updiag{i+1}{h'}$ for some $h'\in\{1,\ldots,l-1\}$, then Lemma \ref{lem:closure under d-quotients}(c) gives that $Q_{\cU}\cap [s_i-(l-2),s_{i+1}-1]=\varnothing$, and we are in case (a). Assume that $\cU_{i+1}=\updiag{i+1}{h'}$ for some $h'\in\{1,\ldots,l-1\}$. If $Q_{\cU}\cap [s_i-(l-2),s_{i+1}-1]=\varnothing$, set $y=0$ and we are in one of the cases (b) and (c). Otherwise, set $m=\min\{Q_{\cU}\cap [s_{i}-(l-2),s_{i+1}-1]\}$ and set $y=s_{i+1}-1-(m-1)$. Then by Lemma \ref{lem:closure under d-quotients}(c) we have that
\[
Q_{\cU}\cap [s_{i}-(l-2),s_{i+1}-1]=[m,s_{i+1}-1]=[s_{i+1}-1-(y-1),s_{i+1}-1].
\]
If $h'=1$, then $s_{i}-(l-2)\leq m \leq s_{i+1}-1$ gives that $1\leq y\leq \tfrac{d}{2}l$ and so we are in case (c). If $h'\geq 2$, then the second part of Lemma \ref{lem:closure under d-quotients}(c) gives that $s_{i+1}-(l-h')\leq m\leq s_{i+1}-1$, which gives that $1\leq y\leq l-h'$ and so we are in case (b).

Assume now that $\cU_i=\downdiag{i}{h}$ for some $h\in\{1,\ldots,l-2\}$. If $\cU_{i+1}\neq \updiag{i+1}{h'}$ for some $h'\in\{1,\ldots,l-1\}$, then Lemma \ref{lem:closure under d-quotients}(c) gives that $Q_{\cU}\cap [s_i-(l-2),s_{i+1}-1]=\varnothing$, and we are in case (a). Assume that $\cU_{i+1}=\updiag{i+1}{h'}$ for some $h'\in\{1,\ldots,l-1\}$. We claim that $h'\leq l-h-1$. Indeed, assume to a contradiction that $h'>l-h-1$. Then $l-h'\leq h$ and so $\ind{s_i-(l-h'-1)}{s_i}\in \cU_i$. Then $\tdo(\ind{s_i-(l-h'-1)}{s_i})=\ind{s_{i+1}}{s_{i+1}+(h'-1)}\in\cU_{i+1}$. Since $\cU$ is closed under minimal $d$-extensions, Lemma \ref{lem:n-extensions in C}(a) gives that 
\[
\ind{s_i+\frac{d-2}{2}l+2}{s_i-(l-h'-1)+\frac{d}{2}l-1} = \ind{s_{i+1}}{s_{i+1}+(h'-1)-1}\in\cU_{i+1},
\]
contradicting that $\cU_{i+1}=\updiag{i+1}{h'}$. Hence indeed $h'\leq l-h-1$. It follows that $\ind{s_i-(h-1)}{s_i}\in \cU_i$ and 
\[
\tdo(\ind{s_i-(h-1)}{s_i}) = \ind{s_{i+1}}{s_{i+1}+(l-h-1)}\in\cU_{i+1}.
\]
Then Lemma \ref{lem:effect of M and its taud inverse}(a) gives that $[s_i-h+1,s_{i+1}-1]\subseteq Q_{\cU}$. Set $m=\min\{Q_{\cU}\cap [s_{i}-(l-2),s_{i}-h]\}$ and $y=s_i-(h-1)-m$. Then $s_i-(l-2)\leq m\leq s_i-h$, which in turn gives that $0\leq y\leq l-h-1$. Using Lemma \ref{lem:closure under d-quotients}(c), we have that 
\[
Q_{\cU}\cap [s_{i}-(l-2),s_{i+1}-1] = [m,s_{i+1}-1] = [s_i-h-(y-1),s_{i+1}-1].
\]
Now assume that $d>2$ or $h=l-2$. Then Lemma \ref{lem:effect of M and its taud inverse} gives that $\cU_{i+1}=\diag{i+1}$. Clearly, the same is true if $h'=1$. Hence in this case we are in case (d). Thus assume that $d=2$ and $h\leq l-3$ and $h'>1$. Set $u=h'$. Then $\cU_{i+1}=\updiag{i+1}{u}$, which implies that $\ind{s_{i+1}}{s_{i+1}+u-1}\in\cU_{i+1}$ and $\ind{s_{i+1}}{s_{i+1}+u-2}\not\in\cU_{i+1}$. In particular, we have that $I(s_{i+1}+u-2-l+1)\not\in\cU$ since there exists an epimorphism $I(s_{i+1}+u-l-1)\to \ind{s_{i+1}}{s_{i+1}+u-2}$. We conclude that $s_{i+1}+u-l\leq m$. Using $m=s_i-(h-1)-y$ and $s_{i+1}-s_{i}=2$ by (\ref{eq:difference of consecutive simples}), we obtain
\[
y = s_i-(h-1)-m \leq s_i-(h-1)-s_{i+1}-u+l = l-(h+u)-1,
\]
showing that we are in case (f).

Finally assume that $\cU_i=\diag{i}$. If $\cU_{i+1}=\downdiag{i+1}{0}$, then we have that $Q_{\cU}\cap [s_i-(l-2),s_{i+1}-1]=\varnothing$ by Lemma \ref{lem:closure under d-quotients}(c) and so we are in case (a). Assume that $\cU_{i+1}\neq \downdiag{i+1}{0}$. Then there exists an indecomposable module $N\in\cU_{i+1}$. Then $N=\tau_d^{-}(M)$ for some $M\in\diag{i}=\cU_i$. By Lemma \ref{lem:effect of M and its taud inverse}(a) we obtain that $Q_{\cU(\chi)}\cap [s_i-(l-2),s_{i+1}-1]=[s_i-(l-2),s_{i+1}-1]$ and that $\cU_{i+1}=\diag{i+1}$ and so we are in case (e).
\end{proof}

Motivated by the previous results in this section, we introduce the following directed multigraph $G=G(\cC)$\index[symbols]{G@$G$}\index[symbols]{G@$G(\cC)$} for $l>2$:

\begin{equation}\label{eq:multigraph1 for d-torsion}
\includegraphics{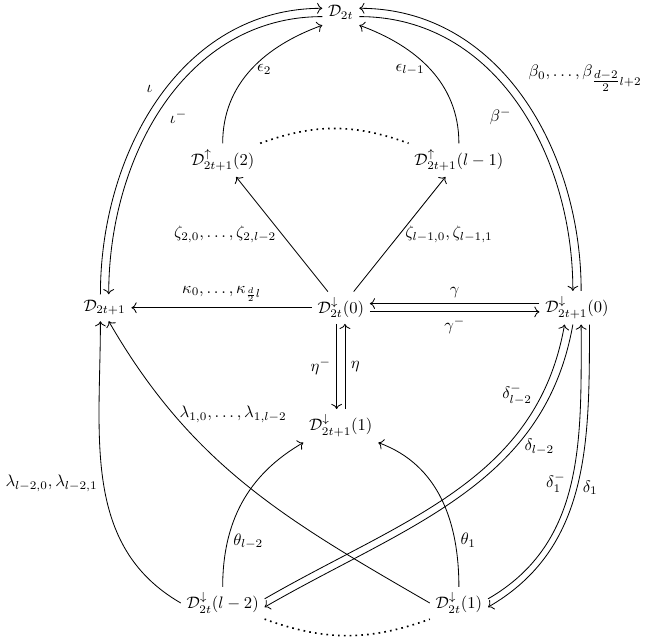}
\end{equation}
When $d=2$, we also include the following arrows in $G$:
\begin{equation}
\label{eq:multigraph1 for d-torsion addon for d=2}
\includegraphics{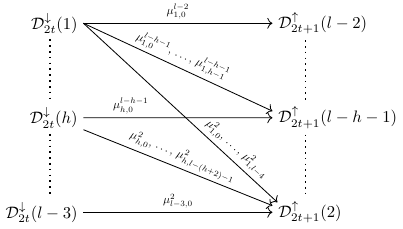}
\end{equation}
For clarity, we note that
\begin{itemize}
    \item[$\bullet$] there are $\tfrac{d-2}{2}l+3$ arrows from $\downdiag{2t+1}{0}$ to $\diag{2t}$, labelled $\beta_{0},\ldots,\beta_{\frac{d-2}{2}l+2}$,
    \item[$\bullet$] for every $h\in\{2,\ldots,l-1\}$, there are $l-h+1$ arrows from $\downdiag{2t}{0}$ to $\updiag{2t+1}{h}$, labelled $\zeta_{h,0},\ldots,\zeta_{h,l-h}$,
    \item[$\bullet$] there are $\tfrac{d}{2}l+1$ arrows from $\downdiag{2t}{0}$ to $\diag{2t+1}$, labelled $\kappa_0,\ldots,\kappa_{\frac{d}{2}l}$,
    \item[$\bullet$] for every $h\in\{1,\ldots,l-2\}$, there are $l-h$ arrows from $\downdiag{2t}{h}$ to $\diag{2t+1}$, labelled $\lambda_{h,0},\ldots,\lambda_{h,l-h-1}$, 
    \item[$\bullet$] if $d=2$, then for every $h\in\{1,\ldots,l-3\}$ and any $u\in\{2,\ldots,l-h-1\}$, there are $l-(h+u)$ arrows from $\downdiag{2t}{h}$ to $\updiag{2t+1}{u}$, labelled $\mu_{h,0}^{u},\ldots,\mu_{h,l-(h+u)-1}^{u}$, and
    \item[$\bullet$] all other arrows have multiplicity one.
\end{itemize}

We say that a vertex in $G$ is \emph{odd}\index[definitions]{vertex!odd} if the subscript appearing in that vertex is $2t+1$ and that it is \emph{even}\index[definitions]{vertex!even} if the subscript appearing in that vertex is $2t$. Notice that all arrows starting at odd vertices terminate at even vertices and all arrows starting at even vertices terminate at odd vertices. Hence a directed path $\chi$ in $G$ of odd length starts and terminates at vertices of opposite parity.

Let $\chi=\chi_1\chi_2\cdots\chi_{p-1}$ be a directed path in $G$ of length $p-1$, such that $\chi$ starts at an odd vertex; since $p$ is even, this is equivalent to $\chi$ terminating at an even vertex. We associate to $\chi$ a full additive subcategory $\cU(\chi)\subseteq \cC$\index[symbols]{U@$\cU(\chi)$}\index[symbols]{U@$\cU_i(\chi)$}. First, let $i\in\{1,\ldots,p\}$ and write $i=2t+1$ if $i$ is odd and $i=2t$ if $i$ is even. Then we set
\[
\cU_i(\chi)=\cU(\chi)\cap \diag{i} \coloneqq \begin{cases}
    \text{start of arrow $\chi_i$,} &\mbox{if $i$ is odd,} \\
    \text{end of arrow $\chi_i$,} &\mbox{if $i$ is even.}
\end{cases}
\]
Next, we want to define $Q_{\cU(\chi)}$. Notice that we may partition $[1,n-(l-1)]$ as
   \begin{align*}
[1,n-(l-1)] &= [s_1,s_2-(l-1)]\sqcup [s_2-(l-2),s_3-1] \sqcup \cdots \sqcup [s_{p-1},s_p-(l-1)] \\
&= F_1\sqcup F_2 \sqcup \cdots \sqcup F_{p-1}
\end{align*}
We describe some conditions which describe $Q_{\cU(\chi)}\cap F_i$\index[symbols]{F@$F_i$} for some $i\in\{1,\ldots,p-1\}$; for all other $i$ we set $Q_{\cU(\chi)}\cap F_i=\varnothing$. 

\begin{conditions}\label{cond:definition of U(chi)}
\begin{enumerate}
    \item[(i)] if $i$ is odd and $\chi_i=\beta_y$ for some $y\in\{0,\ldots,\tfrac{d-2}{2}l+2\}$, then $Q_{\cU(\chi)}\cap [s_i,s_{i+1}-(l-1)]=[s_{i+1}-(l-1)-(y-1),s_{i+1}-(l-1)]$,
    \item[(ii)] if $i$ is odd and $\chi_i=\iota$ or $\chi_i=\epsilon_h$ for some $h\in\{2,\ldots,l-1\}$, then $Q_{\cU(\chi)}\cap [s_i,s_{i+1}-(l-1)]=[s_i,s_{i+1}-(l-1)]$,
    \item[(iii)] if $i$ is even and $\chi_i=\zeta_{h,y}$ for some $h\in\{2,\ldots,l-1\}$ and some $y\in\{0,1,\ldots,l-h\}$, then $Q_{\cU(\chi)}\cap [s_{i}-(l-2),s_{i+1}-1]=[s_{i+1}-1-(y-1),s_{i+1}-1]$,
    \item[(iv)] if $i$ is even and $\chi_i=\kappa_y$ for some $y\in\{0,1,\ldots,\tfrac{d}{2}l\}$, then $Q_{\cU(\chi)}\cap [s_i-(l-2),s_{i+1}-1]=[s_{i+1}-1-(y-1),s_{i+1}-1]$,
    \item[(v)] if $i$ is even and $\chi_i=\lambda_{h,y}$ for some $h\in \{1,\ldots,l-2\}$ and some $y\in\{0,1,\ldots,l-h-1\}$, then $Q_{\cU(\chi)}\cap [s_{i}-(l-2),s_{i+1}-1]= [s_{i}-h-(y-1),s_{i+1}-1]$,
    \item[(vi)] if $i$ is even and $\chi_i=\iota^{-}$, then $ Q_{\cU(\chi)}\cap [s_i-(l-2),s_{i+1}-1]=[s_i-(l-2),s_{i+1}-1]$,
    \item[(vii)] if $d=2$, $i$ is even and $\chi_i=\mu_{h,y}^u$ for some $h\in\{1,\ldots,l-3\}$ and some $u\in\{2,\ldots,l-h-1\}$ and some $y\in\{0,1,\ldots,l-(h+u)-1\}$, then $Q_{\cU(\chi)}\cap [s_{i}-(l-2),s_{i+1}-1]=[s_{i}-h-(y-1),s_{i+1}-1]$.
\end{enumerate}
\end{conditions}

Then we have
\[
\cU(\chi)=\add{\cU_1(\chi),\ldots,\cU_p(\chi),I(Q_{\cU(\chi)})}.
\]

\begin{example}\label{Example:d-torsion class from the graph} Let $\Lambda=\Lambda(37,4)$ and $d=4$ so that $\Lambda$ admits a $4$-cluster tilting subcategory $\cC$. Consider the path $\chi=\epsilon_2\zeta_{2,2}\eta\theta_2\delta_2$. Then $\cU(\chi)$ is the subcategory defined as the additive closure of the encircled modules in the following picture:
\[
\includegraphics{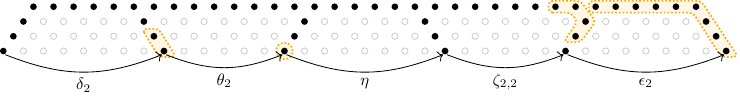}
\]
As the next theorem shows, $\cU(\chi)$ is a $4$-torsion class.
\end{example}

We are ready to prove our main result for this section, namely that $d$-torsion classes can be described using paths in the graph $G$.

\begin{theorem}\label{thrm:classification of d-torsion classes}
Assume that $l>2$. There exists a bijection between the set of paths $\chi$ in $G$ of length $p-1$ starting at an odd vertex and the set of $d$-torsion classes $\cU$ in $\cC$. The bijection is given by $\chi\mapsto \cU(\chi)$.
\end{theorem}

\begin{proof}
Let $\cU$ be a $d$-torsion class in $\cC$. Let $1\leq i\leq p$. Write $i=2t+1$ if $i$ is odd and $i=2t$ if $i$ is even. By Lemma \ref{lem:possible Uis} we have that $\cU_i$ appears as one of the vertices in $G$. Assume that $i\leq p-1$. Then the arrows in $G$ starting at $\cU_i$ are in bijection with the possible configurations $(\cU_i,\cU_{i+1})$ as they are outlined in Lemma \ref{lem:arrows starting at odd vertex} and Lemma \ref{lem:arrows starting at even vertex}; the bijection is described by the conditions on the arrows given when defining $Q_{\cU(\chi)}$. We may thus define $\chi_i$ to be the corresponding arrow under this bijection. By definition of $\cU(\chi)$, we then have that $\cU(\chi)=\cU$.

It remains to show that $\cU(\chi)$ is a $d$-torsion class. For this we use Definition \ref{def:d-torsion class}. Let $C\in\cC$ and we need to produce a $d$-extension
\[
\begin{tikzcd}
	0 & U & C & C_1 & \cdots & C_d & 0 
	  \arrow[from=1-1, to=1-2]
        \arrow["u", from=1-2, to=1-3]
        \arrow["c_0", from=1-3, to=1-4]
        \arrow["c_1", from=1-4, to=1-5]
        \arrow["c_{d-1}", from=1-5, to=1-6]
        \arrow[from=1-6, to=1-7]
\end{tikzcd}
\]
such that $U\in\cU(\chi)$ and, for every $U'\in\cU(\chi)$, the induced sequence
\[
\begin{tikzcd}
	0 & \Hom{\Lambda}{U}{C_1} & \cdots & \Hom{\Lambda}{U}{C_d} & 0 
	  \arrow[from=1-1, to=1-2]
        \arrow[from=1-2, to=1-3]
        \arrow[from=1-3, to=1-4]
        \arrow[from=1-4, to=1-5]
\end{tikzcd}
\]
is exact. Since $\cU(\chi)$ is an additive category, it is enough to show the claim for $C$ indecomposable. If $C\in\cU(\chi)$ we can clearly consider the $d$-extension 
\[
\begin{tikzcd}
	0 & C & C & 0 & \cdots & 0 & 0. 
	  \arrow[from=1-1, to=1-2]
        \arrow["1", from=1-2, to=1-3]
        \arrow["0", from=1-3, to=1-4]
        \arrow["0", from=1-4, to=1-5]
        \arrow["0", from=1-5, to=1-6]
        \arrow[from=1-6, to=1-7]
\end{tikzcd}
\]
Assume that $C\not\in \cU(\chi)$. We consider three complementary cases: $C\in\diag{i}$ and $i$ is odd, $C\in\diag{i}$ and $i$ is even and $C=I(q)$ for some $q\in[1,n-(l-1)]$.

Case $C\in\diag{i}$ and $i$ is odd. Then $C=\ind{s_i}{s_i+h-1}$ for some $h\in\{1,\ldots,l-1\}$. By definition of $\cU(\chi)$ we have that $\cU_i(\chi)\in\{\downdiag{i}{0},\downdiag{i}{1},\updiag{i}{h'}\mid 1\leq h'\leq l-1\}$. Since $C\not\in\cU_i(\chi)$, we conclude that either no submodule of $C$ is in $\cU_i(\chi)$ or $\cU_i(\chi)=\downdiag{i}{1}$. If no submodule of $C$ is in $\cU_i(\chi)$, consider the $d$-extension 
\[
\begin{tikzcd}
	0 & 0 & C & C & 0 & \cdots & 0 & 0,
	  \arrow[from=1-1, to=1-2]
        \arrow["0", from=1-2, to=1-3]
        \arrow["1", from=1-3, to=1-4]
        \arrow["0", from=1-4, to=1-5]
        \arrow["0", from=1-5, to=1-6]
        \arrow["0",from=1-6, to=1-7]
        \arrow[from=1-7, to=1-8]
\end{tikzcd}
\]
and it is enough to show that $\Hom{\Lambda}{U'}{C}=0$ for all $U'\in \cU(\chi)$. If $s_i=1$, then this is clear. Assume that $s_i>1$ and assume to a contradiction that there exists a nonzero morphism $f:U'\to C$ for some indecomposable module $U'\in\cU(\chi)$. Since $i$ is odd, we have that $\Hom{\Lambda}{\diag{i-1}}{\diag{i}}=0$ by Lemma \ref{Lemma:DiagZeroHom}. We conclude that $U'$ is injective and so $I(q)\in \cU(\chi)$ for some $q\in[s_i-(l-1),s_i+h-l]$. Hence $Q_{\cU(\chi)}\cap [s_i-(l-2),s_{i+1}-1] \neq \varnothing$. By the construction of $\cU(\chi)$ we obtain that $Q_{\cU(\chi)}\cap [s_i-(l-2),s_{i+1}-1]\supseteq [q,s_{i+1}-1]$ and that $\chi_{i-1}$ is one of the arrows in cases (iii), (iv), (v), (vi) or (vii) of Conditions \ref{cond:definition of U(chi)}. In each of these cases we obtain that $\updiag{i}{q+l-s_i}\subseteq \cU_i(\chi)$. But then $q\leq s_i+h-l$ gives $q+l-s_i\leq h$ and so $C\in \updiag{i}{q+l-s_i}$, contradicting $C\not\in\cU_i(\chi)$. Hence it remains to consider the case $\cU_{i}(\chi)=\downdiag{i}{1}$. In this case, by the construction of $\cU(\chi)$, we have that $\chi_{i-1}=\theta_h$ and $\chi_i=\eta$. We conclude that $\cU_{i+1}(\chi)=\downdiag{i+1}{0}$ and moreover that $Q_{\cU(\chi)}\cap [s_{i}-(l-2),s_{i+1}-(l-1)]=\varnothing$. Set $b=s_i$ and $j=h-1$ and consider the $d$-extension in Lemma \ref{lem:n-extensions in C}(b), which is of the form
\[
\begin{tikzcd}
	0 & \ind{s_i}{s_i} & \ind{s_i}{s_i+h-1}=C & C_1 & \cdots & C_d & 0. 
	  \arrow[from=1-1, to=1-2]
        \arrow["u", from=1-2, to=1-3]
        \arrow["c_0", from=1-3, to=1-4]
        \arrow["c_1", from=1-4, to=1-5]
        \arrow["c_{d-1}", from=1-5, to=1-6]
        \arrow[from=1-6, to=1-7]
\end{tikzcd}
\]
Similarly to the previous case, it is enough to show that $\Hom{\Lambda}{\cU(\chi)}{C_s}=0$ for $s\in\{1,\ldots,d\}$. All the modules $C_1,\ldots,C_d$ are either indecomposable injective modules of the form $I(q)$ for $q\in [s_i+1,s_{i+1}-(l-1)]$ or modules in $\diag{i+1}$. Since $Q_{\cU(\chi)}\cap [s_{i-1}-(l-1)+1,s_{i+1}-(l-1)]=\varnothing$ and $\cU_{i+1}(\chi)=\downdiag{i+1}{0}$, it follows that $\Hom{\Lambda}{\cU(\chi)}{C_s}=0$ for $s\in\{1,\ldots,d\}$, as required.

Case $C\in\diag{i}$ and $i$ is even. Then $C=\ind{s_i-h+1}{s_i}$ for some $h\in\{1,\ldots,l-1\}$. By definition of $\cU(\chi)$ we have that $\cU_i(\chi)=\downdiag{i}{h'}$ for some $h'\in\{0,1\ldots,l-1\}$. Since $C\not\in\cU_i(\chi)$, we have that $h'< h\leq l-1$. By the construction of $\cU(\chi)$ we obtain that $\chi_{i-1}$ is one of the arrows $\gamma$, $\eta$ or $\delta_1,\ldots,\delta_{l-2}$ and $\cU_{i-1}(\chi)=\downdiag{i-1}{0}$ or $\cU_{i-1}(\chi)=\downdiag{i-1}{1}$. In any case we have that $Q_{\cU(\chi)}\cap [s_{i-1},s_{i}-(l-1)]=\varnothing$ and that $\Hom{\Lambda}{\cU_{i-1}(\chi)}{C}=0$. It follows that $\Hom{\Lambda}{\cU(\chi)}{C}=0$ and so we may again consider the $d$-extension
\[
\begin{tikzcd}
	0 & 0 & C & C & 0 & \cdots & 0 & 0.
	  \arrow[from=1-1, to=1-2]
        \arrow["0", from=1-2, to=1-3]
        \arrow["1", from=1-3, to=1-4]
        \arrow["0", from=1-4, to=1-5]
        \arrow["0", from=1-5, to=1-6]
        \arrow["0",from=1-6, to=1-7]
        \arrow[from=1-7, to=1-8]
\end{tikzcd}
\]

Finally, assume that $C=I(q)$ for some $q\in [1,n-(l-1)]$. Then either $q\in [s_i,s_{i+1}-(l-1)]$ for some odd $i$ or $q\in [s_{i}-(l-2),s_{i+1}-1]$ for some even $i$. We consider the two cases separately.

Assume that $q\in [s_i,s_{i+1}-(l-1)]$ for some odd $i$. Then $\cU_{i}(\chi)\in \{\downdiag{i}{0},\downdiag{i}{1},\updiag{i}{h'}\mid 1\leq h'\leq l-1\}$. If $\cU_{i}(\chi)=\updiag{i}{h'}$ for some $h'\in [1,l-1]$, then by the construction of $\cU(\chi)$ we obtain that $\chi_i=\epsilon_{h'}$ and so $[s_i,s_{i+1}-(l-1)]\subseteq Q_{\cU(\chi)}$ by Condition \ref{cond:definition of U(chi)}(ii). But then $I(q)\in \cU(\chi)$, which is a contradiction. Hence we may assume that $\cU_{i}(\chi)=\downdiag{i}{0}$ or $\cU_{i}(\chi)=\downdiag{i}{1}$. Checking the arrows in $G$ going to the vertices $\downdiag{2t+1}{0}$ and $\downdiag{2t+1}{1}$ reveals that none of them appears in Condition \ref{cond:definition of U(chi)}, from which we obtain that $Q_{\cU(\chi)}\cap [s_{i-1}-(l-2),s_i-1]=\varnothing$. We may again argue that we also have $Q_{\cU(\chi)}\cap [s_i,q]=\varnothing$, since otherwise $C=I(q)\in\cU(\chi)$. All in all we obtain that $\Hom{\Lambda}{\cU(\chi)}{C}=0$ which is enough. 

For the last case we may assume that $q\in [s_{i}-(l-2),s_{i+1}-1]$ for some even $i$. Again we may assume that $\Hom{\Lambda}{\cU(\chi)}{C}\neq 0$. By the definition of $\cU(\chi)$ and since $I(q)\not\in\cU(\chi)$, it is easy to check that this implies that $\ind{q}{s_i}\in\cU_i(\chi)$. Combined with the fact that $I(q)\not\in\cU(\chi)$ we obtain that $\cU_{i+1}(\chi)=\downdiag{i+1}{0}$ or $\cU_{i+1}(\chi)=\downdiag{i+1}{1}$. It follows that $[s_{i},s_{i+1}-(l-1)]\cap Q_{\cU(\chi)}=\varnothing$. Set $a=q$ and $j=s_i-a+1$ and consider the $d$-extension in Lemma \ref{lem:n-extensions in C}(a), which is of the form
\[
\begin{tikzcd}
	0 & \ind{q}{s_i} & I(q)=C & C_1 & \cdots & C_d & 0. 
	  \arrow[from=1-1, to=1-2]
        \arrow["u", from=1-2, to=1-3]
        \arrow["c_0", from=1-3, to=1-4]
        \arrow["c_1", from=1-4, to=1-5]
        \arrow["c_{d-1}", from=1-5, to=1-6]
        \arrow[from=1-6, to=1-7]
\end{tikzcd}
\]
All the modules $C_1,\ldots,C_{d-2}$ are indecomposable injective modules of the form $I(q')$ with $q'\in [q+j,q+\tfrac{d-2}{2}l]=[s_i+1,q+\tfrac{d-2}{2}l]$. Moreover we have
\[
C_{d-1} = I(q'') \oplus \ind{s_i+\tfrac{d-2}{2}l+2}{q+\tfrac{d}{2}l-1} \text{ and } C_{d}=\ind{s_i+\tfrac{d-2}{2}l+2}{s_i+\tfrac{d}{2}l},
\]
where again $q''\in [s_i+1,s_{i+1}-1]$. Set $C_{d-1}'=\ind{s_i+\tfrac{d-2}{2}l+2}{q+\tfrac{d}{2}l-1}$. Since $[s_{i},s_{i+1}-(l-1)]\cap Q_{\cU(\chi)}=\varnothing$ and $C_{d-1}',C_d\in \diag{i+1}$, we obtain that $\Hom{\Lambda}{\cU(\chi)}{C_i}=0$ for $1\leq i\leq d-2$ and that $\Hom{\Lambda}{\cU(\chi)}{I(q'')}=0$ too. Now let $U'\in\cU(\chi)$ be indecomposable. Then the sequence
\begin{equation}\label{eq:the exact sequence for even case}
\begin{tikzcd}
	0 & \Hom{\Lambda}{U'}{C_1} & \cdots & \Hom{\Lambda}{U'}{C_{d-1}} & \Hom{\Lambda}{U'}{C_{d}} & 0 
	  \arrow[from=1-1, to=1-2]
        \arrow[from=1-2, to=1-3]
        \arrow[from=1-3, to=1-4]
        \arrow[from=1-4, to=1-5]
        \arrow[from=1-5, to=1-6]
\end{tikzcd}
\end{equation}
becomes
\[
\begin{tikzcd}
	0 & 0 & \cdots & 0 & \Hom{\Lambda}{U'}{C_{d-1}'} & \Hom{\Lambda}{U'}{C_{d}} & 0. 
	  \arrow[from=1-1, to=1-2]
        \arrow[from=1-2, to=1-3]
        \arrow[from=1-3, to=1-4]
        \arrow[from=1-4, to=1-5]
        \arrow[from=1-5, to=1-6]
        \arrow[from=1-6, to=1-7]
\end{tikzcd}
\]
Notice that both $C_{d-1}'$ and $C_{d}$ belong to $\diag{i+1}$ and that $[s_{i},s_{i+1}-(l-1)]\cap Q_{\cU(\chi)}=\varnothing$. Hence if $U'\not\in \cU_{i+1}(\chi)$, then both $\mathrm{Hom}$-spaces in the above sequence become zero. It remains to consider the case $U'\in\cU_{i+1}(\chi)$. Since $\cU_{i+1}(\chi)=\downdiag{i+1}{0}$ or $\cU_{i+1}(\chi)=\downdiag{i+1}{1}$, we conclude that it remains to check the case $U'=\ind{s_{i+1}}{s_{i+1}}$. But then, using the fact that $C_{d-1}',C_d\in\diag{i+1}$ we obtain that
\[
\Hom{\Lambda}{\ind{s_{i+1}}{s_{i+1}}}{C_{d-1}'}\isom \Hom{\Lambda}{\ind{s_{i+1}}{s_{i+1}}}{C_d},
\]
which shows that the sequence (\ref{eq:the exact sequence for even case}) is exact, as required.
\end{proof}

\subsubsection{The case $l=2$.}\label{subsubsec:the case l=2} We finish this section with this case which is considerably simpler.

\begin{lemma}\label{lem:arrows starting at vertex for l=2}
    Assume that $l=2$. Let $\cU$ be a $d$-torsion class and let $i\in\{1,\ldots,p-1\}$. Then exactly one of the following holds.
    \begin{enumerate}
        \item[(a)] $(\cU_i,\cU_{i+1})\in\{(\downdiag{i}{0},\downdiag{i+1}{0}), (\diag{i},\downdiag{i+1}{0})\}$
        and $Q_{\cU}\cap[s_i,s_{i+1}-1]=\varnothing$.
        \item[(b)] $(\cU_i,\cU_{i+1})=(\downdiag{i}{0},\diag{i+1})$ and there exists some $y\in\{0,1,\ldots,d\}$ such that $ Q_{\cU}\cap [s_i,s_{i+1}-1]=[s_{i+1}-y,s_{i+1}-1]$.
        \item[(c)] $(\cU_i,\cU_{i+1})=(\diag{i},\diag{i+1})$ and $Q_{\cU}\cap [s_i,s_{i+1}-1]=[s_i,s_{i+1}-1]$. 
    \end{enumerate}
\end{lemma}

\begin{proof}
    The proof is similar to the proofs of Lemma \ref{lem:arrows starting at odd vertex} and Lemma \ref{lem:arrows starting at even vertex}; we leave the details to the interested reader.
\end{proof}

Notice in particular that the possibilities for $(\cU_i,\cU_{i+1})$ in Lemma \ref{lem:arrows starting at vertex for l=2} do not depend on the parity of $i$, unlike the case $l>2$. Then, in a similar manner to the previous section, we may introduce the following directed multigraph $G=G(\cC)$\index[symbols]{G@$G$}\index[symbols]{G@$G(\cC)$} for $l=2$.
\begin{equation}\label{eq:multigraph2 for d-torsion}
\begin{tikzcd}[cells={text width={width("$D^{\downarrow}(0)$")},align=center}]
	{\mathcal{D}^{\downarrow}(0)} \arrow["\gamma", loop, out=210, in=150, looseness=4] && {} && {\mathcal{D}} \arrow["\epsilon", swap, loop, out=-30, in=30, looseness=4] \\
	&& {}
	\arrow["{\delta_0}"'{pos=0.4}, from=1-1, to=1-5]
	\arrow["{\delta_d}"'{pos=0.43}, curve={height=30pt}, from=1-1, to=1-5]
	\arrow[shorten <=1pt, shorten >=2pt, dotted, no head, from=1-3, to=2-3]
	\arrow["{\beta}"', curve={height=24pt}, from=1-5, to=1-1]
\end{tikzcd}
\end{equation}

Let $\chi=\chi_1\chi_2\cdots\chi_{p-1}$ be a directed path in $G$ of length $p-1$. Notice that since we are in the case $l=2$, $p$ does not necessarily have to be even. To $\chi$ we associate a full additive subcategory $\cU(\chi)\subseteq \cC$. First, for $i\in\{1,\ldots,p\}$ we set
\[
\cU_i(\chi)=\cU(\chi)\cap \diag{i} \coloneqq \text{start of arrow $\chi_i$}
\]
Next, we want to define $Q_{\cU(\chi)}$. Notice that we may partition $[1,n-1]$ as
   \begin{align*}
[1,n-1] = [s_1,s_2-1]\sqcup [s_2,s_3-1] \sqcup \cdots \sqcup [s_{p-1},s_p-1] = F_1\sqcup F_2 \sqcup \cdots \sqcup F_{p-1}
\end{align*}
We set
\[
Q_{\cU(\chi)}\cap F_i \coloneqq \begin{cases}
    \varnothing, &\mbox{if $\chi_i=\beta$ or $\chi_i=\gamma$,} \\
    [s_{i+1}-y,s_{i+1}-1], &\mbox{if $\chi_i=\delta_y$ for some $y\in \{0,\ldots,d\}$,} \\
    F_i, &\mbox{if $\chi_i=\epsilon$.}
\end{cases}
\]
Then we set
\[
\cU(\chi)=\add{\cU_1(\chi),\ldots,\cU_p(\chi),I(Q_{\cU(\chi)})},
\]
and have the following result.

\begin{theorem}\label{thrm:classification of d-torsion classes in l=2}
Assume that $l=2$. There exists a bijection between the set of paths $\chi$ in $G$ of length $p-1$ starting at an odd vertex and the set of $d$-torsion classes $\cU$ in $\cC$. The bijection is given by $\chi\mapsto \cU(\chi)$.
\end{theorem}

\begin{proof}
    The proof is similar to the proof of Theorem \ref{thrm:classification of d-torsion classes}; we leave the details to the interested reader.
\end{proof}

\begin{example}\label{Example:d-torsion class from the graph l=2}
Let $\Lambda=\Lambda(19,2)$ and $d=3$ so that $\Lambda$ admits a $3$-cluster tilting subcategory $\cC$. Consider the path $\chi=\beta\delta_3\beta\epsilon\delta_2\gamma$. A direct computation shows that $\cU(\chi)$ is the $3$-torsion class given by the encircled modules in the following picture:
\[
\includegraphics{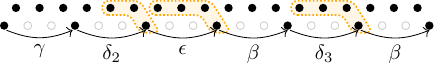}
\]

\subsubsection{The case \texorpdfstring{$d=\gldim(\Lambda)$}{d=gldimL}}

We finish this section by applying our classification of $d$-torsion classes to count their number in the special case where $d=\gldim(\Lambda)$. 

\begin{remark}\label{rem:counting d-torsion classes for d-rep-fin}
\begin{enumerate}
    \item[(a)] Assume that $d=\gldim(\Lambda(n,l))$ and that $\Lambda(n,l)$ admits a $d$-cluster tilting subcategory $\cC$. Then $p=2$ and $d=2\tfrac{n-1}{l}$ by Remark \ref{rem:d=gldim(Lambda) for Nakayama}. If $l>2$, then by Theorem \ref{thrm:classification of d-torsion classes} every $d$-torsion class of $\modfin{\Lambda}$ is given by a path in $G(\cC)$ of length $1$ starting at an odd vertex. As a path of length one is an arrow, counting the arrows in (\ref{eq:multigraph1 for d-torsion}) starting at an odd vertex we find that there are 
    \[
    1+(l-2)+\left(\frac{d-2}{2}l+3\right)+1+(l-2)+1 = \frac{d+2}{2}l +2 = \frac{2\frac{n-1}{l}+2}{2}l+2=n+l+1
    \]
    such paths. If $l=2$ then by Theorem \ref{thrm:classification of d-torsion classes in l=2} every $d$-torsion class of $\modfin{\Lambda}$ is given by an arrow in $G(\cC)$. Counting the arrows in (\ref{eq:multigraph2 for d-torsion}) we find that there are $d+4 = 2\tfrac{n-1}{2}+4=n+3$ such paths. All in all we conclude that in either case $\Lambda(n,l)$ admits exactly $n+l+1$ $d$-torsion classes.

    \item[(b)] A similar calculation shows that when $p=4$ and either ($d>2$ and $l>2$) or ($d=2$ and $l=3$) there exist exactly
    \[
    \frac{1}{18}(17 l^2 + 10ln + 57l + 2n^2 + 30n + 18)
    \]
    $d$-torsion classes.
\end{enumerate}
\end{remark}
\end{example}

\section{\texorpdfstring{$(d+1)$}{(d+1)}-term silting complexes}
\label{Section:d+1 silting}
\subsubsection*{Motivation and aim.} An important result in $\tau$-tilting theory establishes a bijection between basic support $\tau$-tilting pairs and basic $2$-silting complexes in $\HomotopyC{\proj{\Lambda}}{b}$ \cite{adachi_-tilting_2014}. Thus, when trying to generalize support $\tau$-tilting modules in the setting of higher-dimensional homological algebra, it is natural to consider an approach through silting theory. This is done in \cite{MARTINEZ202398} where Martínez and Mendoza give a higher version of $\tau$-tilting modules through $(d+1)$-silting complexes. We call such modules \emph{MM-$\td$-tilting} and we note that it is straightforward to adapt this notion to pairs $(M,P)$ as well. The main aim of this section is to prove Theorem \ref{Theorem C}, which essentially shows that such pairs are characterized by being summand-maximal $\td$-rigid pairs.

\subsubsection*{Setting} The main observations in this section are done for the algebras $\Lambda(n,l)$, however we begin our endeavour in a more general setting by recalling some results from \cite{MARTINEZ202398} and giving some fundamental definitions. The notation introduced in this section is less extensive than in the previous ones, but as before we refer to the index at the end of the article for reminders.

\subsection{Preliminaries}
Let $d\geq 1$ be a positive integer. In this section we begin by letting $\Lambda$ be a finite dimensional $\K$-algebra with $n$ simples and $\HomotopyC{\proj{\Lambda}}{b}$\index[symbols]{Kb@$\HomotopyC{\proj{\Lambda}}{b}$} the category of bounded complexes of finitely generated projective $\Lambda$-modules. Let $\mathbf{P}^{\bullet}\in \HomotopyC{\proj{\Lambda}}{b}$ be a complex. We define the \emph{truncation}\index[definitions]{truncation of complex} $\sigma_{\geq k}\mathbf{P}^{\bullet}$ and $\sigma_{\leq k}\mathbf{P}^{\bullet}$ by
\[
\sigma_{\geq k}\mathbf{P}^{i} = \begin{cases}
    \mathbf{P}^{i}, &\mbox{if $i\geq k$,} \\
    0, &\mbox{if $i<k$,}
\end{cases}
\qquad
\text{ respectively }
\qquad
\sigma_{\leq k}\mathbf{P}^{i} = \begin{cases}
    \mathbf{P}^{i}, &\mbox{if $i\leq k$,} \\
    0, &\mbox{if $i>k$.}
\end{cases}
\]
We similarly define the \emph{truncations} $\sigma_{< k}\mathbf{P}^{\bullet}$ and $\sigma_{>k}\mathbf{P}^{\bullet}$. The \emph{thick closure}\index[definitions]{thick closure} of $\mathbf{P}^{\bullet}$, denoted $\thick{\mathbf{P}^{\bullet}}$\index[symbols]{T@$\thick{\mathbf{P}^{\bullet}}$}, is the smallest triangulated subcategory of $\HomotopyC{\proj{\Lambda}}{b}$ which is closed under direct summands and contains $\mathbf{P}^{\bullet}$. The complex $\mathbf{P}^{\bullet}$ is called \textit{presilting}\index[definitions]{presilting complex} if it admits no non-trivial maps to its positive shifts, i.e. $\Hom{\HomotopyC{\proj{\Lambda}}{b}}{\mathbf{P}^\bullet}{\shift{\mathbf{P}^\bullet}{i}}=0$ for all $i> 0$. 
A presilting complex $\mathbf{P}^\bullet$ is a \emph{silting complex}\index[definitions]{silting complex} if its thick closure is the whole of $\HomotopyC{\proj{\Lambda}}{b}$.
Let $\silt{\Lambda}$\index[symbols]{Sd@$\silt{\Lambda}$} denote the collection of silting complexes. Since we have assumed that $\abs{\Lambda}=n$ and $\Lambda$ is clearly a silting complex, we have the following well-known observation.

\begin{lemma}[{\cite[Corollary 2.28]{SiltingMutation}}]
\label{lemma:silting summands}
    For any $\mathbf{P}^{\bullet}\in\silt{\Lambda}$ we have $\abs{\mathbf{P}^{\bullet}}=n$.
\end{lemma}

For any $M\in\modfin\Lambda$ we fix a minimal projective resolution 
\[
\cdots \rightarrow \mathbf{P}^{-3}(M)\rightarrow \mathbf{P}^{-2}(M)\rightarrow \mathbf{P}^{-1}(M)\rightarrow \mathbf{P}^{0}(M)\rightarrow M
\]
of $M$ and obtain the complex $\mathbf{P}^{\bullet}(M)$ given by
$$
\mathbf{P}^\bullet(M)\colon\qquad \cdots \rightarrow \mathbf{P}^{-3}(M)\rightarrow \mathbf{P}^{-2}(M)\rightarrow \mathbf{P}^{-1}(M)\rightarrow \mathbf{P}^{0}(M)\rightarrow 0 \rightarrow \cdots .
$$
We now recall the following definition from \cite{MARTINEZ202398}.

\begin{definition}\cite[Definition 1.1]{MARTINEZ202398}\label{def:MM-tau_d-tilting}
    We say that a module $M\in\modfin\Lambda$ is \emph{Martínez-Mendoza $\td$-tilting}, abbreviated to \emph{MM-$\td$-tilting}\index[definitions]{MM-$\td$-tilting module}, if $\trunc{\mathbf{P}^\bullet}{\geq -d}(M)\in\silt\Lambda$.
\end{definition}

As shown in \cite[Theorem 5.7]{MARTINEZ202398}, a MM-$\td$-tilting module is sincere and has $n$ non-isomorphic indecomposable summands. Moreover, in \cite[Theorem 3.4]{MARTINEZ202398} it is shown that $\trunc{\mathbf{P}^\bullet}{\geq -d}(M)$ is presilting in $\HomotopyC{\proj{\Lambda}}{b}$ if and only if $\Hom{\Lambda}{M}{\td(M)}=0$ and $\Ext_{\Lambda}^{i}({M},{M})=0$ for $1\leq i\leq d-1$.

Now let $\cC\subseteq\modfin\Lambda$ be a $d$-cluster tilting subcategory and let $(M,P)$ be a $\td$-rigid pair. By the previous paragraph we conclude that $\trunc{\mathbf{P}^\bullet}{\geq -d}(M)$ is presilting in $\HomotopyC{\proj{\Lambda}}{b}$. On the other hand, if $P\neq 0$, then $M$ is not sincere and so $\trunc{\mathbf{P}^\bullet}{\geq -d}(M)$ cannot be a silting complex. To have a chance at obtaining a silting complex from a $\td$-rigid pair $(M,P)$, it is necessary to encode information of $P$ in the candidate complex. We follow the strategy of the $d=1$ case in \cite{adachi_-tilting_2014} and simply add the projective module to the truncated complex as a stalk complex. Let us therefore fix the notation $\mathbf{P}^\bullet_{(M,P)}\coloneqq \shift{P}{d}\oplus\trunc{\mathbf{P}^\bullet}{\geq -d}(M)\in\HomotopyC{\proj\,\Lambda}{b}$\index[symbols]{Pc@$\mathbf{P}^\bullet_{(M,P)}$} for any $\cC$-pair $(M,P)$. We are now ready to give a slight expansion of \cite[Thm. 3.4]{MARTINEZ202398}. The proof is straight-forward using \cite[Lem. 5.2]{MARTINEZ202398}, hence we omit it here.
\begin{lemma}
    \label{lemma:taud-rigid iff presilting}
    Let $(M,P)$ be a $\cC$-pair.
    Then $(M,P)$ is a $\td$-rigid pair if and only if $\mathbf{P}^\bullet_{(M,P)}$
    is a $(d+1)$-presilting complex.
\end{lemma}

As a corollary, we obtain one direction of Theorem \ref{Theorem C}.

\begin{corollary}\label{cor:silting give strongly maximal}
Let $(M,P)$ be a $\cC$-pair. If $\mathbf{P}^\bullet_{(M,P)}\in\silt\Lambda$,
then $(M,P)$ is a $\td$-rigid pair with $\abs{M}+\abs{P}=\abs{\Lambda}$.
\end{corollary}

\begin{proof}
    Follows from Lemma \ref{lemma:silting summands} and \ref{lemma:taud-rigid iff presilting}.
\end{proof}

\subsection{Thick closure of summand-maximal \texorpdfstring{$\td$}{td}-rigid pairs in \texorpdfstring{$\Lambda(n,l)$}{L(n,l)}}

Our aim is to show the converse of Corollary \ref{cor:silting give strongly maximal} in the case of $\Lambda=\Lambda(n,l)$. Hence let $\Lambda=\Lambda(n,l)$ be an algebra that satisfies the conditions of Theorem \ref{thm:VasoClassifyAcyclicCluster}. To attain our goal we see in light of Lemma \ref{lemma:taud-rigid iff presilting} that we only need to show that for a summand-maximal $\td$-rigid pair $(M,P)$ the complex $\mathbf{P}^\bullet_{(M,P)}$ generates the triangulated category $\HomotopyC{\proj{\Lambda}}{b}$, i.e. that $\thick{\mathbf{P}^\bullet_{(M,P)}}=\HomotopyC{\proj{\Lambda}}{b}$. The strategy used for this mirrors largely those of the previous sections, meaning we partition the problem into smaller problems centered around admissible configurations $(M_i,\ldots,M_{i+k})$. 

To show that $\thick{\mathbf{P}^\bullet_{(M,P)}}=\HomotopyC{\proj{\Lambda}}{b}$ we need to show that each indecomposable projective module appears in $\thick{\mathbf{P}^\bullet_{(M,P)}}$. This is clear for indecomposable projective summands of $M$ and $P$; let us call this collection of indecomposable projective modules $\mathcal{P}_1$. As a next step we may find all complexes in $\thick{\mathbf{P}^\bullet_{(M,P)}}$ that contains some of these indecomposable projective modules, together with one other indecomposable projective module. If this other indecomposable projective module appears as a stalk complex, then it is easy to see that it also belongs to $\thick{\mathbf{P}^\bullet_{(M,P)}}$. This gives us a new collection of indecomposable projective modules, say $\mathcal{P}_2$, that necessarily includes $\mathcal{P}_1$ and is included in $\thick{\mathbf{P}^\bullet_{(M,P)}}$. We may repeat this process inductively with the hope of eventually obtaining a collection $\mathcal{P}_k$ containing all of the indecomposable projective modules. 

In paragraph \ref{subsubsec: Restricting silting local} we show that the above strategy works. We begin by obtaining some useful complexes in $\thick{\mathbf{P}^\bullet_{(M,P)}}$: these are the truncated projective resolutions of modules in $M$. We collect the terms of such projective resolutions in sets $\mathcal{P}(X)$ for each indecomposable module $X$. Then we introduce the notion of reducing sets which encode the inductive procedure we describe in the previous paragraph. By this translation we see that $\thick{\mathbf{P}^\bullet_{(M,P)}}=\HomotopyC{\proj{\Lambda}}{b}$ holds if a particular type of reducing set exists for $[1,n]\setminus (\red\cup\blue)$. We go on to show that this type of reducing set can be constructed from reducing sets of $\Xi({i,i+k})\setminus (\red\cup\blue)$ for each admissible configuration $(M_i,\ldots,M_{i+k})$ of a summand-maximal $\td$-rigid pair $(M,P)$.

Then in paragraph \ref{subsubsec: local silting admissible} we are concerned with showing that for each full admissible type $(M_i,\ldots,M_{i+k})$ in Definition \ref{def:admissible configurations} we can find a reducing set of $\Xi({i,i+k})\setminus (\red\cup\blue)$. We conclude the paragraph at the end by proving Theorem \ref{Theorem C}.

\subsubsection{Restricting silting to local conditions}\label{subsubsec: Restricting silting local}

\begin{lemma}[{\cite[Corollary 4.5]{VASO20192101}}]
\label{lemma:Truncated projective presentation}
    Let \ind{a}{b} be a non-projective indecomposable module. Then
    \[
    \mathbf{P}^{-j}(\ind{a}{b})=\begin{cases}
        P(a-\frac{j-1}{2}l-1),&\text{ if }j\geq 1\text{ is odd,}\\
        P(b-\frac{j}{2}l),&\text{ if }j\geq 0\text{ is even.}
    \end{cases}
    \]
\end{lemma}

As a tool for later, we define for an indecomposable module $X=\ind{a}{b}$ the set $\mathcal{P}(X)$\index[symbols]{Pd@$\mathcal{P}(X)$} given by the indices of the indecomposable projective modules appearing in its truncated minimal projective resolution $\trunc{\mathbf{P}^\bullet}{\geq -d}(X)$. In particular, since $\trunc{\mathbf{P}^{-t}}{\geq -d}(X)$ is indecomposable by Lemma \ref{lemma:Truncated projective presentation}, we have that
\[
\begin{split}
    \mathcal{P}(X)=&\{i\in[1,n]\ |\ P(i)=\trunc{\mathbf{P}^{-t}}{\geq -d}(X) \text{ for some } t\}\\
    =&[1,n]\cap\{x_d<\cdots<x_1<x_0\},
\end{split}
\]
where $x_j=b-\tfrac{j}{2}l$ if $j$ is even and $x_j=a-\tfrac{j-1}{2}l-1$ if $j$ is odd. Observe also that $x_{j-2}-x_{j}=l$ for $0\leq j-2<j\leq d$, and that if $i,i'\in\mathcal{P}(X)$, then there exist unique $t_i,t_{i'}$ such that $P(i)=\trunc{\mathbf{P}^{-t_i}}{\geq -d}(X)$ and $P(i')=\trunc{\mathbf{P}^{-t_{i'}}}{\geq -d}(X)$, and then $i\leq i'$ is equivalent to $t_{i}\leq t_{i'}$. 

\begin{example}
\label{Example: Truncated projective resolution}
    Consider the indecomposable $\tau_4$-rigid module $\ind{s_3}{s_3+2}=\ind{15}{17}$ of $\Lambda(23,4)$. Then, we have
    \[
    \mathcal{P}(\ind{15}{17})=\{x_4=9<x_3=10<x_2=13<x_1=14<x_0=17\}
    \]
    
    \[
    \includegraphics{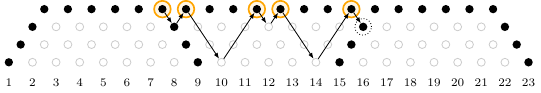}
    \]
\end{example}

We record the following important corollary of Lemma \ref{lemma:Truncated projective presentation} which arises from a couple of easy calculations.

\begin{corollary}\label{Corollary: trunc proj resolution contained in Xi}\label{Corollary:conf (III) gives connected intervals}
    Let $(M,P)$ be a $\cC$-pair with $M_i\neq 0$ for $2\leq i\leq p$. Then
    \[
    \bigcup_{X\in\mathrm{ind}M_i}\mathcal{P}(X)\subseteq \Xi_i.
    \]
    Moreover, when $m_i=l-1$, this strengthens to an equality.
\end{corollary}

For a basic module $M\in\modfin\Lambda$ and a subset $\Gamma$ of $[1,n]$ we let $\mathcal{P}_{\Gamma}(M)$\index[symbols]{Pe@$\mathcal{P}_{\Gamma}(M)$} be the set
\[
\mathcal{P}_{\Gamma}(M)\coloneqq\{\mathcal{P}(X)\cap\Gamma\ |\ X \text{ an indecomposable summand of M}\}.
\]
This set is in the heart and center of the following arguments. In order for it to reach its full potential we need one additional notion, namely reducing sets.

\begin{definition}
\label{def:Reducing set}
    Let $\Gamma$ be a finite set with $k+1$ elements. A subset $\Psi\subseteq \mathbb{P}(\Gamma)$ of the power set of $\Gamma$ is called a \emph{reducing set of $\Gamma$}\index[definitions]{reducing set} if there exists an ordering $\alpha_0,\alpha_1,\ldots,\alpha_k$ of the elements of $\Gamma$ such that for $0\leq i\leq k$ we have $\{\alpha_i\}\in \Psi_i$, where 
    \[
    \Psi_i = \begin{cases}
        \Psi, &\mbox{if $i=0$,} \\
        \{\chi\setminus \{\alpha_{i-1}\} \mid \chi\in \Psi_{i-1}\}, &\mbox{if $1\leq i\leq k$.}
    \end{cases}
    \]
\end{definition}

\begin{example}
For any finite set $\Gamma$, the set $\{\{\beta\}\mid \beta\in \Gamma\}$ is a reducing set, which we call the \emph{trivial reducing set}.
\end{example}

We have some immediate properties of reducing sets which contrary to their simplicity will play a significant part in the following.
\begin{lemma}
\label{lemma:Decompose reducing sets}
    Let $\Gamma=\Gamma^1\sqcup\Gamma^2$ be a finite set and $\Psi=\Psi^1\cup\Psi^2$ a subset of the power set of $\Gamma$. Assume that $\Psi^1$ is a reducing set of $\Gamma^1$ and $\widetilde{\Psi}^2\coloneqq \{\xi\cap \Gamma^2\ | \ \xi\in\Psi^2\}$ is a reducing set of $\Gamma^2$, then $\Psi$ is a reducing set of $\Gamma$. 
\end{lemma}

\begin{lemma}
\label{lemma:Helping reducing tuples and singleton}
    Let $\Gamma$ be a finite set and $\Psi$ a subset of the power set $\mathbb{P}(\Gamma)$ such that 
    \begin{enumerate}
        \item[(a)] $\bigcup_{\chi\in\Psi}\chi=\Gamma$,
        \item[(b)] for all $\chi\in \Psi$, $\abs{\chi}\leq 2$,
        \item[(c)] there exists a unique $\alpha\in\Gamma$ such that $\chi\cap\xi=\{\alpha\}$ for all $\chi,\xi\in \Psi$, and $\{\alpha\}\in\Psi$
    \end{enumerate}
    Then $\Psi$ is a reducing set of $\Gamma$.
\end{lemma}

We can now obtain the thickness condition for silting from the existence of a reducing set as per the following lemma.

\begin{lemma}\label{lemma:strongly maximal is thick iff reducing}
    Let $(M,P)$ be a summand-maximal $\td$-rigid pair and $\Gamma\coloneqq[1,n]\setminus (\red\cup\blue)$. If $\Psi\coloneqq\mathcal{P}_{\Gamma}(\bigoplus_{i=2}^p M_i)$ is a reducing set of $\Gamma$, then $\thick{\mathbf{P}^\bullet_{(M,P)}
    }=\HomotopyC{\proj(\Lambda)}{b}$.
\end{lemma}

\begin{proof}
    We know that $\thick{\Lambda}=\HomotopyC{\proj(\Lambda)}{b}$, hence it is enough to show that for all $i\in[1,n]$, the indecomposable projective $P(i)$ lies in $\mathcal{T}\coloneqq\thick{\mathbf{P}^\bullet_{(M,P)}}$. Immediately, we observe that for $i\in \red\cup\blue$, we have $P(i)\in\mathcal{T}$, hence we only need to consider the projective modules $P(i)$ for $i\in \Gamma$.

    An important, albeit simple, observation we make use of is the following. Fix  $k\in\mathbb{Z}$ and assume that a complex $\mathbf{Q}^\bullet$ and the stalk complexes $\mathbf{Q}^i$ for $i\neq k$ lie in $\mathcal{T}$. Then the stalk complex $\mathbf{Q}^k$ also lies in $\mathcal{T}$. Let us justify this briefly. The triangle $\trunc{\mathbf{Q}^\bullet}{>k}\rightarrow \mathbf{Q}^\bullet\rightarrow\trunc{\mathbf{Q}^\bullet}{\leq k}\rightarrow \shift{\trunc{\mathbf{Q}^\bullet}{>k}}{1}$, shows that $\trunc{\mathbf{Q}^\bullet}{\leq k}\in \mathcal{T}$, since both $\mathbf{Q}^\bullet$ and $\shift{\trunc{\mathbf{Q}^\bullet}{>k}}{1}$ lie in $\mathcal{T}$. Now, the triangle $\shift{\mathbf{Q}^k}{-k}\rightarrow\trunc{\mathbf{Q}^\bullet}{\leq k}\rightarrow \trunc{\mathbf{Q}^\bullet}{<k}\rightarrow\shift{\shift{\mathbf{Q}^k}{-k}}{1}$ shows that $\mathbf{Q}^k\in\mathcal{T}$.
    
    The rest of the proof is done iteratively. Since $\Psi$ is a reducing set, there exists an ordering $\alpha_0,\alpha_1,\ldots,\alpha_k$ of the elements of $\Gamma$ as in Definition \ref{def:Reducing set}. Then
    \[
    \{\alpha_0\} \in \Psi = \{\mathcal{P}(X)\cap \Gamma \mid X \text{ is an indecomposable summand of $M_2\oplus\cdots\oplus M_p$}\}.
    \]
    Hence there exists an indecomposable summand $X_{\alpha_0}$ of $M_2\oplus\cdots\oplus M_p$ such that $\mathcal{P}(X_{\alpha_0})\cap \Gamma_0=\{\alpha_0\}$. By definition of $\mathcal{P}(X_{\alpha_0})$ there exists $t_0\in\mathbb{Z}$ such that $\trunc{\mathbf{P}^{t_0}}{\geq -d}(X_{\alpha_0})=P(\alpha_0)$ and $\trunc{\mathbf{P}^i}{\geq -d}(X_{\alpha_0})$ lies in $\mathcal{T}$ for $i\neq t_0$ (since there are no repeats of indecomposable summands in a projective resolution of $X_{\alpha_0}$). Since $X_{\alpha_0}$ is a summand of $M_2\oplus\cdots\oplus M_p$, the complex $\trunc{\mathbf{P}^\bullet}{\geq -d}(X_{\alpha_0})$ also lies in $\mathcal{T}$. Hence from the discussion above also $\trunc{\mathbf{P}^{t_0}}{\geq -d}(X_{\alpha_0})=P(\alpha_0)$ lies in $\mathcal{T}$.
    
    Now assume that $P(\alpha_0),P(\alpha_1),\ldots,P(\alpha_{j-1})\in \mathcal{T}$ and we show that $P(\alpha_j)\in\mathcal{T}$. As above there exists an indecomposable summand $X_{\alpha_j}$ of $M_2\oplus\cdots \oplus M_p$ such that $\mathcal{P}(X_{\alpha_j})\cap \Gamma_j=\{\alpha_j\}$. Since $\Gamma\setminus\Gamma_j=\{\alpha_0,\alpha_1,\ldots,\alpha_{j-1}\}$, we have that $\mathcal{P}(X_{\alpha_j})\cap \Gamma \subseteq\{\alpha_0,\alpha_1,\ldots,\alpha_j\}$. By induction hypothesis we have that $P(\alpha_0),P(\alpha_1),\ldots,P(\alpha_{j-1})\in\mathcal{T}$. Since $\alpha_j\in \mathcal{P}(X_{\alpha_j})\cap \Gamma$, there exists $t_j\in\mathbb{Z}$ such that $\trunc{\mathbf{P}^{t_j}}{\geq -d}(X_{\alpha_j})=P(\alpha_j)$. Moreover, as in the base case, the complexes $\trunc{\mathbf{P}^\bullet}{\geq -d}(X_{\alpha_j})$ and $\trunc{\mathbf{P}^i}{\geq -d}(X_{\alpha_0})$, for $i\neq t_j$, lie in $\mathcal{T}$ by assumption. Again we conclude that $\trunc{\mathbf{P}^{t_j}}{\geq -d}(X_{\alpha_j})=P(\alpha_j)$ lies in $\mathcal{T}$.
\end{proof}

We are now in the home stretch of this section. As sketched in the introduction of the paragraph 
the strategy is to show that for a summand-maximal $\td$-rigid pair $(M,P)$ and $\Gamma=[1,n]\setminus(\red\cup \blue)$, the set $\mathcal{P}_{\Gamma}(\oplus_{i=2}^pM_i)$ is a reducing set of $\Gamma$. To manage this, we look at the existence of a reducing set more locally. 

\begin{lemma}\label{lemma:reducing set can be computed locally}
    Let $(M,P)$ be a summand-maximal $\td$-rigid pair with diagonal partition $(T_1,\ldots,T_k)$. Then the following hold.
    \begin{enumerate}
        \item[(a)] $\mathcal{P}_{\Xi(T_j)}(M_{t_{i,1}}\oplus\cdots\oplus M_{t_{i,2}})=\{\varnothing\}$ when $i\neq j$.
        \item[(b)] Let $\Gamma=[1,n]\setminus(R\cup B)$. Let $\Gamma_{T_j}=\Xi(T_j)\setminus (R\cup B)$. Then $\mathcal{P}_{\Gamma}(\bigoplus_{i=2}^pM_i)$ is a reducing set of $\Gamma$ if and only if for every $j\in[1,k]$ the set $\mathcal{P}_{\Gamma_{T_j}}(M_{t_{j,1}}\oplus\cdots\oplus M_{t_{j,2}})$ is a reducing set of $\Gamma_{T_j}$.
    \end{enumerate}
\end{lemma}

\begin{proof}
    A summand-maximal $\td$-rigid pair is necessarily well-configured, hence by Definition \ref{def:well-configured}(b)(ii), we have $\Xi(T_j)\cap\Xi(T_i)=\varnothing$ when $i\neq j$. Using Corollary \ref{Corollary: trunc proj resolution contained in Xi} and the fact that $\Xi(T_i)=\bigcup_{\alpha \in T_i}\Xi_\alpha$, we therefore obtain (a). 

    Recall that for a well-configured $\td$-rigid pair $(M,P)$ with non-diagonal support $(R,B)$ we have $(R\cup B)\cap \Theta=\Theta$ and $\Theta\sqcup \Xi=[1,n]$ where
    \(
    \Xi=\bigsqcup_{j=1}^k \Xi(T_j).
    \)
    Hence, (b) follows from (a) and repeated use of Lemma \ref{lemma:Decompose reducing sets}.
\end{proof}

\subsubsection{Local silting through admissible configurations.}\label{subsubsec: local silting admissible}
Now we study each type of admissible configuration separately and show the existence of reducing sets for each such type locally.

\begin{lemma}\label{lemma:reducing set in types (I),(II),(IV),(V)}
    Let $(M,P)$ be a summand-maximal $\td$-rigid pair. Assume that $(M_i,\ldots,M_{i+k})$ is a full admissible configuration and let $\Gamma=\Xi(i,i+k)\setminus (R\cup B)$.
    \begin{enumerate}
        \item[(a)] If the configuration is of type (I), then $\mathcal{P}_\Gamma(M_i)$ is a reducing set of $\Gamma=\notblue{i}$.
        \item[(b)] If the configuration is of type (II), then $\mathcal{P}_\Gamma(M_i)$ is a reducing set of $\Gamma=\notred{i}$.
        \item[(c)] If the configuration is of type (IV) or (VI) with $\Xi$ rigid, then $\mathcal{P}_\Gamma(M_i)\cup \mathcal{P}_\Gamma(M_{i+1})$ is a reducing set of $\Gamma=\notred{i}\cup\notred{i+1}$.
        \item[(d)] If the configuration is of type (V) or (VII) with $\Xi$ support, then $\mathcal{P}_\Gamma(M_i)\cup\mathcal{P}_\Gamma(M_{i+1})$ is a reducing set of $\Gamma=\notblue{i}\cup\notblue{i+1}$. 
    \end{enumerate}
\end{lemma}

\begin{proof}
    If $d$ is odd, then $l=2$ by Theorem \ref{thm:VasoClassifyAcyclicCluster} and so all admissible configurations are of type (III). Hence we assume that $d$ is even.
    \begin{enumerate}
        \item[(a)] This part can be proved similarly to (b), hence we only consider (b) here.
        \item[(b)] The configuration is of type (II), hence $k=0$, $\Xi_i$ is rigid and $\Gamma=\Xi_i\setminus(R\cup B)=\notred{i}$. Further $1<l_i\leq n_i=l-1$. From (\ref{eq:the definition of notblue and notred}) we get
        \[
        \notred{i}=\begin{cases}
            [s_{i-1}-(l-l_i)+1,s_{i-1}],&\text{ if }i\text{ is odd,}\\
            [s_{i-1},s_{i-1}+(l-l_i)-1],&\text{ if }i\text{ is even.}
        \end{cases}
        \]
        Let $X(h)$ be an indecomposable summand of $M_i$ with length $h\in [l_i,l-1]$. By (\ref{eq:the definition of M_i}), (\ref{eq:difference of consecutive simples}) and Lemma \ref{lemma:Truncated projective presentation} we obtain that
        \[
        \mathcal{P}(X(h))\cap\notred{i}=\begin{cases}
            \{s_{i-1}-(l-h)+1\},&\text{ if }i\text{ is odd,}\\
            \{s_{i-1}\},&\text{ if }i\text{ is even and }h=l_i,\\
            \{s_{i-1},s_{i-1}+(l-h)\},&\text{ if }i\text{ is even and }l_i+1\leq h.
        \end{cases}
        \]
        Thus, if $i$ is odd, then $\mathcal{P}_{\notred{i}}(M_i)$ is the trivial reducing set, and if $i$ is even, then $\mathcal{P}_{\notred{i}}(M_i)$ is reducing by Lemma \ref{lemma:Helping reducing tuples and singleton}.
        \item[(c)] This follows as in (b) using Lemma \ref{lemma:Decompose reducing sets} with $\Gamma^1=\notred{i+1}$, $\Gamma^2=\Gamma\setminus\Gamma^1$, $\Psi^1=\mathcal{P}_{\Gamma}(M_{i+1})$ and $\Psi^2=\mathcal{P}_\Gamma(M_i)$.
        \item[(d)] This follows as in (a) using Lemma \ref{lemma:Decompose reducing sets} with $\Gamma^1=\notblue{i}$, $\Gamma^2=\Gamma\setminus\Gamma^1$, $\Psi^1=\mathcal{P}_{\Gamma}(M_{i})$ and $\Psi^2=\mathcal{P}_\Gamma(M_{i+1})$. \qedhere
    \end{enumerate}
\end{proof}

\begin{example}
    Consider the following summand-maximal $\tau_2$-rigid pair $({\color{rigid}M},{\color{support}P})$ of $\Lambda(20,5)$. The diagonal partition is $(T_1)=([4,4])$ and $({\color{rigid}M_4})$ is a full admissible configuration of type (I).
\[
\includegraphics{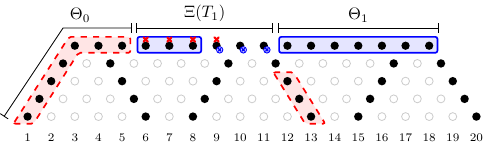}
\]
We have that $\Gamma=\Xi(T_1)\setminus(R\cup B)=\{11,12,13\}$ and
\begin{align*}
    \mathcal{P}(\ind{11}{13})&=\{8<10<13\}
    &
    \mathcal{P}(\ind{12}{13})&=\{8<11<13\}\\
    \mathcal{P}(\ind{13}{13})&=\{8<12<13\}.
\end{align*}
Hence, we have $\mathcal{P}_\Gamma(M_4)=\{\{13\},\{11,13\},\{12,13\}\}$, which is clearly a reducing set of $\Gamma$.
\end{example}

\begin{lemma}\label{lemma:reducing set in type (III)}
    Let $(M,P)$ be a summand-maximal $\td$-rigid pair and $(M_i,\ldots,M_{i+k})$ a full admissible configuration of type (III). Then $\mathcal{P}_\Gamma(M_i)$ is a reducing set of $\Gamma\coloneqq\Xi(i,i+k)\setminus (R\cup B)$
\end{lemma}

\begin{proof}
    If $l=2$, then a direct calculation shows that $\Gamma = \mathcal{P}_\Gamma(M_i)=\{a\}$ for some $a\in [s_i-d,s_i]$ and the claim follows immediately.

    Assume that $l>2$. We only prove the case of $i$ being odd, the even case is essentially the same. Since the admissible configuration is of type (III), we have that $k=0$, that $\Xi(i,i+k)=\Xi_i$ and that $\Xi_i$ is either rigid, support or support to rigid at some $x\in\Xi_i\setminus\notblue{i}$. If
    \begin{itemize}
        \item $\Xi_i$ is rigid, then $\Gamma=\notred{i}=[s_{i-1}-l+2,s_{i-1}]$,
        \item $\Xi_i$ is support, then $\Gamma=\notblue{i}=[s_i,s_i+l-2]$, 
        \item $\Xi_i$ is support to rigid, then $\Gamma=[x+1,x+l-1]$ for some $x\in\Xi_i\setminus\notblue{i}$.
    \end{itemize}
    In any case, we have $\abs{\Gamma}=l-1$, and from the description after Lemma \ref{lemma:Truncated projective presentation}, we see that $\abs{\mathcal{P}(X)\cap \Gamma}\leq 2$ for all indecomposable summands $X$ of $M_i$. Observe now that $\Xi_i=[s_{i-1}-l+2,s_i+l-2]$ is equal to $\bigcup_{X\in\mathrm{ind}M_i}\mathcal{P}(X)$ by Corollary \ref{Corollary:conf (III) gives connected intervals}. Therefore $\bigcup_{\chi\in\mathcal{P}_\Gamma(M_i)}\chi=\Gamma.$      
    Next, for $h\in \{1,\ldots,l-1\}$ let $X(h)=\ind{s_i}{s_i+h-1}$ so that $M_i=\bigoplus_{h=1}^{l-1}X(h)$. Using Lemma \ref{lemma:Truncated projective presentation} we see that 
    \[
    \mathbf{P}^{-j}(X(h))=\begin{cases}
        P(s_i-\frac{j-1}{2}l-1),&\text{ if }j\geq 1\text{ is odd,}\\
        P(s_i+h-1-\frac{j}{2}l),&\text{ if }j\geq 0\text{ is even.}
    \end{cases}
    \]
    It follows that either $(\mathcal{P}(X(h))\cap\Gamma)\cap(\mathcal{P}(X(h'))\cap\Gamma)=\varnothing$ for any $h\neq h'$, in which case $\mathcal{P}_\Gamma(M_i)$ is the trivial reducing set of $\Gamma$, or there exists $\alpha=s_i-\tfrac{j-1}{2}l-1\in \Gamma$ for some $j$ such that $\{\alpha\}= (\mathcal{P}(X(h))\cap\Gamma)\cap(\mathcal{P}(X(h'))\cap\Gamma)$ for any $h\neq h'$ of $M_i$, in which case $\mathcal{P}_\Gamma(M_i)$ is a reducing set of $\Gamma$ by Lemma \ref{lemma:Helping reducing tuples and singleton}.
\end{proof}

\begin{example}
    Let $({\color{rigid}M},{\color{support}P})$ be the following summand-maximal $\tau_4$-rigid pair of $\Lambda(23,4)$ with diagonal partition $(T_1)=([3,3])$ of type (III). 
\[
\includegraphics{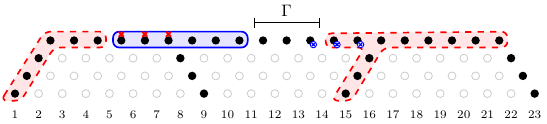}
\]
    Here we got 
    \[
    \mathcal{P}_\Gamma(M)=\{\{14,15\},\{14\},\{13,14\}\}
    \]
    which is clearly a reducing set of $\Gamma=[1,23]\setminus (R\cup B)=[13,15]$.
\end{example}

We continue with admissible configurations of type (VI) and (VII).

\begin{lemma}\label{lemma:reducing set in type (VI)}
   Let $(M,P)$ be a summand-maximal $\td$-rigid pair and let $(M_{i},\ldots, M_{i+k})$ be a full admissible configuration.
   \begin{enumerate}
        \item[(a)] If the configuration is of type (VI), then any subinterval $\Gamma$ of
        $$
        \mathcal{I}=[s_{i}-(l-m_{i+1})+1,s_i+(l-1)-1].
        $$
        with length $l-1$ has $\mathcal{P}_\Gamma(M_i\oplus M_{i+1})$ as a reducing set.
        \item[(b)] If the configuration is of type (VII), then any subinterval $\Gamma$ of
        $$
        \mathcal{I}=[s_i-(l-1)+1,s_i-m_i+(l-1)].
        $$
        with length $l-1$ has $\mathcal{P}_\Gamma(M_i\oplus M_{i+1})$ as a reducing set.
   \end{enumerate}
\end{lemma}

\begin{proof}
    We only prove (a) as the proof for (b) is similar.

    The interval $\mathcal{I}$ can be rewritten as $[s_i-m_i,s_i+(l-1)-1]$, using $m_i+m_{i+1}=l-1$. Now, an arbitrary subinterval $\Gamma$ of $\mathcal{I}$ with length $l-1$ takes the form $\Gamma=[s_i-m_i+\gamma,s_i-m_i+\gamma+(l-1)-1]$ for some $\gamma\in[0,m_i]$. Let $X=\ind{s_i}{s_i+x-1}$ for $x\in[l_i,l-1]$ and $Y=\ind{s_{i+1}-y+1}{s_{i+1}}$ for $y\in[l_{i+1},l-1]$ be indecomposable summands of $M_i$ and $M_{i+1}$ respectively. With some calculations using (\ref{eq:difference of consecutive simples}) and Lemma \ref{lemma:Truncated projective presentation}, we see that
\[
\mathcal{P}(X)\cap \mathcal{I}=\begin{cases}
    \{s_i-1,\ s_i+l_i-1\},&\text{ if }x=l_i,\\
    \{s_i-(l-x)-1,\ s_i-1,\ s_i+x-1\},&\text{ else,}
\end{cases}
\]
and
\[
\mathcal{P}(Y)\cap \mathcal{I}=\{s_i,\ s_{i}+l-y\}.
\]
It is easily verified that the union of all sets in $\mathcal{P}_\mathcal{I}(M_i\oplus M_{i+1})$ is in fact $\mathcal{I}$, and hence, for any $\gamma\in[0,m_i]$ also the union of all sets in $\mathcal{P}_\Gamma(M_i\oplus M_{i+1})$ equals $\Gamma$.

Let us now first consider the case $\gamma=m_i$. Then $\mathcal{P}(X)\cap\Gamma=\{s_i+x-1\}$ and $\mathcal{P}(Y)\cap\Gamma=\{s_i,s_i+l-y\}$. We also have that $s_i+l_i-1=s_i+l-l_{i+1}$, since $l_i+l_{i+1}=l+1$. Hence, we see that $\mathcal{P}_\Gamma(M_i\oplus M_{i+1})$ is a reducing set of $\Gamma$, by ordering the elements of $\Gamma$ first by the elements appearing in the sets of $\mathcal{P}_\Gamma(M_i)$, then the element $s_i$, and then any ordering of the remaining elements in $\Gamma$, e.g.
\[
(s_i+(l-1)-1,\ s_i+(l-2)-1,\ \ldots,\ s_i+l_i-1=s_i+l-l_{i+1}, \ s_i,\ s_i+1,\ \ldots, \ s_i+l-l_{i+1}-1).
\]
We are now left with considering $0\leq\gamma\leq m_i-1$, which gives
\[
\mathcal{P}(X)\cap \Gamma=\begin{cases}
    \{s_i-(l-x)-1,\ s_i-1\},&\text{ if }x>\gamma+l_i,\\
    \{s_i-1\},&\text{ if }x=\gamma+l_i,\\
    \{s_i-1,\ s_i+x-1\},&\text{ if }x<\gamma+l_i.
\end{cases}
\]
Assume that $\gamma=0$. Then for $y=l_{i+1}$ we have $\mathcal{P}(Y)\cap\Gamma=\{s_i\}$. Thus we see that $\mathcal{P}_\Gamma(M_i\oplus M_{i+1})$ is a reducing set of $\Gamma$ by any ordering starting as $(s_i,\ s_i-1,\ \ldots)$, since after removing these two elements all the sets in $\mathcal{P}_\Gamma(M_i\oplus M_{i+1})$ contain either one or zero elements.

Assume now that $1\leq \gamma\leq m_i-1$. Then for $x=l_i$ we have $\mathcal{P}(X)\cap \Gamma=\{s_i-1,s_i+l_i-1\}$. Further, we have seen that $s_i+l_i-1=s_i+l-l_{i+1}$. Hence $\mathcal{P}_\Gamma(M_i\oplus M_{i+1})$ is a reducing set of $\Gamma$ by any ordering starting as $(s_i-1,\ s_i+l_i-1,\ s_i,\ \ldots)$, since after removing these elements every non-empty set in $\mathcal{P}_\Gamma(M_i\oplus M_{i+1})$ has only one element. 
\end{proof}

\begin{example}
    Consider the following summand-maximal $\tau_4$-rigid pair $({\color{rigid}M},{\color{support}P})$ of $\Lambda(23,4)$ with a full admissible configuration $({\color{rigid}M_2},{\color{rigid}M_3})$ of type (VII) with $\Xi$ support to rigid. Then $\mathcal{I}=[7,11]$.
    \[
    \includegraphics{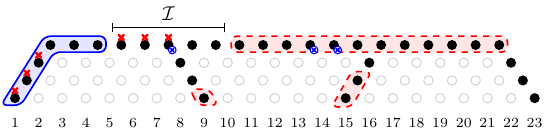}
    \]
    We have
    \begin{align*}
        \mathcal{P}(\ind{9}{9})&=\{1<4<5<8<9\}
        &
        \mathcal{P}(\ind{15}{15})&=\{7<10<11<14<15\}\\
        \mathcal{P}(\ind{15}{16})&=\{8<10<12<14<16\}.
    \end{align*}
    In $\mathcal{I}$ we find three subintervals of length $3$: $\Gamma_1=[7,9]$, $\Gamma_2=[8,10]$ and $\Gamma_3=[9,11]$, which can be seen to have respective reducing sets:
    \begin{align*}
        \mathcal{P}_{\Gamma_1}(M_2\oplus M_3)&=\{\{8,9\},\{7\},\{8\}\}
        &
        \mathcal{P}_{\Gamma_2}(M_2\oplus M_3)&=\{\{8,9\},\{10\},\{8,10\}\}\\
        \mathcal{P}_{\Gamma_3}(M_2\oplus M_3)&=\{\{9\},\{10,11\},\{10\}\}.
    \end{align*}
\end{example}

We are left with showing that also for the last configuration we obtain a reducing set of $\Xi(i,i+2)\setminus (R\cup B)$. The idea of the proof is quite similar to that of Lemma \ref{lemma:reducing set in type (VI)}, however the slight complication of working with three diagonals at once leads to a bit more juggling of indexes. It might be beneficial to revisit Definition \ref{def:admissible configurations}, and Example \ref{ex:Admissible type VIII} before moving on.

\begin{lemma}\label{lemma:reducing set in type (VIII)}
    Let $(M,P)$ be a summand-maximal $\td$-rigid pair and let $(M_i,M_{i+1},M_{i+2})$ be a full admissible configuration of type (VIII). Any subinterval $\Gamma$ of
    \[
    \mathcal{I}=[s_{i+1}-l_{i+2}+1,s_i-n_i+(l-1)]
    \]
    with length $l-1$ has $\mathcal{P}_\Gamma(M_i\oplus M_{i+1}\oplus M_{i+2})$ as reducing set.
\end{lemma}
\begin{proof} We begin by noting that by assumption $d=2$, and that by Remark \ref{Remark: no three consecutive for l=3} also $l>3$. Now, we find it beneficial for the proof to rewrite $\mathcal{I}$ as 
$$[s_{i+1}-(l-l_{i+1}),s_{i+1}+n_{i+1}+2].$$
Let $\Gamma\subseteq\mathcal{I}$ be a subinterval of length $l-1$. For $0\leq k\leq 2$, we denote by $\Gamma_{i+k}$ the subset of $\Gamma$ consisting of integers in $\Gamma$ appearing in $\mathcal{P}(N)$ for indecomposable summands $N$ of $M_{i+k}$. Further, letting $X$, $Y$ and $Z$ be indecomposable summands of $M_i$, $M_{i+1}$ and $M_{i+2}$ respectively, we get by (\ref{eq:difference of consecutive simples}) that
\begin{align*}
    X&=\ind{s_{i+1}-x-1}{s_{i+1}-2},
    &
    Y&=\ind{s_{i+1}}{s_{i+1}+y-1},\\
    Z&=\ind{s_{i+1}+(l-z)+1}{s_{i+1}+l},
\end{align*}
where $1\leq x\leq n_i$, $l_{i+1}\leq y\leq n_{i+1}$ and $l_{i+2}\leq z\leq l-1$. 

With the help of Lemma \ref{lemma:Truncated projective presentation} we calculate the following sets:
\[
\mathcal{P}(X)\cap\mathcal{I}=\begin{cases}
    \{s_{i+1}-2\},&\text{ if }m_{i+1}=1\text{ and }x=l-n_{i+1}-1=n_i, \\
    \{s_{i+1}-x-2\ <\ s_{i+1}-2\},&\text{ else,}
\end{cases}
\]
\[
\mathcal{P}(Z)\cap\mathcal{I}=\begin{cases}
    \{s_{i+1}\},&\text{ if }m_{i+1}=1\text{ and }z=l-l_{i+1}+1=l_{i+2},\\
    \{s_{i+1}\ <\ s_{i+1}+l-z\},&\text{ else,}
\end{cases}
\]
and 
\[
\mathcal{P}(Y)\cap\mathcal{I}=\begin{cases}
    \{s_{i+1}-1\},&\text{ if }m_{i+1}=1,\\
    \{s_{i+1}-1\ <\ s_{i+1}+l_{i+1}-1\},&\text{ if }m_{i+1}>1\text{ and }y=l_{i+1},\\
    \{s_{i+1}+n_{i+1}-l-1\ <\ s_{i+1}-1\},&\text{ if }m_{i+1}>1\text{ and }y=n_{i+1},\\
    \{s_{i+1}+y-l-1\ <\ s_{i+1}-1\ <\ s_{i+1}+y-1\},&\text{ if }m_{i+1}>1\text{ and }l_{i+1}<y<n_{i+1}.
\end{cases}
\]
From this it is easily verified that $\Gamma=\Gamma_{i}\cup\Gamma_{i+1}\cup \Gamma_{i+2}$. We now have two different cases to consider, either $m_{i+1}=1$ or $m_{i+1}>1$. The first of these is significantly easier than the other.

\textbf{Case $m_{i+1}=1$.} In particular $\Gamma=\mathcal{I}$. Further, we see that $\{\Gamma_{i},\Gamma_{i+1},\Gamma_{i+2}\}$ gives a partition of $\Gamma$, and by Lemma \ref{lemma:Helping reducing tuples and singleton} we see that $\mathcal{P}_{\Gamma_{i+k}}(M_{i+k})$ is a reducing set of $\Gamma_{i+k}$ for $0\leq k\leq 2$. Hence, $\mathcal{P}_\Gamma(M_i\oplus M_{i+1}\oplus M_{i+2})$ is a reducing set of $\Gamma$.

\textbf{Case $m_{i+1}>1$.} For some $\gamma\in[0,m_i-1]$, we have 
\(
\Gamma=[s_{i+1}-(l-l_{i+1})+\gamma\ ,\ s_{i+1}+l_{i+1}+\gamma-2],
\)
and 
\[
\mathcal{P}(Y)\cap \Gamma=\begin{cases}
    \{s_{i+1}-1\ <\ s_{i+1}+y-1\},&\text{ if }y<l_{i+1}+\gamma,\\
    \{s_{i+1}-1\},&\text{ if }y=l_{i+1}+\gamma,\\
    \{s_{i+1}+y-l-1\ <\ s_{i+1}-1\},&\text{ if }y>l_{i+1}+\gamma.
\end{cases}
\]
It follows by Lemma \ref{lemma:Helping reducing tuples and singleton} that $\mathcal{P}_\Gamma(M_{i+1})$ is a reducing set of $\Gamma_{i+1}$. We can now observe that 
\[
\mathcal{P}(X)\cap\mathcal{P}(Y)=\begin{cases}
    \{s_{i+1}+n_{i+1}-l-1\},&\text{ if }y=n_{i+1}\text{ and }x=n_i,\\
    \varnothing,&\text{ else.}
\end{cases}
\]
Either $s_{i+1}+n_{i+1}-l-1$ is in $\Gamma_i$, or $\Gamma_{i}$ has $\mathcal{P}_\Gamma(M_{i})$ as reducing set by Lemma \ref{lemma:Helping reducing tuples and singleton}. Similarly, we observe that
\[
\mathcal{P}(Y)\cap \mathcal{P}(Z)=\begin{cases}
    \{s_{i+1}+l_{i+1}-1\},&\text{ if }y=l_{i+1}\text{ and }z=l_{i+2},\\
    \varnothing,&\text{ else,}
\end{cases}
\]
and thus either $s_{i+1}+n_{i+1}-l-1$ is in $\Gamma_{i+2}$ or $\Gamma_{i+2}$ has $\mathcal{P}_\Gamma(M_{i+2})$ as reducing set. 

Hence, letting $\Gamma^1=\Gamma_{i+1}$, $\Gamma^2=\Gamma\setminus \Gamma^1$, $\Psi^1=\mathcal{P}_\Gamma(M_{i+1})$ and $\Psi^2=\mathcal{P}_\Gamma(M_i\oplus M_{i+2})$, it follows by Lemma \ref{lemma:Decompose reducing sets} that $\mathcal{P}_\Gamma(M_i\oplus M_{i+1}\oplus M_{i+2})$ is a reducing set of $\Gamma$.
\end{proof}

\begin{example}
We consider a summand-maximal $\tau_2$-rigid pair $({\color{rigid}M},{\color{support}P})$ of $\Lambda(23,6)$ with a full admissible configuration $({\color{rigid}M_4},{\color{rigid}M_5},{\color{rigid}M_6})$ of type (VIII). Then $\mathcal{I}=[13,19]$
    \[
\includegraphics{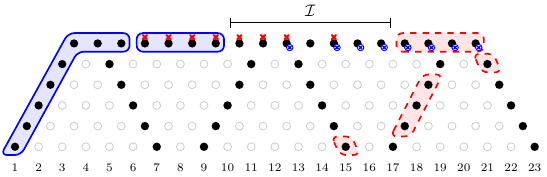}
\]
We calculate the truncated projective resolutions:
\begin{align*}
    \mathcal{P}(\ind{15}{15})&=\{9<14<15\}
    &
    \mathcal{P}(\ind{17}{18})&=\{12<16<18\}\\
    \mathcal{P}(\ind{17}{19})&=\{13<16<19\}
    &
    \mathcal{P}(\ind{17}{20})&=\{14<16<20\}\\
    \mathcal{P}(\ind{19}{23})&=\{17<18<23\}
\end{align*}
There are three possible subintervals of $\mathcal{I}$ of length $5$: $\Gamma_1=[13,17]$, $\Gamma_2=[14,18]$ and $\Gamma_3=[15,19]$. These have reduction sets given respectively by:
\[
\begin{split}
    \mathcal{P}_{\Gamma_1}(M_4\oplus M_5\oplus M_6)&=\{\{14,\ 15\},\ \{16\},\ \{13,\ 16\},\ \{14,\ 16\},\ \{17\}\}\\
    \mathcal{P}_{\Gamma_2}(M_4\oplus M_5\oplus M_6)&=\{\{14,\ 15\},\ \{16,\ 18\},\ \{16\},\ \{14,\ 16\},\ \{17,\ 18\}\}\\
    \mathcal{P}_{\Gamma_3}(M_4\oplus M_5\oplus M_6)&=\{\{15\},\ \{16,\ 18\},\ \{16,\ 19\},\ \{16\},\ \{17,\ 18\}\}\\
\end{split}
\]

\end{example}

Having done all the grunt work we are now ready to sew it all together into our main theorem of the section: Theorem \ref{Theorem C}. 

\begin{theorem}\label{Theorem:strongly maximal and silting}
    Let $\Lambda=\Lambda(m,l)$ be as in Theorem \ref{thm:VasoClassifyAcyclicCluster}, and let $(M,P)$ be a $\cC$-pair. Then $(M,P)$ is a summand-maximal $\td$-rigid pair if and only if $\mathbf{P}^\bullet_{(M,P)}$ is a $(d+1)$-silting complex in $\HomotopyC{\proj{\Lambda}}{b}$.
\end{theorem}

\begin{proof}
    If $\mathbf{P}^\bullet_{(M,P)}$ is a $(d+1)$-silting complex in $\HomotopyC{\proj{\Lambda}}{b}$ then $(M,P)$ is a summand-maximal $\td$-rigid pair by Corollary \ref{cor:silting give strongly maximal} and Theorem \ref{thrm:taud tilting is well-configured}. For the opposite direction, assume that $(M,P)$ is a summand-maximal $\td$-rigid pair. By Lemma \ref{lemma:taud-rigid iff presilting} we have that $\shift{P}{d}\oplus\trunc{\mathbf{P}^\bullet}{\geq -d}(M)$ is a $(d+1)$-presilting complex. Let $\Gamma=[1,n]\setminus (\red\cup\blue)$. By Lemma \ref{lemma:strongly maximal is thick iff reducing} it remains to show that $\mathcal{P}_{\Gamma}(\bigoplus_{i=2}^{p}M_i)$ is a reducing set of $\Gamma$. Since $(M,P)$ is well-configured by Theorem \ref{thrm:taud tilting is well-configured}, this follows by Lemma \ref{lemma:reducing set can be computed locally}(b) together with Lemma \ref{lemma:reducing set in types (I),(II),(IV),(V)}, Lemma \ref{lemma:reducing set in type (III)}, Lemma \ref{lemma:reducing set in type (VI)} and Lemma \ref{lemma:reducing set in type (VIII)}. 
\end{proof}

\section{Summary and concluding remarks}
\makeatletter
\newcommand{\rightleftarrowss}[2]{%
  \mathrel{\mathop{%
    \vcenter{\offinterlineskip\m@th
      \ialign{\hfil##\hfil\cr
        \hphantom{$\scriptstyle\mspace{8mu}{#1}\mspace{8mu}$}\cr
        \rightarrowfill\cr
        \vrule height0pt width 2em\cr
        \leftarrowfill\cr
        \hphantom{$\scriptstyle\mspace{8mu}{#2}\mspace{8mu}$}\cr
        \noalign{\kern-0.3ex}
      }%
    }%
  }\limits^{#1}_{#2}}%
}
\makeatother

\makeatletter
\tikzset{
  edge node/.code={%
      \expandafter\def\expandafter\tikz@tonodes\expandafter{\tikz@tonodes #1}}}
\makeatother
\tikzset{
  Belongs/.style={
    draw=none,
    every to/.append style={
      edge node={node [sloped, allow upside down, auto=false]{$\in$}}}
  }
}

\label{Section:Comparing}
The aim of this section is to compare the theory we have developed in sections \ref{Section:strongly maximal td-rigid}, \ref{Section:d-torsion} and \ref{Section:d+1 silting} with the classical case $d=1$. Throughout this section we fix a finite-dimensional algebra $\Lambda$. 

\subsection{Comparison to classical \texorpdfstring{$\tau$}{t}-tilting theory} We start by recalling the definition of support $\tau$-tilting pairs.

\begin{lemma}[{\cite[Corollary 2.13]{adachi_-tilting_2014}}]\label{lem:equivalence for tau-tilting}
Let $(M,P)$ be a $\tau$-rigid pair over $\Lambda$. Then the following are equivalent.
\begin{enumerate}
    \item[(a)] If $(M\oplus N,P)$ is a $\tau$-rigid pair for some $N\in\modfin{\Lambda}$, then $N\in \add{M}$.
    \item[(b)] If $\Hom{\Lambda}{M}{\tau(N)}=0$, $\Hom{\Lambda}{N}{\tau(M)}=0$ and $\Hom{\Lambda}{P}{N}=0$ for some $N\in\modfin{\Lambda}$, then $N\in\add{M}$.
    \item[(c)] $\abs{M}+\abs{P}=\abs{\Lambda}$.
\end{enumerate}
In particular if $(M,P)$ satisfies any of the above conditions, then it is called a \emph{support $\tau$-tilting pair}\index[definitions]{support $\tau$-tilting pair}.
\end{lemma}

The equivalence of the above statements shows that if $(M,P)$ is a $\tau$-rigid pair over $\Lambda$, then $\abs{M}+\abs{P}\leq \abs{\Lambda}$. Based on this fact and these definitions, we have the following generalizations to the case $d>1$. 

\begin{definition}\label{def:tau_d tilting higher}
Let $\cC\subseteq\modfin{\Lambda}$ be a $d$-cluster tilting subcategory. Let $M\in\cC$ and $P\in\proj\Lambda$.
\begin{enumerate}
    \item[(a)] \cite[Def 2.4(3)]{ZHOU2023193} We say that $(M,P)$ is a \emph{support $\td$-tilting pair}\index[definitions]{support $\td$-tilting pair} if it is a $\td$-rigid pair and moreover it satisfies the following:
    \begin{enumerate}
        \item[(i)] if $(M\oplus N,P)$ is a $\td$-rigid pair for some $N\in\cC$, then $N\in\add{M}$, and
        \item[(ii)] if $Q\in \proj{\Lambda}$, then 
        \[
        Q \in \add{P} \iff \Hom{\Lambda}{Q}{M}=0.
        \]
    \end{enumerate}
    \item[(b)] \cite[Def 0.7]{JACOBSEN2020119}\cite[Def 2.4(4)]{ZHOU2023193} We say that $(M,P)$ is a \emph{maximal $\td$-rigid pair}\index[definitions]{$\td$-rigid pair!maximal} if it satisfies the following:
    \begin{enumerate}
        \item[(i)] if $N\in\cC$, then 
        \[
        N \in \add{M} \iff \begin{cases}
            \Hom{\Lambda}{M}{\td(N)}=0, \\
            \Hom{\Lambda}{N}{\td(M)}=0, \\
            \Hom{\Lambda}{P}{N}=0,
        \end{cases}
        \]
        and
        \item[(ii)] if $Q\in \proj{\Lambda}$, then 
        \[
        Q \in \add{P} \iff \Hom{\Lambda}{Q}{M}=0.
        \]
    \end{enumerate}
    \item[(c)] We say that $(M,P)$ is a \emph{summand-maximal $\td$-rigid pair}\index[definitions]{$\td$-rigid pair!summand-maximal} if it is a $\td$-rigid pair and for any other $\td$-rigid pair $(M',P')$ we have $\abs{M'}+\abs{P'}\leq \abs{M}+\abs{P}$.
\end{enumerate}
\end{definition}

It is easy to see that we have implications
\[
\begin{tikzpicture}
    \node (A) at (4,2) [] {$(M,P)$ is a maximal $\td$-rigid pair};
    \node (B) at ([yshift=-1pt]A.south) [] {$\Downarrow$};
    \node (C) at ([yshift=-1pt]B.south) [] {$(M,P)$ is a support $\td$-tilting pair};
    \node (D) at ([yshift=-1pt]C.south) [] {$\Downarrow$};
    \node (E) at ([yshift=-1pt]D.south) [] {$(M,P)$ is a $\td$-rigid pair};
    \node (F) at ([xshift=-1pt]C.west) [anchor=east] {$\implies$};
    \node (G) at ([xshift=-1pt]F.west) [anchor=east] {$(M,P)$ is a summand-maximal $\td$-rigid pair};
    \node (H) at ([yshift=-1pt]G.south) [] {$\Downarrow$};
    \node (I) at ([yshift=-1pt]H.south) [] {$(M,P)$ is a $\td$-rigid pair such that $\abs{M}+\abs{P}\geq \abs{\Lambda}$};
\end{tikzpicture}
\]

We are interested in studying when these implications can be reversed. It is clear that the bottom right implication can not be reversed. For the upper right implication we consider the following special case.

\begin{lemma}\label{lem:support tau_d tilting and strongly maximal is the same if no homs}
    Let $\cC\subseteq\modfin{\Lambda}$ be a $d$-cluster tilting subcategory. Assume that for every indecomposable module $X\in\cC$ we have that $\Hom{\Lambda}{X}{\td(X)}=0$. Then a pair $(M,P)$ with $M\in\cC$ and $P\in\proj\Lambda$ is a support $\td$-tilting pair if and only if it is a maximal $\td$-tilting pair.
\end{lemma}

\begin{proof}
Assume that $(M,P)$ is a support $\td$-tilting pair and we show that it is a maximal $\td$-tilting pair; the other implication always holds. Let $N\in\cC$ and we need to show that 
\[
N \in \add{M} \iff \begin{cases}
            \Hom{\Lambda}{M}{\td(N)}=0, \\
            \Hom{\Lambda}{N}{\td(M)}=0, \\
            \Hom{\Lambda}{P}{N}=0.
        \end{cases}
\]
Clearly we may assume that $N$ is indecomposable. Since $(M,P)$ is a $\td$-rigid pair, we have that $\Hom{\Lambda}{M}{\td(M)}=0$ and that $\Hom{\Lambda}{P}{M}=0$. Hence if $N\in\add{M}$, then we obtain that $\Hom{\Lambda}{M}{\td(N)}=0$, that $\Hom{\Lambda}{N}{\td(M)}=0$ and that $\Hom{\Lambda}{P}{N}=0$ as required. Assume now that these three equalities hold. By assumption we also have that $\Hom{\Lambda}{N}{\td(N)}=0$. Altogether we conclude that $(M\oplus N,P)$ is a $\td$-rigid pair, and so $N\in\add{M}$ since $(M,P)$ is a support $\td$-tilting pair. 
\end{proof}

Notice that the condition of Lemma \ref{lem:support tau_d tilting and strongly maximal is the same if no homs} is satisfied for all representation-directed algebras. Since a homogeneous linear Nakayama algebra $\Lambda(n,l)$ is representation-directed, in this article we provide no examples of support $\td$-tilting pairs which are not maximal $\td$-rigid pairs. 

On the other hand, Theorem \ref{thrm:taud tilting is well-configured} gives a way to construct maximal $\td$-rigid pairs which are not summand-maximal. Indeed, we may pick a $\td$-rigid pair which is not well-configured and cannot be enlarged to a well-configured pair. For example, by their definition, well-configured pairs are never supported in three consecutive diagonals when $d>2$. But if $d>2$, $p\geq 4$ and $l\geq 4$, then we can always find $\td$-rigid pairs which are supported in three consecutive diagonals. By enlarging such a $\td$-rigid pair as much as possible, we obtain a maximal $\td$-rigid pair which is not summand-maximal. We note that not all maximal $\td$-rigid pairs are of this form either (i.e. supported in three or more consecutive diagonals), for an example of a different kind see Example \ref{Example:Not well-configured pair}(a).

\begin{remark}
    We remark that for an arbitrary finite-dimensional algebra $\Lambda$ it is unknown whether a $\td$-rigid pair $(M,P)$ with $\abs{M}+\abs{P}=\abs{\Lambda}$ is also maximal, let alone summand-maximal. Moreover, as we have seen, a maximal $\td$-rigid pair $(M,P)$ does not necessarily satisfy $\abs{M}+\abs{P}=\abs{\Lambda}$, in contrast to the case $d=1$. It is also unknown whether for a summand-maximal $\td$-rigid pair $(M,P)$ the condition $\abs{M}+\abs{P}=\abs{\Lambda}$ is always satisfied. 

    Both of these remarks are closely tied to the question of when a presilting object in $\HomotopyC{\proj{\Lambda}}{b}$ is \emph{partial silting}\index[definitions]{partial silting object}, i.e. it can be completed to a silting object. Observe that in this case we obtain from Lemma \ref{lemma:silting summands} and Lemma \ref{lemma:taud-rigid iff presilting} that a $\td$-rigid pair $(M,P)$ is summand-maximal if and only if $\abs{M}+\abs{P}=\abs{\Lambda}$. A class of algebas where presilting is known to imply partial silting is silting-discrete algebras \cite[Thm. 2.15]{Aihara2017Classifying}. In particular, by \cite[Prop. 3.1]{aihara2023siltingdiscreteness} and \cite[Prop. 2.1]{Happel2010Piecewise}, one obtains that in the algebras $\Lambda(n,2)$ and $\Lambda(n,n-1)$ (which always admit $d$-cluster tilting subcategories) a $\td$-rigid pair $(M,P)$ is summand-maximal if and only if $\abs{M}+\abs{P}=\abs{\Lambda}$. However, most of the algebras on the form $\Lambda(n,l)$ are not silting-discrete, as seen in Table 1 of \cite{Happel2010Piecewise}, hence Theorem \ref{Theorem A} is needed to establish the property for $\Lambda(n,l)$ in general.
\end{remark}

\subsection{Maximal \texorpdfstring{$\td$}{td}-rigid pairs, \texorpdfstring{$d$}{d}-torsion classes and \texorpdfstring{$(d+1)$}{(d+1)}-term silting complexes} 
We start with recalling some useful notions. Let $\cC\subseteq \modfin{\Lambda}$ be a $d$-cluster tilting subcategory for some $d\geq 1$. Recall that if $d=1$ then $\cC=\modfin{\Lambda}$. Let $\cU\subseteq \cC$ be a $d$-torsion class. An object $X\in\cU$ is called \emph{$\Ext^{d}$-projective in $\cU$}\index[definitions]{$\Ext^{d}$-projective!object} if $\Ext^{d}_{\Lambda}(X,\cU)=0$. If moreover $\add{X}$ contains all $\Ext^{d}$-projective objects in $\cU$, then $X$ is called an \emph{$\Ext^{d}$-projective generator of $\cU$}\index[definitions]{$\Ext^{d}$-projective!generator}. It is shown in \cite{august2024taudtilting} that an $\Ext^{d}$-projective generator of $\cU$ always exists. It then follows by the definition of $\Ext^{d}$-projective generators that there exists a unique up to isomorphism basic $\Ext^{d}$-projective generator of $\cU$, which we denote by $M^{\cU}$. We further set $P^{\cU}$ to be a maximal basic projective $\Lambda$-module satisfying $\Hom{\Lambda}{P^{\cU}}{\cU}=0$.

For a $2$-silting complex $P_1\overset{f}{\to} P_0$ we write $f=\begin{psmallmatrix}
    f' & 0
\end{psmallmatrix}:P_1'\oplus P_1'' \to P_0$ where $f'$ is right minimal. For $d=1$ we have the following result.

\begin{theorem}[{\cite[Theorem 2.7, Theorem 3.2]{adachi_-tilting_2014}}]\label{thrm:bijections when d=1}
There are bijections 
\[
\left\{\begin{array}{c}
    \text{functorially finite} \\
    \text{torsion classes in $\modfin{\Lambda}$}
\end{array} \right\}\rightleftarrowss{\phi}{\phi^{-1}}\left\{\begin{array}{c} \text{basic support $\tau$-tilting} \\ 
\text{pairs in $\modfin{\Lambda}$} 
\end{array}\right\}\rightleftarrowss{\psi}{\psi^{-1}}\left\{\begin{array}{c}
     \text{basic $2$-silting} \\
     \text{complexes in $\HomotopyC{\proj{\Lambda}}{b}$}
\end{array}\right\}.
\] 
Explicitly, for a functorially finite torsion class $\cU\subseteq\modfin{\Lambda}$, a basic support $\tau$-tilting pair $(M,P)$ in $\modfin{\Lambda}$ and a basic $2$-silting complex $P=P_1\overset{f}{\to}P_0$ in $\HomotopyC{\proj{\Lambda}}{b}$ we have
\begin{itemize}
    \item $\phi(\cU)=(M^{\cU},P^{\cU})$ and $\phi^{-1}(M,P)=\fac{M}$, and
    \item $\psi(M,P)=P[1]\oplus \trunc{\mathbf{P}^\bullet}{\geq -1}(M)$ and $\psi^{-1}(P)=(\coker(f),P_1'')$.
\end{itemize}
\end{theorem}

When $d>1$, we have the following generalization of Theorem \ref{thrm:bijections when d=1} from \cite{august2024taudtilting}.

\begin{theorem}\cite{august2024taudtilting}\label{thrm:injections when d>1}
Let $\cC\subseteq\modfin{\Lambda}$ be a $d$-cluster tilting subcategory. There exist maps
\[
\left\{\begin{array}{c}
    \text{functorially finite} \\
    \text{$d$-torsion classes in $\cC$}
\end{array} \right\} \overset{\phi_d}{\longrightarrow} \left\{\begin{array}{c} \text{basic maximal $\td$-rigid} \\ 
\text{pairs $(M,P)$ in $\cC$ with} \\
\text{$\abs{M}+\abs{P}=\abs{\Lambda}$}
\end{array}\right\}\supseteq \im(\phi_d) \overset{\psi_d}{\longrightarrow} \left\{\begin{array}{c}
     \text{basic $(d+1)$-silting} \\
     \text{complexes in $\HomotopyC{\proj{\Lambda}}{b}$}
\end{array}\right\}.
\]
Explicitly, for a functorially finite $d$-torsion class $\cU\subseteq \modfin{\Lambda}$ and a basic maximal $\td$-rigid pair $(M,P)$ in $\cC$ with $\abs{M}+\abs{P}=\abs{\Lambda}$ we have
\begin{itemize}
    \item $\phi_d$ is injective, is in general not surjective when $d>1$, is given by $\phi_d(\cU)=(M^{\cU},P^{\cU})$, and it has a partial inverse given by $(M^{\cU},P^{\cU})\mapsto \fac{M^{\cU}}\cap\cC$, and
    \item $\psi_d$ is defined on $\im(\phi_d)$, and is given by $\psi_d(M,P)=P[d]\oplus \trunc{\mathbf{P}^\bullet}{\geq -d}(M)$.
\end{itemize}
\end{theorem}

Let $M$ be any module in a $d$-cluster tilting subcategory $\cC\subseteq\modfin{\Lambda}$. Let $\dtors{M}$ be the smallest $d$-torsion class in $\cC$ which contains $M$; this is well-defined since $d$-torsion classes in $\cC$ form a complete lattice with meet given by intersection by \cite[Theorem 4.3]{august2023characterisation}. It follows that we may define a map
\[
\left\{\begin{array}{c}
    \text{functorially finite} \\
    \text{$d$-torsion classes in $\cC$}
\end{array} \right\} \overset{\tilde{\phi}_d}{\longleftarrow} \left\{\begin{array}{c} \text{basic maximal $\td$-rigid} \\ 
\text{pairs $(M,P)$ in $\cC$ with} \\
\text{$\abs{M}+\abs{P}=\abs{\Lambda}$}
\end{array}\right\}
\]
by $\tilde{\phi}_d(M,P)=\dtors{M}$. Since $\fac{M}\cap \cC\subseteq \dtors{M}$, we conclude that if $(M,P)$ is in the image of $\phi_d$ then 
\[
\dtors{M} = \fac{M}\cap \cC,
\]
and so $\tilde{\phi}_d$ is surjective but not injective. 

As far as the map $\psi_d$ is concerned, we have the following result in our case.

\begin{theorem}\label{thrm:the silting map can be extended}
    Let $\Lambda=\Lambda(n,l)$ be a homogeneous Nakayama algebra admitting a $d$-cluster tilting subcategory $\cC\subseteq\modfin{\Lambda}$. Then the map $\psi_d$ in Theorem \ref{thrm:injections when d>1} can be extended to an injective map
    \[
    \left\{\begin{array}{c} \text{basic maximal $\td$-rigid} \\ 
\text{pairs $(M,P)$ in $\cC$ with} \\
\text{$\abs{M}+\abs{P}=\abs{\Lambda}$}
\end{array}\right\}\overset{\psi_d}{\longrightarrow} \left\{\begin{array}{c}
     \text{basic $(d+1)$-silting} \\
     \text{complexes in $\HomotopyC{\proj{\Lambda}}{b}$}
\end{array}\right\}.
    \]
\end{theorem}

\begin{proof}
    Follows by Theorem \ref{Theorem:strongly maximal and silting}.
\end{proof}

As we have already mentioned, the bijection $\psi$ in Theorem \ref{thrm:bijections when d=1} can be seen as a motivation for the definition of MM-$\td$-tilting modules in \cite{MARTINEZ202398}, see Definition \ref{def:MM-tau_d-tilting}. Based on our results, we have the following description of MM-$\td$-tilting modules lying in $d$-cluster tilting subcategories in the case $\Lambda=\Lambda(n,l)$.

\begin{corollary}
    Let $\Lambda=\Lambda(n,l)$ be a homogeneous Nakayama algebra admitting a $d$-cluster tilting subcategory $\cC\subseteq\modfin{\Lambda}$. Let $M\in\cC$ be a basic module. Then the following are equivalent.
    \begin{enumerate}
        \item[(a)] $M$ is MM-$\td$-tilting. 
        \item[(b)] $(M,0)$ is $\td$-rigid and $\abs{M}=\abs{\Lambda}$.
    \end{enumerate}
\end{corollary}

\begin{proof}
    If (a) holds, then (b) holds for any finite-dimensional algebra by \cite[Theorem 3.4]{MARTINEZ202398} and \cite[Theorem 5.7]{MARTINEZ202398}. If (b) holds then (a) follows by Theorem \ref{Theorem:strongly maximal and silting}.
\end{proof}

We illustrate Theorem \ref{thrm:injections when d>1} and Theorem \ref{thrm:the silting map can be extended} with an example.

\begin{example}
Let $\Lambda=\Lambda(37,4)$ and $d=4$ so that $\Lambda$ admits a $4$-cluster tilting subcategory $\cC$. Let $\cU$ be the $4$-torsion class defined as the additive closure of the encircled modules in the following picture:
    \[
\includegraphics{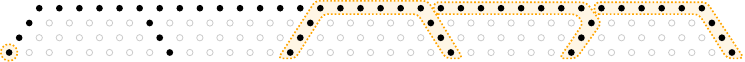}
\]
Then $\phi_4(\cU)=(M^{\cU},P^{\cU})$ where $M^{\cU}$ is the additive closure of the modules encircled in red while $P^{\cU}$ is the additive closure of the modules encircled in blue:
\[
\includegraphics{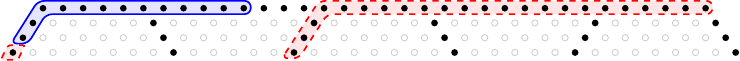}
\]
    We see that $(M^{\cU},P^{\cU})$ satisfies $\abs{M^{\cU}}+\abs{P^{\cU}}=37$. Also we have that $\tilde{\phi}_4(M^{\cU},P^{\cU})=\dtors{M^{\cU}}=\cU$, as expected. Now consider the summand-maximal $\tau_4$-rigid pair $(M',P')$ where $M'$ is the additive closure of the modules encircled in red while $P'$ is the additive closure of the modules encircled in blue:
\[
\includegraphics{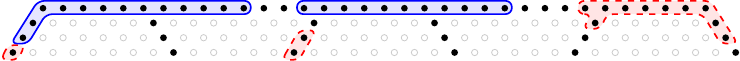}
\]  
    Then we have that $\abs{M'}+\abs{P'}=37$ again, and also that $\tilde{\phi}_4(M',P')=\dtors{M'}=\cU$. Hence $(M',P')\not\in \im(\phi_d)$ and we cannot apply Theorem \ref{thrm:injections when d>1}. However Theorem \ref{Theorem:strongly maximal and silting} gives that both 
    \[
    P^{\cU}[4]\oplus \trunc{\mathbf{P}^\bullet}{\geq -4}(M^{\cU}) \text{ and } 
    P'[4]\oplus \trunc{\mathbf{P}^\bullet}{\geq -4}(M')
    \]
    are $5$-silting complexes in $\HomotopyC{\proj\Lambda}{b}$.    
\end{example}

\subsection{Mutation of summand-maximal \texorpdfstring{$\td$}{td}-rigid pairs}
Finally we share some observations on mutations of summand-maximal $\td$-rigid pairs. These observations are not on the level of a full description. However, since mutation as a concept has been a driving factor for development within several branches of representation theory, e.g. the generalization from tilting to support $\tau$-tilting, we feel justified in sharing some of them here.

Let $\cC\subseteq\modfin{\Lambda}$ be a $d$-cluster tilting subcategory. We say that two summand-maximal $\td$-rigid pairs $(M,P)$ and $(M',P')$ are \emph{$\td$-mutations of each other}\index[definitions]{higher tilting theory!$\td$-mutation of summand-maximal $\td$-rigid pairs} if there exists a $\td$-rigid pair $(N,Q)$ which is a direct summand of both $(M,P)$ and $(M',P')$, and $\abs{N}+\abs{Q}+1=\abs{M}+\abs{P}=\abs{M'}+\abs{P'}$. We also define the \emph{mutation graph of summand-maximal $\td$-rigid pairs}\index[definitions]{higher tilting theory!mutation graph of summand-maximal $\td$-rigid pairs} (abbreviated \emph{sm $\td$-rigid graph}) of $\Lambda$ as the graph having the summand-maximal $\td$-rigid pairs as nodes, and edges between two nodes if the pairs are $\td$-mutations of each other. 

When $d=1$ the definition of $\td$-mutation coincides with the definition of mutation for support $\tau$-tilting pairs. In that case we have that when the algebra is $\tau$-tilting finite, the mutation graph of support $\tau$-tilting modules is connected \cite[Corollary 2.38]{adachi_-tilting_2014}. We may show the same for $\Lambda(n,l)$ and $d>1$ as well.

\begin{proposition}\label{prop:mutation graph is connected}
Assume that $\Lambda(n,l)$ admits a $d$-cluster tilting module $C$ for some $d\geq 1$. Then the mutation graph of summand-maximal $\td$-rigid pairs of $\Lambda(n,l)$ is connected. 
\end{proposition}

\begin{proof}
We only give a sketch of the proof. Let $(M,P)$ be a summand-maximal $\td$-rigid pair. Equivalently, $(M,P)$ is well-configured. One can show that if $(M_i)$ is a full admissible configuration of type (III), then we can mutate the modules in the diagonal $\diag{i}$ to the interval $\Xi_i$. In other words, $(M,P)$ is $\td$-mutation equivalent to a well configured pair $(M',P')$ with $M_j=M_j'$ for $j\neq i$ and $M_i'=0$. Similarly, if $(M_i,\ldots,M_{i+k})$ is a full admissible configuration of type different than (III), then we can show that $(M,P)$ is $\td$-mutation equivalent to a well configured pair $(M',P')$ which is the same as $(M,P)$ except for the fact that we have replaced $(M_i,\ldots,M_{i+k})$ with a full admissible configuration of type (III) and adjusted $\Xi(i,i+k)$ accordingly. 

By repeatedly applying these two steps we show that $(M,P)$ is $\td$-mutation equivalent to a summand-maximal $\td$-rigid pair $(P(\red'),P(\blue'))$ for some $\red'$ and $\blue'$. Necessarily we have that $\red'=[1,x]$ and $\blue'=[x+1,n]$ for some $x\in [0,n+1]$. Such a pair can be easily seen to be $\td$-mutation equivalent to $(0,\Lambda)$. Hence every summand-maximal $\td$-rigid pair is $\td$-mutation equivalent to $(0,\Lambda)$ which proves connectedness of the sm $\td$-rigid graph.
\end{proof}

On the other hand, the case $d>1$ has also some important differences as well when compared to the case $d=1$. Recall that if $d=1$ then the mutation graph of support $\tau$-tilting modules is not necessarily finite. Currently every known example of a $d$-cluster tilting subcategory $\cC\subseteq\modfin{\Lambda}$ has an additive generator $C$ and it is an open question if that is always the case. In particular, there is no known example of an infinite sm $\td$-rigid graph for $d>1$.

Another crucial difference concerns the regularity of the sm $\td$-rigid graph. Bongartz' (co)completion guarantees that a support $\tau$-tilting pair can be mutated in $\abs{\Lambda}=n$ different ways and so the mutation graphs for $d=1$ are $n$-regular. However, for summand-maximal $\td$-tilting pairs this is not always the case as we can see in the following example.

\begin{example}
    Consider the algebra $\Lambda(9,3)$ which admits a $2$-cluster tilting subcategory $\cC$ with $4$ diagonals. From the formula at the end of Remark \ref{rem: Counting strongly maximal when d-rep finite} one can calculate that there are $160$ summand-maximal $\tau_2^{\phantom{-}}$-rigid pairs in $\cC$. These pairs give rise to the sm $\tau_2^{\phantom{-}}$-rigid graph drawn in Figure \ref{fig:mutationGraphLambda94}\footnote{This figure can also be found at \url{https://endresr.github.io/Higher_Tau_Nakayama/resources/CompleteGraph.pdf} in higher resolution}. Observe that in the neighbourhood of $(\Lambda,0)$ and $(0,\Lambda)$ we get planar subgraphs with nodes of degree either $3$ or $4$, which both connects to two subgraphs (A) and (B) which are not planar and have nodes of degrees $3$, $4$, $5$ and $6$. 

    Both (A) and (B) are characterized by having nodes where $(R,B)$ is not uniquely determined by the diagonal component. We also note that the nodes of maximal degree are all in (A) and have diagonal partition $\{3\}$ of admissible type (III). Equivalently, the nodes of maximal degree are summand-maximal $\tau_2^{\phantom{-}}$-rigid pairs such that $\tau_2^{-2}\Lambda$ lies in $\add{M}$. This may suggest that for a general algebra $\Lambda$ admitting a $d$-cluster tilting subcategory a strategy for finding summand-maximal $\td$-rigid pairs with the maximal amount of possible mutations might be to restrict to those having $\tau_d^{-r}\Lambda$ as rigid summands for some $r$.
    
\begin{figure}[p]
    \centering
    \includegraphics[width=1.6\textwidth,angle=90]{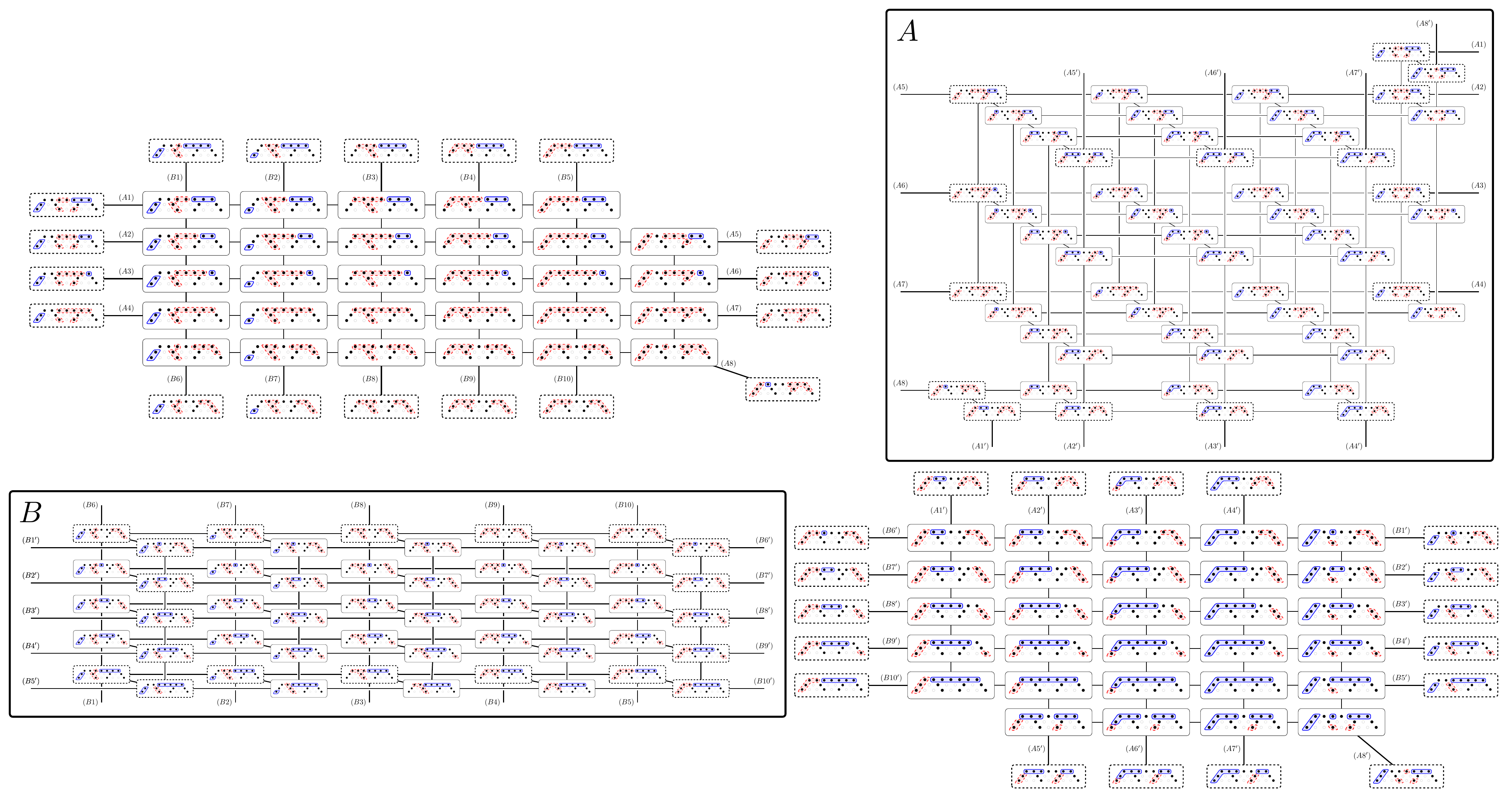}
    \caption{The sm $\tau_2^{\phantom{-}}$-rigid graph of $\Lambda(9,3)$. }
    \label{fig:mutationGraphLambda94}
\end{figure}
\end{example}

Despite the complexity of the previous example we see that when $\Lambda=\Lambda(n,l)$ is $d$-representation finite, that is when $d=\gldim(\Lambda)$, the situation is greatly simplified.

\begin{proposition}\label{prop:2-regular for d-rep-fin}
    Assume that $\Lambda(n,l)$ is $d$-representation-finite for some $d>1$. Then the sm $\td$-rigid graph of $\Lambda(n,l)$ is an extended Dynkin diagram of type $\tilde{A}$ with $2n+l-1$ vertices. In particular, it is a $2$-regular graph.
\end{proposition}

\begin{proof}
By \cite[Theorem 3]{VASO20192101} we have that $\Lambda(n,l)$ is $d$-representation-finite if and only if $p=2$. By Remark \ref{rem: Counting strongly maximal when d-rep finite} we have that there exist $2n+l-1$ summand-maximal $\td$-rigid pairs. The claim about the shape of the graph follows easily using this and the description of summand-maximal $\td$-rigid pairs as well-configured $\cC$-pairs in Theorem \ref{thrm:taud tilting is well-configured}; Example \ref{ex:a 2-regular d-mutation graph} below is indicative of the general picture.
\end{proof}

\begin{example}\label{ex:a 2-regular d-mutation graph}
    The sm $\tau_2^{\phantom{-}}$-rigid graph of the $2$-representation-finite algebra $\Lambda(4,3)$ is
\[
\includegraphics{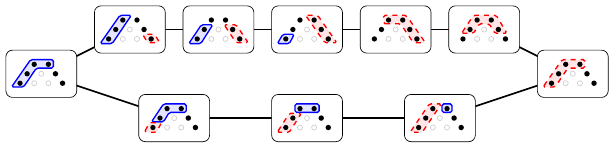}
\]
\end{example}

Note that $\Lambda(4,3)$ in Example \ref{ex:a 2-regular d-mutation graph} has $l=3$, $d=2$ and $p=2$ and $\Lambda(9,3)$ in Figure \ref{fig:mutationGraphLambda94} has $l=3$, $d=2$ and $p=4$. In other words $\Lambda(9,3)$ is the simplest example of $\Lambda(n,l)$ where $l>2$ and where $\Lambda(n,l)$ is not $2$-representation-finite. There is a perhaps surprising jump in complexity between the two graphs. One might therefore ask if $d$-representation finite algebras can always be expected to posses such nice behaviour. With regard to regularity the answer is no, a counterexample is the Auslander algebra of type $\mathbb{A}_3$ which is $2$-representation finite, but which has nodes of degree $3$ and $4$ in it's sm $\tau_2$-rigid graph.

We finish this section with a remark about the lattice of $d$-torsion cases when $\Lambda$ is $d$-representation-finite.

\begin{remark}\label{rem:lattice of d-torsion classes in d-rep-fin}
Let $\Lambda(n,l)$ be $d$-representation-finite. By Remark \ref{rem: Counting strongly maximal when d-rep finite} and Remark \ref{rem:counting d-torsion classes for d-rep-fin} we conclude that there are 
\[
2n+l-1 - (n+l+1) = n-2 
\]
more summand-maximal $\td$-rigid pairs than $d$-torsion classes. Continuing from Example \ref{ex:a 2-regular d-mutation graph}, the lattice of $2$-torsion classes of $\Lambda(4,3)$ is 
    \[
\includegraphics{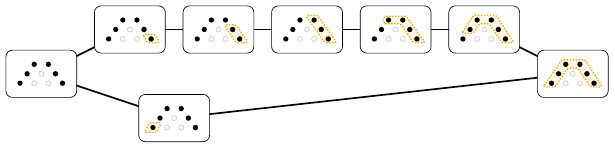}
\]
As expected, it has $2$ vertices less than the sm $\tau_2^{\phantom{-}}$-rigid graph. The above picture is indicative of the general situation, where the lattice of $d$-torsion classes can be obtained by removing all summand-maximal $\td$-rigid pairs of the form $(P(\red),P(\blue))$ where $\red=[1,x]$ and $\blue=[x+1,n]$ for some $x\in [2,n-1]$, and replacing each of the other summand-maximal $\td$-rigid pair $(M,P)$ with $\dtors{M}$.
\end{remark}

\appendix
\section{An illustration for the notation}\label{sec:appendix}
One aspect of this paper is the reliance on a wide variety of notations, which can be challenging but is necessary for clarity and precision. This is particularly evident in Sections \ref{Section:Preliminaries} and \ref{Section:strongly maximal td-rigid}. In an effort to reduce the cognitive load imposed by this fact, we include an illustration which encodes the most widely used notation in the article. Our hope is that it can be used as a place to quickly (re)familiarise oneself with the meaning of the different symbols.
\[
\includegraphics[width=\linewidth]{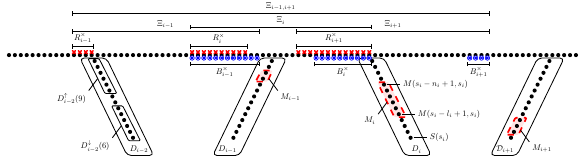}
\]
The illustration above should be thought of as 4 consecutive diagonals in the $d$-cluster tilting subcategory of an algebra $\Lambda(n,l)$ for $n$ and $l$ large enough, and for suitable $d$. In the diagonals $\diag{i-1}$, $\diag{i}$ and $\diag{i+1}$ we have denoted some modules, to be thought of as direct summands of a $\td$-rigid pair $(M,P)$. These in turn define all the intervals which we have drawn at the top. For a quick way to find the formal definitions of all of this notation we refer to the index at the end of this paper.

As a specific example for verifying some values, we note that if $l=15$, $p=6$, $d=4$ (thus $n=121$) and $i=4$, then we obtain the following values:
\[
\begin{array}{c|cccc}
     & i-2 & i-1 & i & i+1  \\ \hline
    m_j & 0 & 2 & 6 & 3 \\
    n_j & 0 & 12 & 10 & 4 \\
    l_j & 15 & 11 & 5 & 2 \\
    s_j & 31 & 48 & 78 & 95 \\
    \Xi_j & \varnothing & [28,58] & [47,78] & [65,98]
\end{array}
\]
Notice that increasing $p$ (and thus also $n$) has no effect on the numbers in the above table.

\bibliographystyle{plain} 
\bibliography{bibliography} 

\begin{thebibliography}{10}

\bibitem{ADACHI2016227}
Takahide Adachi.
\newblock The classification of {$\tau$}-tilting modules over {N}akayama algebras.
\newblock {\em J. Algebra}, 452:227--262, 2016.

\bibitem{adachi_-tilting_2014}
Takahide Adachi, Osamu Iyama, and Idun Reiten.
\newblock {$\tau$}-tilting theory.
\newblock {\em Compos. Math.}, 150(3):415--452, March 2014.

\bibitem{aihara2023siltingdiscreteness}
Takuma {Aihara} and Takahiro {Honma}.
\newblock When is the silting-discreteness inherited?
\newblock {\em Nagoya Math. J.}, page 1–23, 2024.

\bibitem{SiltingMutation}
Takuma Aihara and Osamu Iyama.
\newblock Silting mutation in triangulated categories.
\newblock {\em J. Lond. Math. Soc. (2)}, 85(3):633--668, 2012.

\bibitem{Aihara2017Classifying}
Takuma Aihara and Yuya Mizuno.
\newblock Classifying tilting complexes over preprojective algebras of {D}ynkin type.
\newblock {\em Algebra Number Theory}, 11(6):1287--1315, 2017.

\bibitem{hugel2007handbook}
Lidia Angeleri~H\"{u}gel, Dieter Happel, and Henning Krause, editors.
\newblock {\em Handbook of tilting theory}, volume 332 of {\em London Mathematical Society Lecture Note Series}.
\newblock Cambridge University Press, Cambridge, 2007.

\bibitem{asadollahi2022higher}
Javad Asadollahi, Peter J{\o}rgensen, Sibylle Schroll, and Hipolito Treffinger.
\newblock On higher torsion classes.
\newblock {\em Nagoya Math. J.}, 248:823--848, 2022.

\bibitem{ASS}
Ibrahim Assem, Daniel Simson, and Andrzej Skowro\'{n}ski.
\newblock {\em Elements of the representation theory of associative algebras. {V}ol. 1}, volume~65 of {\em London Mathematical Society Student Texts}.
\newblock Cambridge University Press, Cambridge, 2006.
\newblock Techniques of representation theory.

\bibitem{august2024taudtilting}
Jenny August, Johanne Haugland, Karin~M Jacobsen, Sondre Kvamme, Yann Palu, and Hipolito Treffinger.
\newblock Higher torsion classes, $\tau_d$-tilting theory, and silting complexes.
\newblock {\em In preparation}, 2024+.

\bibitem{august2023characterisation}
Jenny August, Johanne Haugland, Karin~M. Jacobsen, Sondre Kvamme, Yann Palu, and Hipolito Treffinger.
\newblock A characterisation of higher torsion classes.
\newblock {\em Forum Math. Sigma}, To appear.

\bibitem{auslander1974representation}
Maurice Auslander.
\newblock Representation theory of {A}rtin algebras. {I}, {II}.
\newblock {\em Comm. Algebra}, 1:177--268; ibid. 1 (1974), 269--310, 1974.

\bibitem{auslander1979coxeter}
Maurice Auslander, Mar\'{\i}a~In\'{e}s Platzeck, and Idun Reiten.
\newblock Coxeter functors without diagrams.
\newblock {\em Trans. Amer. Math. Soc.}, 250:1--46, 1979.

\bibitem{auslander1975representation}
Maurice Auslander and Idun Reiten.
\newblock Representation theory of {A}rtin algebras. {III}. {A}lmost split sequences.
\newblock {\em Comm. Algebra}, 3:239--294, 1975.

\bibitem{auslander1977representation}
Maurice Auslander and Idun Reiten.
\newblock Representation theory of {A}rtin algebras. {IV}. {I}nvariants given by almost split sequences.
\newblock {\em Comm. Algebra}, 5(5):443--518, 1977.

\bibitem{ARS}
Maurice Auslander, Idun Reiten, and Sverre~O. Smal\o.
\newblock {\em Representation theory of {A}rtin algebras}, volume~36 of {\em Cambridge Studies in Advanced Mathematics}.
\newblock Cambridge University Press, Cambridge, 1995.

\bibitem{bernstein1973coxeter}
I.~N. Bern\v{s}te\u{\i}n, I.~M. Gel'fand, and V.~A. Ponomarev.
\newblock Coxeter functors, and {G}abriel's theorem.
\newblock {\em Uspehi Mat. Nauk}, 28(2(170)):19--33, 1973.

\bibitem{BongCompl}
Klaus Bongartz.
\newblock Tilted algebras.
\newblock In {\em Representations of algebras ({P}uebla, 1980)}, volume 903 of {\em Lecture Notes in Math.}, pages 26--38. Springer, Berlin-New York, 1981.

\bibitem{brenner1980generalizations}
Sheila Brenner and M.~C.~R. Butler.
\newblock Generalizations of the {B}ernstein-{G}el'fand-{P}onomarev reflection functors.
\newblock In {\em Representation theory, {II} ({P}roc. {S}econd {I}nternat. {C}onf., {C}arleton {U}niv., {O}ttawa, {O}nt., 1979)}, volume 832 of {\em Lecture Notes in Math.}, pages 103--169. Springer, Berlin, 1980.

\bibitem{darpo2020representationfinite}
Erik Darp\"{o} and Osamu Iyama.
\newblock {$d$}-representation-finite self-injective algebras.
\newblock {\em Adv. Math.}, 362:106932, 50, 2020.

\bibitem{gelfand1972problems}
I.~M. Gel'fand and V.~A. Ponomarev.
\newblock Problems of linear algebra and classification of quadruples of subspaces in a finite-dimensional vector space.
\newblock In {\em Hilbert space operators and operator algebras ({P}roc. {I}nternat. {C}onf., {T}ihany, 1970)}, volume Vol. 5 of {\em Colloq. Math. Soc. J\'{a}nos Bolyai}, pages 163--237. North-Holland, Amsterdam-London, 1972.

\bibitem{gupta2024dtermsiltingobjectstorsion}
Esha Gupta.
\newblock $d$-term silting objects, torsion classes, and cotorsion classes, 2024.

\bibitem{happel1982tilted}
Dieter Happel and Claus~Michael Ringel.
\newblock Tilted algebras.
\newblock {\em Trans. Amer. Math. Soc.}, 274(2):399--443, 1982.

\bibitem{Happel2010Piecewise}
Dieter Happel and Uwe Seidel.
\newblock Piecewise hereditary {N}akayama algebras.
\newblock {\em Algebr. Represent. Theory}, 13(6):693--704, 2010.

\bibitem{happel1989almost}
Dieter Happel and Luise Unger.
\newblock Almost complete tilting modules.
\newblock {\em Proc. Amer. Math. Soc.}, 107(3):603--610, 1989.

\bibitem{Herschend-Jorgensen}
Martin Herschend and Peter J{\o}rgensen.
\newblock Classification of higher wide subcategories for higher {A}uslander algebras of type {$A$}.
\newblock {\em J. Pure Appl. Algebra}, 225(5):Paper No. 106583, 22, 2021.

\bibitem{ingalls2009noncrossing}
Colin Ingalls and Hugh Thomas.
\newblock Noncrossing partitions and representations of quivers.
\newblock {\em Compos. Math.}, 145(6):1533--1562, 2009.

\bibitem{Iyama200722}
Osamu Iyama.
\newblock Higher-dimensional {A}uslander--{R}eiten theory on maximal orthogonal subcategories.
\newblock {\em Adv. Math.}, 210(1):22 – 50, 2007.

\bibitem{iyama2008auslander}
Osamu Iyama.
\newblock Auslander--{R}eiten theory revisited.
\newblock In {\em Trends in representation theory of algebras and related topics}, EMS Ser. Congr. Rep., pages 349--397. Eur. Math. Soc., Z\"{u}rich, 2008.

\bibitem{iyama2011cluster}
Osamu Iyama.
\newblock Cluster tilting for higher {A}uslander algebras.
\newblock {\em Adv. Math.}, 226(1):1--61, 2011.

\bibitem{JACOBSEN2020119}
Karin~M. Jacobsen and Peter Jørgensen.
\newblock Maximal $\tau_d$-rigid pairs.
\newblock {\em J. Algebra}, 546:119--134, 2020.

\bibitem{jasso2016n}
Gustavo Jasso.
\newblock {$n$}-abelian and {$n$}-exact categories.
\newblock {\em Math. Z.}, 283(3-4):703--759, 2016.

\bibitem{Jorgensen2014}
Peter J{\o}rgensen.
\newblock Torsion classes and t-structures in higher homological algebra.
\newblock {\em Int. Math. Res. Not. IMRN}, (13):3880--3905, 2016.

\bibitem{MARTINEZ202398}
Luis Mart\'{\i}nez and Octavio Mendoza.
\newblock n-term silting complexes in {$\mathrm{K}^b(\mathrm{proj}(\Lambda))$}.
\newblock {\em J. Algebra}, 622:98--133, 2023.

\bibitem{mcmahon2021support}
Jordan McMahon.
\newblock Support {$\tau_2 $}-tilting and 2-torsion pairs.
\newblock {\em arXiv preprint arXiv:2102.08254}, February 2021.

\bibitem{riedtmann1991simplicial}
Christine Riedtmann and Aidan Schofield.
\newblock On a simplicial complex associated with tilting modules.
\newblock {\em Comment. Math. Helv.}, 66(1):70--78, 1991.

\bibitem{Ringel_2007}
Claus~Michael Ringel.
\newblock Appendix: {S}ome remarks concerning tilting modules and tilted algebras. {O}rigin. {R}elevance. {F}uture.
\newblock In {\em Handbook of tilting theory}, volume 332 of {\em London Math. Soc. Lecture Note Ser.}, pages 413--472. Cambridge Univ. Press, Cambridge, 2007.

\bibitem{UNGER1990205}
Luise Unger.
\newblock Schur modules over wild, finite-dimensional path algebras with three simple modules.
\newblock {\em J. Pure Appl. Algebra}, 64(2):205--222, 1990.

\bibitem{VASO20192101}
Laertis Vaso.
\newblock n-cluster tilting subcategories of representation-directed algebras.
\newblock {\em J. Pure Appl. Algebra}, 223(5):2101--2122, 2019.

\bibitem{ZHOU2023193}
Panyue Zhou and Bin Zhu.
\newblock Support $\tau_n$-tilting pairs.
\newblock {\em J. Algebra}, 616:193--211, 2023.

\bibitem{zhou2025tiltingextendedmodule}
Yu~Zhou.
\newblock Tilting theory for extended module categories, 2025.

\end{thebibliography}

\printindex[definitions]
\printindex[symbols]

\end{document}